\documentclass[10pt,final]{article}
\usepackage{amsmath, amsfonts, amsthm, amssymb, verbatim,a4wide, psfrag}
\usepackage[dvips]{graphicx}
\usepackage{color}
\usepackage[dvipsnames]{xcolor}
\usepackage[mathcal]{euscript}  
\usepackage{perpage}
\MakePerPage{footnote} 

\newcommand \bse {\begin{subequations}}
\newcommand \ese {\end{subequations}}

\newcommand \Phif {\Phi^\flat} 
\newcommand \Phifo {\Phi^\flat_{i0}} 
\newcommand \Phis {\Phi^\sharp} 
\newcommand \map {\widetilde{\Phif_{i0}}}
\newcommand \SMdomain {\Gamma}
\newcommand \Bcal {\mathcal B}
\newcommand \Bzero {\Bcal({\delta_0})} 
\newcommand \Bone {\Bcal({\delta_1})} 

\newcommand \RR   {\mathbb{R}}
\newcommand \RN   {\RR^N}  
\newcommand \be  {\begin{equation}}
\newcommand \ee  {\end{equation}} 
\newcommand \del {{\partial}}
\newcommand \eps {\epsilon}
\newcommand \lam {\lambda} 
\newcommand \lamb {\mybar \lambda} 
\newcommand \mui  {\mu_i}
\newcommand \muis {\mu_i^\sharp}
\newcommand \muiN {\mu_i^n}

\newcommand \HH   {{\mathcal H}}
\newcommand \AAA   {{\mathcal A}}
\newcommand \BB   {{\mathcal B}}
\newcommand \WW   {{\mathcal W}}
\newcommand \SSS   {{\mathcal S}}
\newcommand \MM   {{\mathcal M}}
\newcommand \overu {{\mybar u}} 

\newcommand \bei {\begin{itemize}}
\newcommand \eei {\end{itemize}}

\newtheorem{definition}{Definition}[section]
\newtheorem{lemma}[definition]{Lemma}
\newtheorem{theorem}[definition]{Theorem}
\newtheorem{proposition}[definition]{Proposition}

\numberwithin{equation}{section}
\newcommand{\mybar}[1]{\overline{#1}}
\newcommand{\myW}{W}
\newcommand{\Sgn}{\operatorname{sign}}

\def\muj   {{\mu_j}}  
 
\def\mujf  {\mu_j^\flat} 
\def\mujs  {\mu_j^\sharp}

\def\Phiis {\Phi^\sharp_{i}}
\def\Phiif {\Phif_{i}}
\def\mn    {{-\natural}}
\def\nat   {{\natural}} 

\def\OO    {{\mathcal O}}
  
\def\Se    {{\mathcal X}}

\newcommand	\CDown	{{C^\downarrow}}
\newcommand	\Cpm 	{{C^\downarrow_\pm}}
\newcommand	\Npm		{{N^\downarrow_\pm}}
\newcommand	\Rm		{{R^\downarrow_-}}
\newcommand	\Cm		{{C^\uparrow_-}}
\newcommand	\CUp		{{C^\uparrow}}

\newcommand	\kappaL		{k_L}
\newcommand	\kappaM		{k_M}
\newcommand	\kappaR		{k_R}

\newcommand	\kappaLless		{k_L^<}
\newcommand	\kappaMless		{k_M^<}
\newcommand	\kappaRless		{k_R^<}
\newcommand	\kappaLgrt		{k_L^>}
\newcommand	\kappaMgrt		{k_M^>}
\newcommand	\kappaRgrt		{k_R^>}
\newcommand  \Cff                    {C^{\flat\flat}}

\begin{document} 

\title{The mathematical theory of splitting-merging patterns
\\
 in phase transition dynamics}

\author{Eva Kardhashi$^1$, Marc Laforest$^1$, and Philippe G. LeFloch$^3$}
\footnotetext[1]{D\'epartement de math\'ematiques et statistique, Universit\'e de Montr\'eal,
	Montr\'eal, Qu\'ebec, Canada, H3C 3A7. 
	\\
	Email: {\tt evakard@dms.umontreal.ca, marc.laforest@polymtl.ca}.}
\footnotetext[3]{ Laboratoire Jacques-Louis Lions \& Centre National de la Recherche Scientifique, 
Sorbonne Universit\'e, 4 Place Jussieu, 
75252 Paris, France. Email: {\tt contact@philippelefloch.org}. 
\newline 
2000\textit{\ AMS Classification} Primary. 35L65, 74XX. Secondary. 76N10, 76L05. 
}
\date{May 2021}

\maketitle

\begin{abstract} For nonlinear hyperbolic systems of conservation laws in one space variable, we establish the existence of nonclassical entropy solutions exhibiting nonlinear interactions between shock waves with strong strength. The proposed theory is relevant in the theory of phase transition dynamics, and the solutions under consideration enjoy a splitting--merging pattern, consisting of (compressive) classical and (undercompressive) nonclassical waves, interacting together as well as with classical waves of smaller strength. Our analysis is based on three novel ideas.  First, a generalization of Hayes--LeFloch's nonclassical Riemann solver is introduced for systems and is based on prescribing, on one hand, a kinetic relation for the propagation of nonclassical undercompressive shocks and, on the other hand, a nucleation criterion that selects between classical and nonclassical behavior. 
Second, we extend LeFloch-Shearer's theorem to systems and we prove that the presence of a nucleation condition implies that only a finite number of splitting and merging cycles can occur. Third, our arguments of nonlinear stability build upon recent work by the last two authors who identified a natural total variation functional for scalar conservation laws and, specifically, for systems of conservation laws  we introduce here novel  functionals which measure the total variation and wave interaction of nonclassical and classical waves. 
\end{abstract}

\tableofcontents  

%********************************************************************************************************************************************
\section{Introduction}
\label{Sec1}

\subsection{Objective of this paper}

We consider the initial value problem associated with nonlinear hyperbolic systems of conservation laws in one space variable, and 
we establish the existence of {\sl nonclassical entropy solutions} (in the sense introduced in \cite{LeFloch-book}), 
which, for instance, arise in complex fluid flows in the limit of vanishing viscosity and capillarity. The solutions under consideration in the present paper 
exhibit a ``splitting--merging pattern'' and consist of finitely many (classical or nonclassical) big waves and classical small waves. More precisely, our analysis relies on a new nonclassical Riemann solver which is motivated by the theory of phase transition dynamics and is based on
\bei
\item a {\sl kinetic relation} for the propagation of nonclassical undercompressive shocks and, 
\eei 
\bei
\item optionally, a {\sl nucleation criterion} that selects between classical and nonclassical behavior. 
\eei 
All solutions satisfy a single entropy inequality associated with a given entropy pair, and we note that 
the kinetic relation is equivalent to imposing the rate of entropy dissipation across nonclassical shocks. 
The existence of nonclassical entropy solutions for {\sl scalar} equations was recently established for the initial value problem 
by the last two authors \cite{LL} who, building on the earlier work \cite{ABLP,BLP1,BLP2,LeFloch-book,LeFlochShearer}, 
succeeded in identifying a ``natural'' total variation functional adapted to deal with nonclassical shock waves and their interactions.  
By natural, we mean that the construction of the total variation relies explicitly on the main property of
the kinetic function, namely on the contraction property of its iterates.
The present paper provides a new and significant development to this theory and we are able to 
\bei

\item  construct nonclassical entropy solutions to nonlinear strictly hyperbolic {\sl systems,} and  

\item exhibit the {\sl splitting--merging} phenomena for systems, which illustrates the complexity of fluid flows involving phase transitions. 

\eei
To achieve these results, we introduce here new functionals that suitably measure the total variation and wave interaction of 
nonclassical and classical waves. When the nucleation criterion is introduced,
we show that these functionals are strictly decreasing by an amount proportional to the strength of the nucleation criteria.

Recall that the kinetic relation is an additional jump relation required only for nonclassical undercompressive shocks. 
It generalizes to general systems of conservation laws a notion originally introduced for a (hyperbolic-elliptic) model of phase transition dynamics 
first studied by Abeyaratne and Knowles \cite{AK1,AK2}, 
and Truskinovsky \cite{Trusk1,Trusk2}, in relation to earlier work by Slemrod \cite{Slemrod1,Slemrod2}. 
The notion of a kinetic relation was introduced for {\sl general} hyperbolic systems by LeFloch and collaborators, and was then investigated mathematically \cite{HL0,HL1,HL2,LeFloch1,LeFloch-freiburg,LeFloch-book,LeFloch-oslo}.  
In addition, for scalar conservation laws LeFloch and Shearer \cite{LeFlochShearer} observed that 
the classical solution may be physically relevant in certain applications {\sl even in a range} where a nonclassical solution is available. 
This observation led them to propose a nonclassical Riemann solver with {\sl kinetics and nucleation,}
which allows for nonclassical behavior only in Riemann problems having sufficiently `large'' initial jump.  
This proposal was motivated by studies on the dynamics of thin films, in which capillarity effects 
play a driving role on propagating interfaces; cf.~Bertozzi, M\"unch, and Shearer \cite{BMS} and Bertozzi and Shearer \cite{BS}. 

%-------------------------------------------------------------------------------------------------------------

\subsection{Background on nonclassical shocks and phase boundaries}
 
We thus consider nonlinear hyperbolic systems of conservation laws in one space dimension, of the general form  
\be 
\label{11}
\del_t u + \del_x f(u) = 0, \quad u=u(t,x) \in \Bzero 
\ee
and, as is customary, we focus attention on weak solutions understood in the sense of distributions which are assumed to take values near a given constant state in $\RN$, normalized to be the origin. Throughout, $\Bzero \subset \RN$ denotes the ball centered at the origin and with small radius $\delta_0>0$, while the flux $f: \Bzero \subset \RN$ is a smooth ($C^3$, say) vector-valued map, 
In addition, weak solutions of interest should satisfy the {\sl entropy inequality} 
\be
\label{12}
\del_t U(u) + \del_x F(u) \leq 0,
\ee
in which the entropy $U: \Bzero \to \RR$ is assumed to be strictly convex, and the associated entropy flux $F: \Bzero \subset \RN$ satisfies, by definition, $\nabla F(u)^T = \nabla U(u)^T \, Df(u)$. 

Let us first stress that our motivation for imposing the {\sl single} entropy inequality \eqref{12} comes from the fact 
that most (if not all) regularizations of \eqref{11} arising in applications, of the form  
$$
\del_t u^\eps + \del_x f(u^\eps) = R^\eps = R^\eps \big(u^\eps, \eps \, u_x^\eps, \eps^2 \, u_{xx}^\eps, \ldots\big) 
$$
are ``conservative'' and "entropy dissipative'' in the limit $\eps \to 0$ and, formally at least, lead to weak solutions satisfying \eqref{11}--\eqref{12}. (See \cite[Chap.~1]{LeFloch-book} for further details.)

Furthermore, we assume that \eqref{11} is a {\sl strictly hyperbolic} system, for which the Jacobian matrix $A(u) := Df(u)$ 
has real distinct eigenvalues $\lam_1(u) < \ldots < \lam_N(u)$ and corresponding basis of left-- and right--eigenvectors $l_j(u), r_j(u)$. 
Following Hayes and LeFloch \cite{HL1,HL2}, a nonlinearity condition on the flux $f$ is also imposed which roughly speaking 
requires that the flux $f$ has ``at most one inflection point''  along every eigenvector direction  
--- as this is a typical feature of many systems of interest in continuum physics, especially phase transition models. More precisely, 
we assume that each characteristic field of \eqref{11} is either (with $u \in \Bzero$) 
\bei 

\item {\sl genuinely nonlinear} (GNL), $j \in J_1$, that is, $\nabla \lam_j(u) \cdot r_j(u) >0$ after normalization, or 

\item {\sl linearly degenerate} (LD), $j \in J_2$, that is, $\nabla \lam_j(u) \cdot r_j(u) \equiv 0$, or 

\item {\sl concave-convex} (C-C), $i \in J_3$. 

\eei 
\noindent In the latter case by definition (see \cite{HL1,HL2}), the scalar-valued function $m_i(u) := (\nabla \lam_i\cdot r_i) (u)$ does not keep a constant sign, but $\MM_i = \big\{ u \in \Bzero \, |\,  m_i(u) = 0  \big\}$
is a smooth manifold with dimension $N-1$, such that the vector field $r_i$ is transverse to $\MM_i$ with 
\be
\text{concave-convex (C--C):} \quad  \nabla m_i(u) \cdot r_i (u) > 0.
\ee
Convex-concave fields can be defined similarly and, for simplicity in the presentation,
 will not be considered in the present paper. 

We are now in a position to consider the initial value problem associated with \eqref{11}, whose simplest version 
is provided by the so--called {\sl Riemann problem,} for which the 
initial data consist of two constant states $u_l,u_r \in \Bone$ initially separated by a single jump discontinuity: 
\be 
u(0,x) = u_0(x) = \begin{cases}
u_l, & x <0, 
\\
u_r, & x > 0. 
\end{cases}
\ee 
Hayes and LeFloch \cite{HL1,HL2} considered  weak solutions to the Riemann problem (in the class of piecewise Lipschitz continuous functions, say)  
associated with general systems with GNL, LD, or CC fields, and observed that solutions satisfying the single entropy inequality \eqref{12}, 
are {\sl not uniquely} determined by their initial data.  In order to select a unique nonclassical solution to the Riemann problem, it is necessary   to impose an additional selection criterion, called a {\sl kinetic relation,}
which selects, within a two--dimensional wave set, a unique wave curve associated with any state (in $\Bone$ with $\delta_1 \leq \delta_0$, say).
This selection criteria for the CC fields, together with the usual solutions for the GNC and LD fields, yields the 
 {\sl nonclassical Riemann solver}, as described in \cite{HL1,HL2}.  

%-------------------------------------------------------------------------------------------------------------

\subsection{Main results in this paper}

In the present paper, we go beyond Hayes and LeFloch's construction and we introduce a new 
nonclassical Riemann solver for systems, referred to as 
the {\sl nonclassical Riemann solver with kinetics and nucleation.} 
We investigate its properties, especially from the point of view of wave interactions. 
A main result established below is the existence of a large class of splitting-merging solutions 
containing one or two big (classical or nonclassical) waves, propagating within a background of small waves.
In short, we will establish the following result:  Consider the initial value problem associated with a 
strictly hyperbolic system \eqref{11} with GNL, LD, or CC fields. 
Suppose that the initial data is a perturbation of a large classical wave which splits into a nonclassical wave and a big classical. 
Then, there exists a globally defined in time solution with finitely many large waves. 
The two big waves can merge and, later on, split due to interactions between themselves or with small waves. 
In presence of the nucleation criterion, this splitting-merging phenomena only takes place finitely many times, 
but can continue forever if the nucleation criterion is removed.  

Let us stress that the long-term objective is to establish an existence result for initial data of small total variation
with nonclassical shocks. As a step towards this objective,
this research has treated in detail small perturbations to a single strong isolated nonclassical shocks
using LeFloch's theory of kinetic functions in order to constrain the family of admissible shocks. The main
novelties are the new measure of shock strength, extending a previous study
of scalar problems \cite{LL}, and a new weighted interaction potential 
inspired by earlier works \cite{Liu-monograph,IL}. This work explicitly identifies the 
relationship between the strength of the nonclassical shocks and of the perturbation,
even as the strength of the nonclassical wave vanishes. Finally, it is shown
that a nucleation condition provides additional stability to the splitting-merging process, thereby generalizing
a result of LeFloch and Shearer for scalar conservation laws~\cite{LeFlochShearer}. The 
techniques described here appear to be applicable to the analysis of solutions to
problems with general initial data.

An outline of this paper is as follows. In Section~2, we introduce our nonclassical Riemann solver with kinetics and nucleation. 
Section~3 contains the construction of the splitting-merging functional and presents our generalized wave strength and total variation functionals. Section~4 is devoted to showing that the nucleation condition implies that only finitely many splittings and mergings may occur.  

%================================================================================================

\section{Nonclassical Riemann solver with kinetics and nucleation}
\label{Sec2} 

\subsection{Background and notation} 

Solutions to the Riemann problem are made of shock waves and rarefaction waves, and we will first 
introduce some basic notation relative to such waves. The results can be found in~\cite{HL2} and in
the monograph of LeFloch~\cite{LeFloch-book}.   
 
A left-hand state $u_-$ being fixed, we consider the Hugoniot set consisting of all right-hand states 
$u_+$ satisfying, for some $\lam$, the condition 
$- \lam \, \bigl(u_+ - u_-\bigr) + f(u_+) - f(u_-) =0.$   
It is easily checked that there exist $\delta_1 < \delta_0$ and $\eps>0$ such that 
for each $u_- \in \Bone$ the following holds. The Hugoniot set  
can be decomposed into $N$ curves $\HH_j(u_-) = \Big\{v_j(s; u_-) \, / \, s \in (-\eps, \eps) \Big\}$  
($1 \leq j \leq N$) depending smoothly on $s$ and $u_-$, such that 
\be
v_j(s;u_-) = u_-  + s \, r_j(u_-) 
              + {1 \over 2}  \, s^2\, \big( Dr_j(u_-)  \, r_j \big)(u_-) 
	    + O(s^3) 
\label{19} 
\ee
and, for the shock speed $\lam = \lamb_j(s;u_-)$,   
$$
\aligned 
\lamb_j(s;u_-) = &\lam_j(u_-)
 + {1 \over 2} \, s \, \bigl( \nabla \lam_j \cdot r_j \bigr)(u_-)  \\
& + {1 \over 6}  \, s^2\, \Big( \bigl(\nabla ( \nabla \lam_j \cdot r_j) \cdot r_j \bigr)
             + {\nabla \lam_j \cdot r_j \over 2 } \, l_i (Dr_j  \, r_j) \Big) (u_-) + O(s^3). 
\endaligned 
\label{110}
$$  
Recall that we have assumed a normalization of the eigenvectors $r_j(u)$ for genuinely nonlinear fields
for which $\nabla \lambda_j \cdot r_j(u) > 0$ for all $u \in \BB(\delta_0)$.

Rarefaction waves are Lipschitz continuous, self-similar solutions of \eqref{11}, determined by integral curves
of the vector fields $r_j$. It is easily checked that, for $\delta_1 < \delta_0$ and $\eps>0$,
and for all $u_- \in \BB(\delta_1)$ the following holds. 
For each $j \in \bigl\{1, \ldots, N\bigr\}$ the integral curve $\OO_j(u_-)= \Big\{w_j(s; u_-) \, / \, s \in (-\eps, \eps) \Big\}$ 
depends smoothly upon $u_-$ and $s$ and satisfies  
\be
    w_j(s;u_-) = u_- + s \, r_j(u_-) + {1 \over 2}  \, s^2\, (Dr_j \, r_j)(u_-) + O(s^3)
\label{120}
\ee  
and 
$$
\lam_j(w_j(s;u_-)) = \lam_j(u_-) 
+ s \, \big( \nabla \lam_j \cdot r_j \big) (u_-) 
+ {1 \over 2} \, s^2 \bigl(\nabla ( \nabla \lam_j \cdot r_j) \cdot r_j\bigr)
            (u_-) + O(s^3). 
\label{121} 
$$

To ensure the uniqueness of the weak solution, we impose the entropy 
inequality \eqref{12} associated with a given strictly convex entropy pair $(U,F)$. 
For a shock wave connecting $u_-$ to $u_+$, the entropy inequality imposes 
that the rate of entropy dissipation be non-positive, i.e.  
\be
E(u_-, u_+) := - \lam \, \bigl(U(u_+) - U(u_-)\bigr) + F(u_+) - F(u_-) \leq 0. 
\label{123} 
\ee
Hence, we have a constraint on the sign of the entropy dissipation 
$\mybar E(s; u_-) := E(u_-, v_i(s; u_-))$ 
along the $j$-Hugoniot curve issuing from $u_-$. 
In fact, one has 
\be
\mybar E(s;u_-) = {s^3 \over 12} \left(r_j^T \,  D^2 U \, r_j \right)(u_-) 
\, \big( \nabla \lam_j \cdot r_j  \big) (u_-) + O(s^4).  
\label{124}
\ee

Recalling Lax shock inequalities
\be
\lam_j(u_-) \geq \lam \geq \lam_j(u_+)  
\label{126} 
\ee
for a shock wave connecting $u_-$ to $u_+$ at the speed $\lam$, we state the 
following standard result about genuinely nonlinear fields. 
Consider the $j$-Hugoniot curve $s \mapsto u_+=v_j(s; u_-)$ 
issuing from a state $u_-$ and associated with a genuinely nonlinear 
family. Then, the shock speed $s \mapsto \lamb_j(s;u_-)$ is an increasing 
function of $s$, and Lax shock inequalities  
are equivalent to imposing the entropy inequality \eqref{123} and select the part $s<0$ of the 
Hugoniot curve. Both inequalities above are violated for $s>0$. 
The composite curve $\WW_j(u_-):= \SSS_j(u_-) \cup \OO_j(u_-)$,
 called the $j$-wave curve issuing from $u_-$, is defined as 
$$
\aligned 
& \WW_j(u_-) = \SSS_j(u_-) \cup \OO_j(u_-) 
=: \bigl\{ \psi_j(s ; u_-) \, / \, s \in (-\eps, \eps)\bigr\}, 
\\
& \psi_j(s; u_-) = 
\begin{cases} 
v_j(s; u_-),      &  s \in (-\eps, 0], 
\\
w_j(s; u_-),      &  s \in [0, \eps). 
\end{cases} 
\endaligned 
\label{128} 
$$
The mapping $\psi_j : (-\eps, \eps) \times \Bone \to \RN$ 
admits continuous derivatives up to second-order and 
bounded third-order derivatives in $s$ and $u_-$, with 
$$
\psi_j(s; u_-) = u_- + s \, r_j(u_-) + {1 \over 2}  \, s^2\, (Dr_j \, r_j)(u_-) + O(s^3). 
\label{129} 
$$ 

%------------------------------------------------------------------------------------------------------ 

\subsection{Construction based on Lax shock inequalities} 

We now turn our attention to the characteristic fields $i \in J_3$ that are not globally genuinely nonlinear and such that 
genuine nonlinearity condition fails at one point along each wave curve.
In this case, the Lax shock inequalities \eqref{126} still single out a unique Riemann solution, 
but the entropy inequality \eqref{123} is no longer sufficiently discriminating in itself
(as we will discuss in the following subsection).   
It is convenient now to use a global parameter defined below. 
 
We now assume that the scalar-valued function 
$m_i(u) := \nabla \lam_i(u) \cdot r_i(u)$ 
does not keep a constant sign and that  
$\MM_i = \big\{ u \in \BB(\delta_0) \, |\,  m_i(u) = 0  \big\}$ 
is a smooth affine manifold with dimension $N-1$ such that
the vector field $r_i$ is {\sl transverse to the manifold} $\MM_i$
and that  
$$
\text {Concave-convex (CC) field: } \nabla m_i(u) \cdot r_i (u) > 0,   \quad u \in \BB(\delta_0). 
\label{22} 
$$  
The transversality assumption above implies that   
the wave speed $\lam_i$ has  
one critical point (a minimum) along each wave curve, 
that is,
$\nabla \lam_i(u) \cdot r_i(u) =0$ if and only if $m_i(u)=0$. 
To ensure that $\MM_i \cap \Bone \neq \emptyset$, and hence that the problem 
is not reduced to the genuinely nonlinear case, we assume that the origin belongs to
$\MM_i$, possibly at the cost of an additional translation of $u$.

To parametrize the Hugoniot curves
$\HH_i(u_-)= \big\{v_i(m;u_-)\big\}$ and integral curves 
$\OO_i(u_-)= \big\{w_i(m;u_-)\big\}$,
we introduce a global parameter $\mu_i(u)\in \RR$ 
which depends smoothly upon $u$ and is strictly monotone along the wave curves, specifically
$$
\nabla \mu_i(u) \cdot r_i(u) \neq 0, \quad u \in \Bzero
\label{23} 
$$  
and 
$$
\mu_i(u) = 0 \quad \text{ if and only if } \quad m_i(u) = 0.
\label{24}
$$ 
The obvious choice is to set the parameter $\mu_i$ now coincides with the parameter $m_i$, that is, 
\be
\mu_i \bigl(v_i(m;u_-)\bigr) = m, \text{ for $m \leq m_i(u_-)$},  \quad  
\mu_i \bigl(w_i(m;u_-)\bigr) = m,  \text{ for $m \geq m_i(u_-)$.}
\label{25} 
\ee
We observe that although the eigenvector $r_i$ associated to the CC family $i$  can be parametrized smoothly throughout 
$\BB(\delta_0)$, the sign of $m_i = \nabla \lambda_i \cdot r_i(u)$ changes and hence 
the Hugoniot curve $v_i$ can no longer be parametrized everywhere so that negative $s$
corresponds to the entropy admissible portion of the curve.

From now on we consider a concave-convex field $i \in J_3$ and we restrict attention to shocks
that satisfy Lax shock inequalities. Let $u_-$ be a state with $\mu_i(u_-) >0$
and we consider the integral and Hugoniot curves.
 
\begin{lemma}[Characteristic speed along the integral curve] 
The $i$-characteristic speed along the integral curve 
$m \mapsto \lam_i(w_i(m; u_-))$
is a strictly convex function of $m$ which achieves its minimum value at $m=0$. In particular,
it is decreasing for $m<0$ and increasing for $m>0$. 
\end{lemma}

The part of the classical wave curve $\WW_i^c(u_-)$ associated with rarefaction waves 
is the part $m \geq \mu_i(u_-)$ of the integral curve $\OO_i(u_-)$.

\begin{lemma}[Characteristic speed and shock speed along the Hugoniot curve]  
The $i$-characteristic speed along the Hugoniot curve 
$m \mapsto \lam_i(v_i(m; u_-))$
is a strictly convex function. 
On the other hand, the $i$-shock speed along the Hugoniot curve 
$$
m \mapsto \lamb_i(m; u_-) := \lamb_i\left(u_-, v_i(m;u_-)\right)
$$
is a strictly convex function, which is either  globally increasing or else 
achieves a minimum value at a point, 
denoted by $\mu_i^\natural (u_-)$.  
At this critical value, the characteristic speed and 
the shock speed coincide: 
$$ 
\lamb_i(m; u_-) = \lam_i(v_i(m; u_-))  \quad \text{ at } m = \mu_i^\natural(u_-)
\label{28} 
$$ 
and, for some smooth function $e=e(u_-) < 0$,  
$$
\del_m v_i(\mu_i^\natural(u_-);u_-) = e(u_-) \, r_i\bigl(v_i(\mu_i^\natural(u_-);u_-)\bigr). 
\label{29} 
$$    
\end{lemma}

For definiteness, from now on we assume that the shock speed does admit a minimum value,
that is to say that $v_i(\mu_i^\natural(u_-),u_-) \in \Bone$.
The characteristic speed and the shock speed satisfy 
$$
\aligned 
& \lamb_i(m; u_-) - \lam_i(v_i(m; u_-)) > 0, \quad  
      m \in \big(\mu_i^\natural(u_-), \mu_i(u_-)\big),  
      \\
& \lamb_i(m; u_-) - \lam_i(v_i(m; u_-)) < 0,  \quad m < \mu_i^\natural(u_-)  
  \text{ or } m > \mu_i(u_-). 
\endaligned
\label{210} 
$$

\begin{proposition}[Classical wave curves]
There exists $\delta_1<\delta_0$ and $\eps>0$ with the following property. 
For $i \in J_3$, $u_- \in \Bone$ and $\mu_i(u_-) > 0$, the {\sl $i$-wave curve of right-hand states 
connected to $u_-$} by a combination of $i$-elementary waves satisfying Lax shock inequalities
is composed  of the following three pieces:
$$
\WW_i^{nc}(u_-) \, = \, 
\begin{cases} 
\OO_i(u_-),            & \mu_i(u) \geq \mu_i(u_-), 
\\
\HH_i(u_-),            & \mu_i^\natural(u_-) \leq \mu_i(u) \leq \mu_i(u_-), 
\\ 
\OO_i(u_-^\natural),      & \mu_i(u) \leq \mu_i^\natural(u_-). 
\end{cases}  
$$
where $u_-^\natural := v_i(\mu_i^\natural(u_-); u_-)$. 
It is continuously differentiable with bounded second-order derivatives in $m$ and $u_-$ and satisfies 
\begin{equation}
\psi_i(m; u_-) 
= u_- + {m-\mu_i(u_-) \over \bigl(\nabla \mu_i \cdot r_i \bigr) (u_-)} \,  r_i(u_-) 
+ O\bigl(m-\mu_i(u_-)\bigr)^2. 
\label{27} 
\end{equation}
The second-order derivatives of wave curves for non-genuinely nonlinear fields are {\sl not continuous,} in general. 
\end{proposition}

It is easy to check that the expansion \eqref{27} still holds for states $u_-$ satisfying $\mu_i(u_-)<0$
for our extension of $r_i$ to all of $\BB(\delta_0)$.
  
%------------------------------------------------------------------------------------------------------------- 

\subsection{Nonclassical wave sets based on a single entropy inequality} 
  
This section generalizes Hayes-LeFloch's construction \cite{HL1,HL2} of a nonclassical Riemann solver based
on a kinetic function and presents 
our new Riemann solver based on, both, kinetics and nucleation. 

We now investigate the consequences of a single entropy inequality for the solutions of the Riemann problem. 
A shock wave is called a {\sl nonclassical shock} if it satisfies the single entropy inequality \eqref{123}
but not Lax shock inequalities \eqref{126}. It is called a {\sl classical 
shock} if both \eqref{123} and \eqref{126} hold. 
Before imposing any entropy condition at this stage, we have that a $i$-shock
wave connecting a left-hand state $u_-$ to a right-hand state $u_+$ 
can be:  
\begin{enumerate} 

\item {\sl a Lax shock}, satisfying 
$\lam_i(u_-) \geq \lamb_i(u_-, u_+) \geq \lam_i(u_+)$, 

\item {\sl a slow undercompressive shock:} 
$\lamb_i(u_-, u_+)  \, \leq \, \min\bigl(\lam_i(u_-), \lam_i(u_+)\bigr),$

\item a {\sl fast undercompressive shock:}  
$\lamb_i(u_-, u_+)  \, \geq \, \max\bigl(\lam_i(u_-), \lam_i(u_+) \bigr),$

\item or {\sl a rarefaction shock:}  
$\lam_i(u_-) < \lamb_i(u_-, u_+) < \lam_i(u_+).$
\end{enumerate}

Let $u_-$ be such that $\mu_i(u_-) >0$ and denote by $\mu_i^{-\natural}(u_-)$ the point of $\HH_i(u_-)$ such that  
\be
\lamb_i(\mu_i^{-\natural}(u_-); u_-) = \lam_i(u_-),  
\quad \mu_i^{-\natural}(u_-) < \mu_i^\natural(u_-), 
\label{35} 
\ee
whenever such a point exists. For simplicity in the presentation we 
assume that both points $u_-^{\natural} := \mu_i^\natural(u_-)$ and $u_-^{-\natural} := \mu_i^{-\natural}(u_-)$ 
belong to $\Bone$, since the present discussion would be 
much simpler otherwise. 
Note also that the shock wave connecting $u_-$ to $u_-^\nat$ is 
a {\sl right-contact,} while the shock 
connecting $u_-$ to $u_-^\mn$ is a {\sl left-contact.}  
 
Note also that the function $\mu_i^\natural$ depends smoothly upon its argument, 
and 
\be 
{\nabla\mu_i^\natural \cdot r_i} \sim - {1 \over 2} \, 
{\nabla\mu_i \cdot r_i} \quad \text{ near the manifold } \MM_i,
\label{213}
\ee
as established in \cite{LeFloch-book}. 
In particular, \eqref{213} implies that $\mu_i^\natural < 0$. 
Let $u_-$ be given with $\mu_i(u_-) \geq 0$ and consider a point $u_+$ 
on the Hugoniot curve $\HH_i(u_-)$, say $u_+= v_i(m; u_-)$ with $m=\mu_i(u_+)$. 
It can be checked that the shock  is
$$
\aligned 
&  - \text{ a rarefaction shock if } m > \mu_i(u_-) \, 
   \text{ or } m < \mu_i^{-\natural}(u_-), 
\\ 
&  - \text{ a Lax shock if } m \in \big[\mu_i^\natural(u_-), \mu_i(u_-)\big], 
\\
&  - \text{ an undercompressive shock if }  
        m \in \big[\mu_i^{-\natural}(u_-), \mu_i^\natural(u_-)\big). 
\endaligned
\label{37i}  
$$

\begin{lemma}[Entropy dissipation along the Hugoniot curve]   

\

\begin{enumerate}
\item[(i)]  The entropy dissipation  $\mybar E(m;u_-) := E(u_-,v_i(m;u_-))$ 
vanishes at $\mu_i(u_-)$ and at a point 
$$
\mu^\flat_{i0}(u_-) \in \big(\mu_i^{-\natural}(u_-), \mu_i^\natural(u_-)\big).
$$ 
The entropy dissipation is decreasing for 
$m < \mu_i^\natural(u_-)$, increasing 
for $m > \mu_i^\natural(u_-)$, and 
achieves a negative maximum value  
at the critical point of the wave speed $ \mu_i^\natural(u_-)$. 

\item[(ii)] A shock satisfying the entropy inequality \eqref{123} 
cannot be a rarefaction shock. A nonclassical shock is undercompressive
and satisfies 
$
m \in \big(\mu^\flat_{i0}(u_-), \mu_i^\natural(u_-)\big). 
$
\item[(iii)] Any shock satisfying Lax shock inequalities 
also satisfies the entropy inequality \eqref{123}. 
\end{enumerate} \label{lem:entropy_hugoniot}
\end{lemma}

We refer to $\mu^\flat_{i0}$ as the {\sl zero-entropy dissipation function} 
associated with the $i$-characteristic field. In view of Lemma \ref{lem:entropy_hugoniot} 
we can also define the companion function $\mu^\sharp_{i0}$ by 
$$ 
\mu^\flat_{i0}(u_-) < \mu_i^\natural(u_-) < \mu^\sharp_{i0}(u_-), \quad 
\lamb_i\bigl(\mu^\sharp_{i0}(u_-); u_- \bigr) 
= 
\lamb_i\bigl(\mu^\flat_{i0}(u_-); u_- \bigr). 
\label{38} 
$$
It can be checked by the implicit function theorem that $\mu^\flat_{i0}$ and $\mu^\sharp_{i0}$ are at least
as smooth as $f, F$ and $U$.

Given $u_-$, we will also denote by $\Phi_{i0}^\flat(u_-)$, $\Phi_i^\natural(u_-)$, 
and $\Phi_i^\sharp(u_-)$ on the shock curve $\HH_i(u_-)$
the states with coordinates $\mu^\flat_{i0}(u_-)$, $\mu_i^\natural(u_-)$, and $\mu_i^\sharp(u_-)$, respectively. 
It is a fundamental identity of the theory that the second iterate of the map 
$\Phi_{i0}^\flat$ is the identity map $id$ \cite{LeFloch-book}, that is, 
\be
\label{phiphi}
\Phi_{i0}^\flat \circ \Phi_{i0}^\flat = id.
\ee

Given a left-hand state $u_-$, the set of all states that can be reached using only $i$-waves
is called the {\sl nonclassical $i$-wave set issuing from $u_-$} 
and denoted by $\Se_i(u_-)$. 

\begin{proposition}[Nonclassical wave set]  
In addition to the classical one, the $i$-wave fan may 
contain a nonclassical $i$-shock 
connecting $u_-$ to some intermediate state $u_+ \in \HH_i(u_-)$ with 
$\mu_i(u_+) \in \big[\mu^\flat_{i0}(u_-), \mu_i^\natural(u_-)\big)$
followed by 
\begin{enumerate} 
\item[(a)]  either a {\rm non-attached} rarefaction connecting 
      $u_+$ to $u \in \OO_i(u_+)$ if $\mu_i(u) <\mu_i(u_+)$, 
\item[(b)]  or by a classical shock connecting $u_+$ to $u \in \HH_i(u_+)$ 
      if $\mu_i(u)>\mu_i(u_+)$. 
\end{enumerate} 
This defines a two-parameter family 
of right-hand states $u$ which can be reached from $u_-$ by nonclassical solutions.
\end{proposition}  

%------------------------------------------------------------------------------------

\subsection{Kinetic relation and nucleation criterion}

The nonclassical wave set $\Se_i(u_-)$ is a two-dimensional manifold within which we will select 
the nonclassical wave curve $\WW_i^{nc}(u_-)$. 
One parameter should be prescribed for each wave family in $J_3$, and
we postulate that
for all $u_-$ a {\sl single} right-hand state $u_+$ can be reached 
from $u_-$ across a nonclassical shock.   

For each concave-convex family $i$, a {\sl kinetic function} for the 
$i$-characteristic field is a Lipschitz continuous mapping  
$\mujf : \Bone \to \RR$ satisfying 
\be
\aligned 
\mu^\flat_{i0}(u) < \mu^\flat_{i}(u) \leq \mu_i^\natural(u), \quad \mu_i(u) > 0, 
\\
\mu_i^\natural(u) \leq \mu^\flat_{i}(u) < \mu^\flat_{i0}(u), \quad \mu_i(u) < 0. 
\endaligned 
\label{41}
\ee 
We impose the following {\sl kinetic relation}: for every $i$-nonclassical shock the right-hand state 
$u_+$ is determined from the left-hand state $u_-$ by 
\be
u_+ = u_-^\flat := v_i(m_-^\flat; u_-), \qquad  m_-^\flat = \mu^\flat_{i}(u_-). 
\label{42} 
\ee
To the kinetic function we shall associate its companion function $\mu_i^\sharp$ 
by 
\be
\lamb_i(u_-, u_-^\sharp) = \lamb_i(u_-, u_-^\flat),  \qquad  
u_-^\sharp := v_i(\mu_i^\sharp(u_-); u_-).
\label{43} 
\ee
It can be checked that the map $\mu_i^\sharp$ depends Lipschitz continuously upon its argument and 
\be
\aligned 
\mu_i^\natural(u) \leq \mu_i^\sharp(u) < \mu_i(u) \quad \text{ when }  \mu_i(u) > 0, 
\\
\mu_i(u) < \mu_i^\sharp(u) \leq \mu_i^\natural(u) \quad \text{ when }  \mu_i(u) < 0. 
\endaligned 
\label{419}
\ee

In the present paper, we need one additional map which allows us to select between the classical and the nonclassical construction of the wave fans. 
We introduce a Lipschitz continuous, {\sl nucleation threshold} function $\mu_i^n$ associated with the kinetic function and satisfying   
\be
\aligned 
& \mu_i^\natural(u_-) \leq \mu_i^n(u_-) \leq \mu_i^\sharp(u_-) \quad \text{ when }   \mu_i(u_-) \geq 0, 
\\
& \mu_i^\sharp(u_-) \leq \mu_i^n(u_+) \leq \mu_i^\natural(u_-) \quad \text{ when }  \mu_i(u_-) \leq 0.  
\endaligned 
\label{44}
\ee
The nucleation criterion is stated in the following way: 
\be
\label{state}
\aligned
& \text{For $|\mu_i(u_-) - \mu_i(u_+)| < |\mu_i(u_-) - \mu_i^n(u_-)|$, classical $i$-shocks 
are always preferred,}  
\\
& \text{ whenever they are available.}    
\endaligned
\ee
When $\mu_i^n(u_-) = \mu_i^{\sharp}(u_-)$, then the nucleation criteria is superfluous.

The Riemann solver with kinetics and nucleation is defined as the solution to the Riemann problem 
that satisfies the entropy inequality, the kinetic relation, and the nucleation criterion. 
In other words, a large initial jump always ``nucleates'' and produces a nonclassical shock wave. 
This is consistent with observations made in phase transition dynamics, especially in
material science where nucleations in austenite-martensite transformations are observed \cite{AK2}.

\begin{theorem}[Nonclassical Riemann solver with kinetics and nucleation]  \label{thm:solver_nucleation}
Suppose that all fields of the strictly hyperbolic system \eqref{11} are genuinely 
nonlinear, linearly degenerate, or concave-convex. 

\begin{enumerate}
\item Consider a concave-convex field $i \in J_3$ and fix a Lipschitz continuous, kinetic function $\mu^\flat_{i}$
satisfying \eqref{41}. Then, for each $u_- \in \Bone$ the kinetic relation 
\eqref{42} and the nucleation criterion \eqref{44} select a unique 
{\rm nonclassical $i$-wave curve} $\WW_i^{nc}(u_-)$ within the nonclassical wave set $\Se_i(u_-)$. 
When $\mu_i(u_-)>0$ it is composed of the following two connected components: 
$$
\WW_i^{nc}(u_-) \, = \, 
\begin{cases} 
\OO_i(u_-),            & \mu_i(u) \geq \mu_i(u_-), 
\\
\HH_i(u_-),            & \mu_i^n(u_-) \leq \mu_i(u) \leq \mu_i(u_-), 
\\
\HH_i(u_-^\flat),      & \mu^\flat_{i}(u_-) \le \mu_i(u) < \mu_i^\sharp(u_-), 
\\
\OO_i(u_-^\flat),      & \mu_i(u) \leq \mu^\flat_{i}(u_-),  
\end{cases}  
$$
where $u_-^\flat := v_i\big(\mu^\flat_{i}(u_-);u_-\big)$. The Riemann 
solution is a single rarefaction shock, or a single classical shock,
or a nonclassical shock followed by a classical shock, or finally a 
nonclassical shock followed by a rarefaction, respectively. 
Each connected component of $\WW_i^{nc}(u_-)$ is a continuous and monotone curve in the parameter 
$m=\mu_i(u)$. When $\mu_i^n(u_-)=\mu_i^\sharp(u_-)$, it has bounded 
second-order derivatives for all $m \neq \mu_i^\sharp(u_-)$ and is Lipschitz 
continuous (at least) at $m=\mu_i^\sharp(u_-)$.  

\item For all $u_l$ and $u_r$ in $\Bone$ the Riemann problem \eqref{11} 
and \eqref{12} admits a unique solution satisfying a single entropy inequality, 
the kinetic relation, and the nucleation criterion, which 
is determined by the intersection of the classical wave curves $\WW_j^c$ 
for genuinely nonlinear fields and the nonclassical wave curves $\WW_i^{nc}$ for concave-convex fields $i$.  
The solution is a self-similar solution made of $N+1$ constant states 
$$
u_l = u^0, \, u^1, \ldots, \, u^N = u_r
$$
separated by $j$ wave fans. The intermediate constant states satisfy 
$$
u^j \in \WW_j(u^{j-1}),     \quad    u^j = \psi_j(m_j; u^{j-1}), 
$$
with $m_j:= \muj(u^{j-1})$, and are connected with combination of classical or nonclassical waves.  
\end{enumerate}    
\end{theorem}

We recover the classical wave curve $\WW_i^c(u_-)$ with the trivial choice $\mu^\flat_{i} \equiv \mu_i^n \equiv \mu_i^\natural$. 
Observe that, contrary to the scalar case,
the curve $\HH_i(u_-^\flat)$ is continued up to the value $\mu_i^\sharp(u_-)$, 
since the intersection with a wave curve from another family (which enters in the construction of the Riemann
solution) may generate a point $u$ which has $\mu_i(u) > \mu_i^n(u_-)$, while the intersection of the classical 
curve $\HH_i(u_-)$ (with $\mu_i^n(u_-) \leq \mu_i(u) \leq \mu_i(u_-)$) is empty. 

It is natural to ask about the physical relevance of the proposed hyperbolic theory with kinetics and nucleation.   
It is expected that the nucleation criterion is required when instabilities (one-wave/two-wave patterns) 
are observed numerically or in regularization involving competition between 
viscosity, capillarity, or relaxation effects. 
It is conceivable that the nucleation function (as is the case for the kinetic relation) 
could be determined {\sl numerically} within a range
where both classical and nonclassical behaviors co-exist.  
Applications could be sought in the fields of phase transition dynamics,
thin film dynamics, and Camassa-Holm--type models. 
 
%*********************************************************************************

\section{Splitting-merging patterns with kinetics and nucleation}
\label{Sec3}

\subsection{Objectives and notation}

In the case that the nucleation threshold is chosen to coincide with the critical value determined by the kinetic mapping, the nonclassical Riemann solution constructed in the previous section depends continuously in the $L^1$ norm upon its 
initial data, but not so in a pointwise sense.  Observe that, along a nonclassical wave curve, the speeds of the (rarefaction or shock)
waves change continuously
since
$$
    \lim_{m \to \muj^{\sharp}+} \lam_j\big( v_j(m,u_-) \big) =  \lim_{m \to \muj^{\sharp}-} \lam_j\big( v_j(m,u_-^{\flat}) \big),
$$
and
$$
  \lim_{m \to \muj^{\flat}-} \lam_j\big( w_j(m,u_-^{\flat}) \big) =  \lim_{m \to \muj^{\flat}+} \lam_j\big( v_j(m,u_-^{\flat}) \big).
$$
However, at $\mujs(u_-)$, one has to compare the shock 
speed of the nonclassical shock with the shock speeds of the nonclassical and classical shocks.  
The wave speeds (only) are continuous but the total variation of the nonclassical Riemann solution 
is {\sl not} a continuous function of its end points. This lack of continuity makes it delicate to control the strengths of 
waves at interactions. This motivates the study made in the present section which covers general nucleation thresholds, including the ``trivial'' choice.  

We consider three distinct states $u_*$, $u_*^\sharp$, and $u_*^\flat$ associated with a given wave family (say $i \in J_3$), so that 
$u_*$ can be connected to $u_*^\sharp$ by a classical shock (denoted by $C_*^\downarrow$),  
$u_*$ can be connected to $u_*^\flat$ by a nonclassical shock (denoted by $N_*$), 
and $u_*^\flat$ can be connected to $u_*^\sharp$ by a classical shock (denoted by $C_*^\uparrow$); moreover, we assume that all three  speeds coincide, that is, 
\be
\lamb_i(u_*,u_*^\sharp) = \lamb_i(u_*, u_*^\flat) 
= \lamb_i(u_*^\flat, u_*^\sharp) 
=: \lamb_*. 
\label{3.1} 
\ee
Such a structure is directly determined by the kinetic relation that was prescribed in Section~\ref{Sec2}. 

\begin{definition}  The triplet $\big( u_*, u_*^\sharp, u^\flat_*\big)$ satisfying the condition above is called a classical--non\-classical three state pattern. 
\end{definition} 

The nonclassical shock is stable by perturbation in the phase space, in the sense that for every $u$ in a sufficiently small 
neighborhood of $u_*$ there exists a state $u^\flat \in \HH_i(u)$ which lies in a small neighborhood of 
$u^\flat_*$ and still satisfies the kinetic relation. The classical shock connecting $u_*^\flat$ to $u_*^\sharp$ is also stable in the 
same manner. However, the classical shock connecting $u_*$ to $u_*^\sharp$ is {\sl not a stable structure}  
when its end points are perturbed. Actually, depending upon the size of the perturbation 
the shock connecting $u_*$ to $u_*^\sharp$ may either keep its structure, or else transform 
into a perturbation $N^\downarrow$ of the nonclassical shock $N_*^\downarrow$, plus 
a perturbation $C^\uparrow$ of the classical shock $C_*^\uparrow$.  When the nucleation condition  \eqref{44} 
is present, then the instability of the classical shock connecting $u_*$ to $u_*^\sharp$ requires 
the perturbation to be strong enough to overcome the difference $\muis(u_*) - \muiN(u_*)$. 
In general, we will write $u_*^n$ to denote the state $\psi_i(\mu_iN(u_*);u_*)$, which will coincide with
$u_*^\sharp$ when nucleation is absent.

Following LeFloch and Shearer \cite{LeFlochShearer}, we are interested here in the effect of adding a perturbation to this double wave structure and establishing a global--in--time stability result, when the perturbation has sufficiently small total variation.  
Indeed, we are going to establish the existence of a large class of solutions undergoing repeated splitting and merging, whose global structure at each time consists of a single shock or a double-shock structure plus small waves. 
As was the case for scalar conservation laws, we will see that the splitting/merging feature can take place infinitely many times when the nucleation function coincides with the kinetic function but only finitely many times when a non-trivial nucleation criterion is selected. 
The analysis for systems is much more involved, since now interactions between different wave families may take place and
even interactions between waves of the same family may contribute to increase the wave strengths. 

The analysis of hyperbolic systems of conservation laws, since the work of Glimm \cite{Glimm},
proceeds by constructing an approximation within a simple class of solutions in which all waves,
in particular rarefactions, are localized to a point. Following the new standard Ansatz of front-tracking 
approximations \cite{Bressan-FT, Risebro-FT}, it suffices to construct piecewise constant approximations $u_h$
in space with a finite number of discontinuities which depend on a small parameter $h$, 
and for which the initial data $u_0$ converge in the $\operatorname{TV}$ norm
in space and in $L^1(\mathbb{R})$ to the initial data as $h$ converges to zero. 
The main argument is then to show that the
sequence of approximations satisfy the hypothesis of Helley's theorem. For these reasons, the
analysis in the rest of the paper will suppose the {\it solution}, 
denoted $u_h = u_h(t,x)$, is in fact a front-tracking approximation,
i.e. a piecewise constant approximation with discontinuities connecting states along waves,
all up to a small error in $L^{\infty}(\operatorname{TV}) \cap L^1(\mathbb{R} \times [0,T])$.
Usually, the dependence $h$ will be implicit and the front-tracking approximation 
will simply be denoted by $u$.

The front-tracking approximation $u$ we construct is associated with piecewise constant initial data.  
Given the state $u_*$ consider initial data $u(x,0)=u_0(x)$ having the following specific structure:  
\be
\label{3.2} 
u_0(x) = \overu_0(x) + v_0(x),
\ee
where 
\be
\overu_0(x) 
:= 
\begin{cases} 
u_*, & x < 0,
\\
u_*^n, & x > 0.
\end{cases}
\label{3.3}
\ee
We assume that the small perturbation $v_0$ is piecewise constant, formed of a collection
of discontinuities $\gamma \in \mathcal{D}_0$ located at $x_{\gamma}$, and of small total variation:  
\be
\eps_0 > TV(v_0) := \sum_{\gamma \in \mathcal{D}_0} \big| v_0(x_{\gamma}+) - v_0(x_{\gamma}-) \big|. 
\label{3.4}
\ee
The total variation above uses the euclidean distance between the states in phase-space to mesure the
strength of the discontinuity. In order to distinguish the strength of the initial perturbation from the
strength of the initial strong wave in the $i$-th family, we assume that $v_0(0) = 0$.
The strength of the large unperturbed shock $|u^n_* - u_*|$ provides a first length scale, the size of the perturbation 
$\eps_0$ provides a second, but smaller, length scale, and the nucleation criteria provides the third and final length scale 
\be
\eta := \muis(u_*) - \muiN(u_*), 
\label{3.5} 
\ee 
which we also assume to be the smallest of all three scales. Note that $\overu_0$ is a single step function 
which admits two distinct Riemann solutions $\overu=\overu(t,x)$ made of admissible waves.

For instance, if the perturbation $v_0$ is simply chosen to be a single step function 
having a jump discontinuity located at $x=0$ and connecting $0$ (say) to
$v_0^+$, then one just need to solve a Riemann problem with left-hand state $u_*$ and right-hand state 
$u_*^n + v_0^+$. In view of the properties of the nonclassical Riemann solver (Section 2), 
we can find a cone $\AAA$ centered at the point $u_*^n$ such that:  
\begin{enumerate}
\item[$\bullet$] If $\mui(u_*^n + v_0^+) > \mui(u_*^n)$ and $u_*^n + v_0^+ \in \AAA$ then 
the Riemann solution contains a single big shock $C^\downarrow$ plus small waves
in all families.
\item[$\bullet$] If $\mui(u_*^n + v_0^+) < \mui(u_*^n)$ and $u_*^n + v_0^+ \in \AAA$ 
the Riemann solution contains a big nonclassical shock $N^\downarrow$  
plus a faster big classical shock $C^\uparrow$ plus small waves in other families $j \neq i$.
\item[$\bullet$] When $u_*^n + v_0^+ \notin \AAA$ either one of the above behaviors may be 
observed. 
\end{enumerate}

%----------------------------------------------------------------------------------------------

\subsection{Generalized wave strength and interaction patterns}
\label{sec:generalized_strength}

Waves will be measured by a {\sl generalized strength}  
which extends the earlier definitions for scalar equations \cite{LeFloch-book,LeFlochShearer, LL}. 
For genuinely nonlinear and linearly degenerate families, say indexed $j$, the strength of a simple wave connecting 
$u_-$ on the left to $u_+$ on the right will be given by the classical definition
$$
    \sigma(u_-,u_+) = \mu_j(u_+) - \mu_j(u_-).
$$
For every concave-convex family indexed $i$, we propose here to use the zero--dissipation map 
$\Phifo$ and identify any state in the $\mu_i < 0$ region to a state in the $\mu_i > 0$ region. Precisely, we introduce the map 
$$
\map(u) = \begin{cases}
u & \text{if } \mui(u) > 0,
\\
\Phifo(u) & \text{if } \mui(u) <0.
\end{cases} 
$$

\begin{definition} \label{defn:waveLL}
The signed strength of a wave joining states $u_-$ and $u_+$  in a concave-convex family indexed $i$ is measured by
\begin{equation}  \label{defn:ncwavestrength}
  \sigma(u_-,u_+) = 
             \mui \bigl( \map( u_+) \bigr) -\mui \bigl( \map(u_-) \bigr). 
\end{equation}
\end{definition}
 
In other words, we have 
$$
  \sigma(u_-,u_+) =
  \begin{cases}
             \mui(u_+) -\mui(u_-)          & \text{if } \mui(u_-), \mui(u_+) > 0,
 \\
              \mui \bigl( \Phi^\flat_{i0}(u_+) \bigr) -\mui(u_-)    &       \text{if } \mui(u_+) \leq 0 < \mui(u_-),
\\
             \mui(u_+) -\mui \bigl( \Phi^\flat_{i0}(u_-) \bigr)    & \text{if } \mui(u_-) \leq 0 < \mui(u_+),
 \\
             \mui \bigl( \Phi^\flat_{i0}( u_+) \bigr) -\mui \bigl( \Phi^\flat_{i0}(u_-) \bigr)    & \text{if } \mui(u_-), \mui(u_+) \leq 0.  
  \end{cases}
$$ 
This definition will be used for both strong and weak waves of the $i$-th family, that is 
the concave-convex family. In particular, for weak waves of the $i$-th family only the first and last cases above apply.
Finally, we observe that the weak rarefactions and weak shocks have strengths that are respectively 
positive and negative.

We can now summarize our assumptions on the kinetic map $\Phif_i$ for each concave-convex characteristic field, as follows: 
\bei

\item[(H1)] $\Phif_i(u) \in \WW_i^{nc}(u)$ for all $u \in \Bone$.  Throughout, we always assume that $\delta_1 \leq \delta_0$ is small enough so that, for all $u \in \Bcal (\delta_1)$, the two states $\Phi^\sharp(u)$ and $\Phif(u)$ are well-defined (and, of course, belong to $\Bzero$). 

\item[(H2)] $\Phif_i: \Bone \to \RR^n$ is Lipschitz continuous and one--to--one on its image. 

\item[(H3)] The restriction of $\Phif_i$ to the hypersurface $\MM_i$ is the identity map, and the function
$s \mapsto \mu_i(\Phif_i(\psi_i(u; s)))$ is monotone descreasing for all $u \in \Bone$, provided $s$ satisfies $\psi_i(u;s) \in \Bone$.

\item[(H4)]  There exists a constant $\Cff \in (0,1)$ such that 
                   $\big| \mu_i \big( \Phi^\flat_{i0} \circ \Phif_i(u)\big) \big| \leq \Cff \, |\mu_i(u)|$ 
                  for all $u \in \Bone$.

\eei 
The property (H4) was expressed in a form that will be convenient for later analysis,
but it should be pointed out that it is equivalent, because of property \eqref{phiphi}, 
to the strict and uniform contractivity of $ \Phif_i \circ \Phif_i$. 

%*************************
\begin{lemma}      \label{lem:equivalence_norm}
Under the conditions (H1)--(H4),  the generalized shock strength satisfies the following properties: 

\bei
\item When the nucleation map $\Phi^n_i$ coincides with the map $\Phis_i$ determined by the kinetic map, 
the generalized wave strength is {\sl continuous}
(up to perturbation due to small waves) when the two large waves combine together or when a classical shock splits.  

\item The generalized wave strength is equivalent to the standard wave strength. 
\eei
\end{lemma} 

\begin{proof}
Consider the states $\big( u_*, u_*^\sharp, u^\flat_*\big)$ forming a classical-non-classical three state wave pattern.
For an arbitrarily small perturbation, the classical shock  connecting $u_*$ to $u_*^\sharp$ can be split into a 
non-classical shock between $u_*$ and $u_*^\flat$ followed by a classical shock  connecting $u_*^\flat$
to $u_*^\sharp$. Using the definition of wave strength, we verify
\begin{align*}
   \sigma(u_*,u_*^\sharp) = & \, \mu_i\big(\Phifo \circ \Phis_i(u_*)\big) - \mu_i(u_*) \\
    = & \, \mu_i\big(\Phifo \circ \Phis_i(u_*)\big) - \mu_i\big(\Phifo \circ \Phif_i(u_*)\big) + \mu_i\big(\Phifo \circ \Phif_i(u_*)\big) - \mu_i(u_*) \\
    = & \,   \sigma(u_*,u_*^\flat) +   \sigma(u_*^\flat,u_*^\sharp),
\end{align*}
which establishes the continuity property. 

To prove the equivalence, we must show that there exists 
strictly positive constants $c$ and $C$ such that for every $j$-wave connecting two states $u_-, u_+$, then 
$$
  c \big|  \mu_j(u_+) - \mu_j(u_-) \big| \leq  \big| \sigma(u_-,u_+) \big| \leq C \big| \mu_j(u_+) - \mu_j(u_-) \big|.
$$
If $j \neq i$, that is if the wave does not belong to the same concave-convex family of the strong shock, then 
the definition of wave strength is the same as usual and there is nothing to show. Consider therefore the case with $j = i$
with $u_+ \in \mathcal{W}^{nc}_i(u_-)$. If both $\mu_i(u_-)$ and $\mu_i(u_+)$ are positive, then 
$|\sigma(u_-,u_+) | = |\mu_i(u_+) - \mu_i(u_-)|$ is the usual definition. If both $\mu_i(u_-)$ and
$\mu_i(u_+)$ are negative, then because $\Phiif$ is Lipschitz and one-to-one 
there exists $c$ and $C$ such that
$$
 c \big| \mu_i(u_+)  -   \mu_i(u_-)   \big| 
          \leq  | \sigma(u_-,u_+)| = \Big| \mu_i\big( \Phiif(u_+) \big) -   \mu_i\big( \Phiif(u_-) \big)  \Big|
           \leq C  \big| \mu_i(u_+)  -   \mu_i(u_-)   \big|.
$$

Consider now the last case where $\mu_i(u_-) > 0 > \mu_i(u_+) $. 
First of all, the upper bound
obviously holds $|\sigma(u_-,u_+)| <  |\mu_i(u_+)  -   \mu_i(u_-) |$. To derive the lower bound,
we observe that for a single wave,  $\mu_i(u_+) > \mu_i(\Phif_i(u_-))$.
Given that $\mu_i ( \Phiif ( \psi_i(u_-;m)))$ is monotone decreasing, clearly
\begin{align*}
     \big| \sigma(u_-,u_+) \big| = & \, \Big| \mu_i\big( \Phifo(u_+) \big) - \mu_i (u_-) \Big| \\
           \geq & \, \Big| \mu_i\big( \Phifo \circ \Phiif(u_-) \big) - \mu_i (u_-) \Big| 
           \geq  (1-\Cff) \big| \mu_i(u_-) \big|.
\end{align*}
The Lipschitz continuity of $ \Phif_i$ implies that there
exists a constant $B^\flat$ such that 
$$
   \big| \mu_i(\Phif_i(u_-)) \big| \leq B^\flat \big| \mu_i(u_-) \big|.
$$ 
Continuing from our earlier lower bound, we find
\begin{align*}
     \big| \sigma(u_-,u_+) \big|  \geq & \, \frac{1-\Cff}{1+B^\flat} (1+B^\flat) \big| \mu_i(u_-) \big| \\
      \geq & \, \frac{1-\Cff}{1+B^\flat}  \Big(  \big| \mu_i(u_-) \big| +  \big| \mu_i(\Phif_i(u_-)) \big| \Big) 
      \geq  \frac{1-\Cff}{1+B^\flat} \big| \mu_i(u_+)  -   \mu_i(u_-)  \big|.
\end{align*}
This completes the proof of the equivalence.
\end{proof}

We are now in a position to state our main result. 

\begin{definition} A {\sl splitting--merging pattern} is a nonclassical entropy solution to \eqref{11} which contains at most two big waves located at $y=y(t) \leq z=z(t)$ separated by small classical waves, so that at each time the solution contains either (i) only  a single large classical shock,
or (ii) a large nonclassical shock followed by a large classical shock. 
\end{definition}

\begin{theorem}[Nonlinear stability of splitting--merging patterns]\label{mainresult} 
Given any strictly hyperbolic system \eqref{11} with genuinely nonlinear, linearly degenerate, 
or concave-convex characteristic families,  consider a kinetic mapping $\Phif_i$ 
associated with a concave-convex family (indexed by $i$), uniformly satisfying the conditions (H1)--(H4). 
Then, any two wave pattern $\big( u_*, \Phif_i(u_*), \Phis_i(u_*) \big)$ is globally-in-time nonlinearly 
stable, that is,  for every perturbation with sufficiently small total variation  
there exists a global--in--time solution $u=u(t,x)$ to \eqref{11} which is 
a splitting--merging pattern satisfying the initial condition \eqref{3.2}. 
\end{theorem}

\begin{theorem}[Nonlinear stability of splitting--merging patterns with arbitrarily small strength]\label{mainresult_small} 
Given any strictly hyperbolic system \eqref{11} with genuinely nonlinear, linearly degenerate, 
or concave-convex characteristic families, there exist positive constants $\kappa_*$ and $\delta_* < \delta_0$
depending only upon the flux function and the constant $\Cff$ arising in (H4), so that the following property holds: 
for every family of kinetic mappings $\Phif_i$ 
associated with concave-convex families (indexed by $i$), provided they satisfy the conditions (H1)--(H4), 
any three wave pattern $\big( u_*, \Phif_i(u_*), \Phis_i(u_*) \big) \in \mathcal{B}(\delta_*)^3$ 
is globally-in-time nonlinearly stable, that is,  
for every perturbation with total variation $\eps_0 < \kappa_* |\sigma(u_*,\Phis_i(u_*))|$ 
there exists a global--in--time solution $u=u(t,x)$ to \eqref{11} which is 
a splitting--merging pattern satisfying the initial condition \eqref{3.2}. 
\end{theorem} 

Observe that the above theorem applies to a broad class of kinetic functions and allows the strength of the classical/nonclassical pattern to have arbitrarily small strength, while the perturbation is allowed to be a small (but fixed) ratio of those waves.

%-----------------------------------------------------------------------------------------------------

\subsection{Interaction estimates} 
\label{sec:interaction_estimates}

The purpose of this section is to introduce a functional $\myW$ that measures the total variation
of the weak waves and provides estimates for its variation during the interactions which occur in splitting-merging
solutions. This will require an extension of Glimm's interaction estimates~\cite{Glimm,Smoller,Dafermos} to the
interactions involving non-classical waves. 

We begin with our construction of a functional equivalent to the
total variation of the weak waves and depending on parameters whose values
will be determined in Section \ref{sec:globalestimates}. Henceforth, it will be convenient
to use the symbols  $\Npm, \CUp, \CDown$ to denote respectively, the non-classical
wave, its classical wave partner connecting increasing states, and the classical wave
present when the non-classical wave is absent.
We define $y^h(t)$ to be the position of either crossing wave,  $\Cpm$ or $\Npm$,
 in the front-tracking approximations and, when $\CUp$ is present, we define $z^h(t)$ to be the position of $\CUp$.
 When $\CUp$ doesn't appear, we assume $z^h = y^h$.

 The shock positions allow us to define three functionals, each associated to specific families of waves, 
 \begin{gather*}
       V_{L,k}(t) = \sum_{\substack{\text{$k$-waves}  \\x < y_h(t)}} |\sigma_{k,x}|, \qquad  
       V_{M,k}(t) = \sum_{\substack{ \text{$k$-waves} \\ y^h(t) < x < z^h(t)}} |\sigma_{k,x}|, \qquad 
       V_{R,k}(t) =  \,\, \sum_{\substack{ \text{$k$-waves} \\ z^h(t) < x}} |\sigma_{k,x}|.
 \end{gather*}
 In practice, a linear combination of these measures of variation will be used. Assuming the
 strong waves belong to the $i$-th family of waves, then  we introduce 
 the positive constants $k_L, k_M, k_R$, 
 $k^{\lessgtr}_L, k^{\lessgtr}_M$ and $k^{\lessgtr}_R$, and the functionals
 \begin{align*}
      V_L(t) = & \,\, \kappaL V_{L,i}(t)  +   \kappaLless \sum_{j < i}  V_{L,j}(t) +  \kappaLgrt  \sum_{j > i} V_{L,j}(t), \\
      V_M(t) = & \,\, \kappaM V_{M,i}(t)  +   \kappaMless \sum_{j < i}  V_{M,j}(t) +  \kappaMgrt  \sum_{j > i} V_{M,j}(t),\\
      V_R(t) = & \,\, \kappaR V_{R,i}(t)  +   \kappaRless \sum_{j < i}  V_{M,j}(t) +  \kappaRgrt  \sum_{j > i} V_{R,j}(t), \\
   \myW(t) = & \,\,  V_L(t) + V_R(t) + V_M(t).
 \end{align*}
For convenience, we will also be writing $|\, \alpha| := |\sigma(\alpha)|$ and
$$
  \operatorname{sign}( \alpha ) = 
   \begin{cases}
        -1, & \text{ $\alpha$ is a shock,} \\
        +1, & \text{ otherwise,}
   \end{cases}
$$

We start by listing all possible wave interactions which may occur within a splitting-merging solution. All possible interactions in 
the systems case are, up to the generation of weak waves in families $i \neq j$, a subset of those which may be found
in the scalar case and described in LeFloch's monograph \cite{LeFloch-book}.
Following the notation introduced there, let $C$ and $N$ denote respectively classical and nonclassical waves. Outgoing 
waves will be distinguished by a prime superscript, such as $C'$, and if waves
connect states $u_- > u_+$ ($u_- < u_+$) then we will write $\CDown$ ($\CUp$)
depending on the type of wave. Sometimes, we may also include a subscript $\pm$ ($\mp$) to indicate more precisely
that a wave  $C_{\pm}$ ($C_{\mp}$) connects $u_- > 0 > u_+$ ($u_- < 0 < u_+$).  
Finally, lowercase greek letters, like $\alpha_i$ and $\beta_j$, will denote weak waves in respectively
the $i$-th and $j$-th family. Using the notation above, interactions involving two $i$-waves are the following.

%\begin{flushleft}
\begin{enumerate}

\item[Case 1.]  {\bf Splitting of a big wave:} $C^\downarrow \alpha_i \longrightarrow N C^\uparrow$ or 
$\alpha_i C^\downarrow \longrightarrow N C^\uparrow$.
% (the decreasing shock is split). 

\item[Case 2.]  {\bf Merging of two big waves:} $N C^\uparrow \longrightarrow C^\downarrow$.
% (the nonclassical shock and the decreasing shock merge). 

\item[Case 3.] {\bf Perturbation of a big classical wave:} $C \alpha_i \longrightarrow C$ or 
$\alpha_i C  \longrightarrow C$ with $C = \Cpm, \CUp$.
% (the decreasing shock remains stable). 

\item[Case 4.] {\bf Perturbation of a big nonclassical wave:}  $\alpha_i N \longrightarrow N \alpha_i'$. 

%\item[Case 5.]  {\bf Bouncing against a big nonclassical wave:} $N C^\uparrow \longrightarrow N C^\uparrow$.
% (the nonclassical shock and the  decreasing shock bounces back on each other). 

\end{enumerate}
%\end{flushleft}

Secondly, interactions may involve a big $i$-wave and a small wave of another family $i \neq j$: 
\begin{enumerate}

\item[Case 5.]  {\bf Splitting of a big wave:} $C^\downarrow \beta_j \longrightarrow N C^\uparrow$ or 
$\beta_j C^\downarrow \longrightarrow N C^\uparrow$.
%  (the decreasing shock is split);  

\item[Case 6.]  {\bf Perturbation of  a big nonclassical wave:} $N \beta_j \longrightarrow \beta_j' N $.
% (the nonclassical shock remain stable, up to the creation of a small $i$-wave). 

\item[Case 7.] {\bf Perturbation of  a big classical wave:} $C \beta_j \longrightarrow \beta_j' C$ or 
$\beta_j C  \longrightarrow C \beta_j'$ with $C = \Cpm, \CUp$.
%  (the decreasing shock remains stable); 

%\item[Case 9.] {\bf Perturbation of a big increasing classical wave:} $C^\uparrow \beta_j \longrightarrow C^\uparrow$ or 
%$\beta_j C^\uparrow  \longrightarrow C^\uparrow$.
% (the increasing shock remains stable); 

%\item[Case 9.]  {\bf Perturbation of a big nonclassical wave.} $N \beta_j \longrightarrow \beta_j' N$. 
\end{enumerate}

That these are the only possible interactions within a splitting-merging solution is clear for very short times
but will only hold at all times when we show that the strengths of the weak waves remain small.
%Finally, the case of interaction between two small waves ($i$ and/or $j$-waves)
%is standard and will be handled via Glimm's wave interaction potential. 

We now introduce an extension of Glimm's interaction estimates to the interactions given above not previously covered
by interactions between classical waves. Informally, these estimates show that up to a quadratic error,
the strengths of the outgoing waves are the sum of the strengths of the incoming waves.
Classically, if $\boldsymbol{\alpha} = (\alpha_1, \ldots, \alpha_N)^T \in \mathbb{R}^N $ is the vector
of the signed strengths of the $N$ waves connecting $u_\ell$ to $u_m$ according to the 
Riemann problem, then it means that there are $N-1$ intermediate states connecting $u_\ell$
to $u_m$ defined by 
\begin{align*}
    \widetilde{u}_1 = & \, \psi_1(\alpha_1;u_\ell) , \\
    \widetilde{u}_2 = & \, \psi_1(\alpha_1;\widetilde{u}_1) , \\
        \vdots  \, & \\
     u_m = & \, \psi_1(\alpha_1; \widetilde{u}_{N-1}).
\end{align*}
Hence, if $\boldsymbol{\alpha} \in \mathbb{R}^N$ describes the solution to the Riemann problem connecting
$u_\ell$ to $u_m$ and $\boldsymbol{\beta} \in \mathbb{R}^n$ does the same between
$u_m$ and $u_r$, then the solution $\boldsymbol{\gamma}$ to the Riemann problem
connecting $u_\ell$ to $u_r$ satisfies
$$
    \boldsymbol{\gamma} = \boldsymbol{\alpha} + \boldsymbol{\beta} + \big( \text{ quadratic }\big).
$$
Unfortunately, it is not trivial to extend the previous estimate to simple waves from a concave-convex
family $i$ connecting $u_-$ and $u_+$ because the map $\psi_i$ does not satisfy
\be   \label{failure-addition}
     u_+ = \psi_i \big( \sigma(u_-,u_+); u_-\big),
\ee
for $\sigma$ defined by \eqref{defn:ncwavestrength}. Hence, care will be needed when discussing
the parametrization of the wave curves. Before proceeding,
recall that two simple waves $\boldsymbol{\alpha}$ and $\boldsymbol{\beta}$, with  $\boldsymbol{\alpha}$
to the left of  $\boldsymbol{\beta}$ and belonging repesctively to families $k$ and $\ell$,
are said to be {\it approaching} if $k > \ell$ or if $k=\ell$ but at least one of the waves is a shock.

%******************
\begin{theorem}[Glimm's interaction estimates for a nonclassical Riemann solver]  \label{thm:Glimm_interaction_estimates}
Consider any strictly hyperbolic system \eqref{11} with genuinely nonlinear, linearly degenerate, 
or concave-convex characteristic families,  and assume there exists  a kinetic mapping $\Phif_i$ 
associated with a concave-convex family (indexed by $i$), uniformly satisfying the conditions (H1)--(H4).
Then there exists a sufficiently small $\delta_1$ such that for any states
$u_\ell,u_m, u_r \in \Bone$, connecting simple waves $\boldsymbol{\alpha}, \boldsymbol{\beta} \in \mathbb{R}^N$
and interacting according to Case 1-Case 7, or connecting only only weak waves, we have the estimates
$$
        \boldsymbol{\gamma}  =  \boldsymbol{\alpha} + \boldsymbol{\beta} 
         + \mathcal{O}\bigg( \sum_{\substack{\alpha_k, \beta_\ell \\ \text{approaching}}} |\alpha_k| \cdot |\beta_\ell| \bigg), 
$$
where $\boldsymbol{\gamma}$ is the vector of signed strengths for the waves
in the solution to the Riemann problem for the consecutive states $u_\ell$ and $u_r$.
\end{theorem}

\begin{proof} The literature contains a few proofs of this result for interactions between classical 
waves \cite{Glimm,Smoller,Dafermos}; here we will adapt the proof given in Smoller \cite{Smoller}. The proof
proceeds in two steps. The first one is to demonstrate the simpler quadratic estimate 
\be   \label{eq:weak_Glimm}
         \boldsymbol{\gamma} =  \boldsymbol{\alpha} + \boldsymbol{\beta}  
         + \mathcal{O}\big( \big[ |\boldsymbol{\alpha}| + |\boldsymbol{\beta}| \big]^2 \big). 
\ee
In a second step, one uses the additivity of simple rarefaction waves to demonstrate
inductively that the more accurate estimate holds. A careful analysis of the second part
shows that the argument also applies to nonclassical waves if the weaker
estimate \eqref{eq:weak_Glimm} still holds. Our objective will therefore be to
demontrate the weaker estimate for the new interactions present in a
splitting-merging pattern. There are in fact five new interactions that need to be discussed, 
from among the interactions listed previously,
but the cases 1 and 4 are similar to the cases 5 and 6 respectively. Furthermore,
Case 3 and Case 7 are entirely classical and easy to treat, despite the failure
of \eqref{failure-addition}. For the sake of brevity, 
we will therefore only treat Case 1, 2 and 4.

%******************
\textbf{Case 1.} ($\Cpm \alpha_i \rightarrow N' {\CUp}'$ or $\alpha_i \Cpm  \rightarrow N' {\CUp}'$). 
We consider the first interaction where $\Cpm$ is connecting states $u_\ell$ and $u_m$
while $\alpha_i$ connects $u_m$ and $u_r$. For convenience, we will write
$\delta_C = \mu_i(u_m) - \mu_i(u_\ell)$ and $\delta_\alpha = \mu_i(u_r) - \mu_i(u_m)$. Using the 
expansion~\eqref{27}, we find
\bse
\begin{align}    
u_m = & \, u_{\ell}+ \frac{ \delta_C }{(\nabla \mu_i \cdot r_i)(u_\ell)} r_i(u_{\ell})
	                      +\mathcal{O}\big( \delta_C^2 \big),  \label{um1}  \\
u_r   = & \, u_m + \frac{\delta_\alpha}{(\nabla \mu_i \cdot r_i)(u_m)} r_i(u_m)
                                            +\mathcal{O}\big( \delta_\alpha^2 \big). \label{ur1} 
\end{align}
\ese
Substituting the first expansion for $u_m$ into the second, and accounting for the
dependence on $u_m$ in the second expansion, we find
\begin{align} 
u_r = & \,  \Big( u_{\ell}  +  \frac{\delta_C}{(\nabla \mu_i \cdot r_i)(u_\ell)} r_i(u_{\ell})  \Big)
                     + \frac{\delta_\alpha}{(\nabla \mu_i \cdot r_i)(u_m)} r_i(u_m)
                      + \mathcal{O}\big( \delta_C^2 + \delta_\alpha^2 \big) \nonumber \\
       = & \, u_\ell +  \frac{\delta_C + \delta_\alpha}{(\nabla \mu_i \cdot r_i)(u_\ell)} r_i(u_{\ell})
                     + \mathcal{O}\big( \big[ |\delta_C| + |\delta_\alpha| \big]^2 \big). 
                     \label{ur12}
\end{align}
This expansion used the following two simple identities
\begin{align*}
   \frac{1}{(\nabla \mu_i \cdot r_i)(u_m)} = & \frac{1}{(\nabla \mu_i \cdot r_i)(u_\ell)} + \mathcal{O}(|\delta_C|), 
   \qquad
   \quad 
     r_i(u_m) = r_i(u_\ell) + \mathcal{O}(|\delta_C|).
\end{align*}
The solution to the Riemann problem is described by the vector $\boldsymbol{\gamma}$ of
signed strengths of the waves, leading to the expansion
\be  \label{riemannsolver:Taylor}
    u_r = u_\ell + \sum_k \gamma_k r_k(u_\ell) + \mathcal{O}( |\boldsymbol{\gamma}|^2).
\ee
Because a Taylor's expansion with respect to small parameters is unique, we have that
\begin{align*}
    \gamma_k = & \mathcal{O}\big( \big[ |\delta_C| + |\delta_\alpha|\big]^2  \big) \quad \text{for $k \neq i$;} \qquad
    \quad 
    \gamma_i =   \frac{\delta_C + \delta_\alpha}{(\nabla \mu_i \cdot r_i)(u_\ell)} 
           + \mathcal{O}\big( \big[ |\delta_C| + |\delta_\alpha|\big]^2  \big).
\end{align*}
If $u_\ell'$ and $u_r'$ are the states connected by the $i$-th wave in the Riemann problem, then
\begin{align*}
    | u_\ell - u_\ell'| = &  \mathcal{O} \big(   \big[ |\delta_C| + |\delta_\alpha|\big]^2 \big), 
    \qquad \quad 
       | u_r - u_r'| = \mathcal{O} \big(   \big[ |\delta_C| + |\delta_\alpha|\big]^2 \big).
\end{align*}
The Riemann solver of Proposition \ref{thm:solver_nucleation} tells us the outgoing waves
of the $i$-th family are a nonclassical shock $N'$ followed by a classical shock $\CUp'$.
To check additivity, we compute
\begin{align*}
        \sigma({\Npm}') + \sigma({\CUp}') = &\, \mu_i\big( \Phifo \circ \Phiif (u_{\ell}') \big)
                                                                         - \mu_i(u_{\ell}') 
                               + \mu_i \big( \Phifo(u_r') \big) - \mu_i\big( \Phifo \circ \Phiif (u_{\ell}') \big) \\
                               = & \, \mu_i \big( \Phifo(u_r) \big) - \mu_i(u_{\ell}) 
                                        +   \mathcal{O} \big(   \big[ |\delta_C| + |\delta_\alpha|\big]^2 \big) \\
                               = & \, \mu_i( \Phifo(u_m) ) - \mu_i(u_{\ell})  
                                        + \mu_i \big( \Phifo(u_r) \big) - \mu_i(\Phifo(u_m)) 
                                  +   \mathcal{O} \big(   \big[ |\delta_C| + |\delta_\alpha|\big]^2 \big)  \\
                               = & \, \sigma(\Cpm) + \sigma(\alpha_i) 
                                      +   \mathcal{O} \big(   \big[ |\delta_C| + |\delta_\alpha|\big]^2 \big).              
\end{align*}

%******************
\textbf{Case 2.} ($ \Npm \CUp \rightarrow {\Cpm}'$). 
We consider a nonclassical wave $\Npm$, separated by states $u_\ell$ and $u_m$,
which interacts with a classical wave connecting $u_m$ and $u_r$. For the scalars 
$\delta_N = \mu_i(u_m) - \mu_i(u_\ell)$ and $\delta_C = \mu_i(u_r) - \mu_i(u_m)$, the 
expansion~\eqref{27} provides
\begin{align*}    
u_m = & u_{\ell}+ \frac{ \delta_N }{(\nabla \mu_i \cdot r_i)(u_\ell)} r_i(u_{\ell})
	                      +\mathcal{O}\big( \delta_N^2 \big),   
	                      \qquad \quad 
u_r   = u_m + \frac{\delta_C}{(\nabla \mu_i \cdot r_i)(u_m)} r_i(u_m)
                                            +\mathcal{O}\big( \delta_C^2 \big).  
\end{align*}
Substituting the first expansion into the second and following the procedure established in
the previous interaction, we find
$$
   u_r =  u_\ell +  \frac{\delta_N + \delta_C}{(\nabla \mu_i \cdot r_i)(u_\ell)} r_i(u_{\ell})
                     + \mathcal{O}\big( \big[ |\delta_N| + |\delta_C| \big]^2 \big) .
$$
By comparing the solution $\boldsymbol{\gamma}$ to the Riemann problem between 
$u_\ell$ and $u_r$, given by the expansion \eqref{riemannsolver:Taylor}, to the previous expansion
we deduce that
\begin{align*}
    \gamma_k = & \mathcal{O}\big( \big[ |\delta_N| + |\delta_C| \big]^2 \big), \qquad \text{for $k \neq i$;} 
    \quad
    \qquad 
    \gamma_i = \frac{\delta_C + \delta_\alpha}{(\nabla \mu_i \cdot r_i)(u_\ell)} 
           + \mathcal{O}\big( \big[ |\delta_N| + |\delta_C| \big]^2 \big).
\end{align*}
If the outgoing wave $\Cpm'$ connects the states $u_\ell'$ and $u_r'$, then 
\begin{align*}
    | u_\ell - u_\ell'| = & \,  \mathcal{O} \big(   \big[ |\delta_N| + |\delta_C|\big]^2 \big), \\
       | u_r - u_r'| = & \,  \mathcal{O} \big(   \big[ |\delta_N| + |\delta_C|\big]^2 \big).
\end{align*}
Using these observations, we verify the additivity of the wave strengths
\begin{align*}
      \sigma({\Cpm}')= & \,   \mu_i\big( \Phifo (u_{r}') \big) - \mu_i(u_{\ell}')  \\
                     = & \,  \mu_i\big( \Phifo (u_{r}) \big) - \mu_i(u_{\ell})
                                      +  \mathcal{O} \big(   \big[ |\delta_N| + |\delta_C|\big]^2 \big) \\ 
                     = & \,  \mu_i\big( \Phifo (u_{r}) \big) 
                                - \mu_i\big( \Phifo \circ \Phiis (u_{\ell}) \big) 
                                + \mu_i\big( \Phifo \circ \Phiis (u_{\ell}) \big)
                                 - \mu_i(u_{\ell}) 
                                +  \mathcal{O} \big(   \big[ |\delta_N| + |\delta_C|\big]^2 \big) \\
                     = & \, \sigma(\Npm) + \sigma(\CDown)   
                               + \mathcal{O} \big(   \big[ |\delta_N| + |\delta_C|\big]^2 \big) .           
\end{align*}

%***************************************************************************
\textbf{Case 4.} ($\alpha_i \Npm  \rightarrow {\Npm}' \alpha_i'$). 
 Consider a pair of simple waves of the $i$-th family with $\alpha_i$ approaching
a nonclassical shock $\Npm$ from the left. We suppose that the neighboring states of the waves are, from left to right, 
$u_\ell, u_m$ and $u_r$ and we introduce the scalars $\delta_\alpha = \mu_i(u_m)- \mu_i(u_\ell)$,
$\delta_N = \mu_i(u_r) - \mu_i(u_m)$. With this notation, the expansion \eqref{27} gives us
\begin{align*}    
u_m = & u_{\ell}+ \frac{ \delta_N }{(\nabla \mu_i \cdot r_i)(u_\ell)} r_i(u_{\ell})
	                      +\mathcal{O}\big( \delta_\alpha^2 \big),   
	                      \qquad \quad 
u_r   = u_m + \frac{\delta_C}{(\nabla \mu_i \cdot r_i)(u_m)} r_i(u_m)
                                            +\mathcal{O}\big( \delta_N^2 \big), 
\end{align*}
which can be combined to obtain
$$
   u_r =  u_\ell +  \frac{\delta_\alpha + \delta_N}{(\nabla \mu_i \cdot r_i)(u_\ell)} r_i(u_{\ell})
                     + \mathcal{O}\big( \big[ |\delta_\alpha| + |\delta_N| \big]^2 \big) .
$$
By comparing this expansion with the one given by the unique solution $\boldsymbol{\gamma}$
to the Riemann problem, written previously as \eqref{riemannsolver:Taylor}, we conclude that
\begin{align*}
    | u_\ell - u_\ell'| = &  \mathcal{O} \big(   \big[ |\delta_\alpha| + |\delta_N|\big]^2 \big), 
    \qquad \quad 
       | u_r - u_r'| =   \mathcal{O} \big(   \big[ |\delta_\alpha| + |\delta_N|\big]^2 \big).
\end{align*}
From these estimates, we can verify additivity in the following way.
\begin{align*}
 \sigma({\Npm}') + \sigma(\alpha_i') = & \, \mu_i \big( \Phifo \circ \Phiif(u_{\ell}') \big) - \mu_i(u_{\ell}') 
                                                         + \mu_i( \Phifo (u_r') ) - \mu_i \big( \Phifo \circ \Phiif(u_{\ell}') \big) \\
                                                     = & \, \mu_i( \Phifo (u_r') ) - \mu_i(u_{\ell}') \\
                                                     = & \, \mu_i \big( \Phifo \circ \Phiif (u_m) \big) - \mu_i(u_{\ell}) 
                                                               + \mathcal{O} \big(   \big[ |\delta_\alpha| + |\delta_N|\big]^2 \big)   \\
                                                     = & \, \mu_i (u_m)  - \mu_i(u_{\ell})  
                                                            + \mu_i( \Phifo \circ \Phiif (u_m) ) - \mu_i (u_m) 
                                                            + \mathcal{O} \big(   \big[ |\delta_\alpha| + |\delta_N|\big]^2 \big)  \\
                                                     = & \,  \sigma(\alpha_i) + \sigma({\Npm}) 
                                                             + \mathcal{O} \big(   \big[ |\delta_\alpha| + |\delta_N|\big]^2 \big).                         
\end{align*}
This completes the proof of Case 4.
\end{proof}
 
 We will now describe a series of lemmas which quantify the changes in the functional 
$\myW$ during the interactions presented in the last section. These results
will sometimes use the nomenclature presented in LeFloch \cite{LeFloch-book},
where the $16$ possible nonclassical interactions were given names. For example,
the expression CR-4 was used to identify the $4$-th type of interaction which may occur 
(under certain conditions) between a shock $C$ on the left and a rarefaction $R$ on the right.

%*****************************************************************************************************************************
%*****************************************************************************************************************************
%*****************************************************************************************************************************
%*****************************************************************************************************************************
%*****************************************************************************************************************************
%*****************************************************************************************************************************
\begin{lemma}[Case 1. $\alpha_i \Cpm  \rightarrow N' \CUp'$ or $\Cpm \alpha_i \rightarrow N' {\CUp}'$] \label{case:1}
Suppose a weak wave $\alpha_i$ interacts with $\Cpm$ to generate a nonclassical and classical shock pair
${\Npm}' {\CUp}'$. Three types of interactions may occur. When the wave $\alpha_i$ is slower than $\Cpm$, then
the interaction must be CR-$4$ and $\alpha_i$ must be a rarefaction. If the wave $\alpha_i$ is faster than $\Cpm$,
then the interaction can be either RC-$3$, with a rarefaction $\alpha_i$, or CC-3 with a shock $\alpha_i$.
When $\alpha_i$ is slower than $\Cpm$, then
\begin{align}
     [V_L] = & \,\, + \kappaLless \, \mathcal{O}(|\Cpm| \cdot | \alpha_i|), \nonumber \\
     [V_M] = & \,\, 0, \nonumber \\ 
      [V_R] = & \,\,   - \kappaR \, | \alpha_i| + \kappaRgrt \, \mathcal{O}(|\Cpm| \cdot | \alpha_i |), \nonumber \\
     [ \myW ] = & \,\,  |\alpha_i | \Big( -\kappaR  
                           + \big(  \kappaLless +  \kappaRgrt \big) \mathcal{O}(|\Cpm| ) \Big). \label{lem1-a}
\end{align} 
When $\alpha_i$ is faster than $\Cpm$, then
\begin{align}
     [V_L] = & \,\, - \kappaL \, | \alpha_i| + \kappaLless \, \mathcal{O}( |\Cpm| \cdot | \alpha_i|),  \nonumber  \\
     [V_M] = & \,\, 0,  \nonumber  \\ 
      [V_R] = & \,\,+ \kappaRgrt \, \mathcal{O}(|\Cpm| \cdot | \alpha_i|),  \nonumber  \\
      [ \myW ] =  & \,\,  |\alpha_i | \Big( -\kappaL 
            + \big(   \kappaLless +  \kappaRgrt \big) \mathcal{O}(|\Cpm| ) \Big). \label{lem1-b}
\end{align} 
\end{lemma}

\begin{proof}
Consider the case where the weak wave $\alpha_i$ approaches $\Cpm$ from the right and generates a CR-$4$ interaction.
The interaction may generate weak waves in the $j$-th family, $j \neq i$, of strength $\mathcal{O}( |\Cpm| \cdot | \alpha_i|)$,
but otherwise it is easy to see that 
\begin{align*}
      [V_L] = & \,\, \kappaLless \, \mathcal{O}( |\Cpm| \cdot | \alpha_i|), \\
      [V_M ] = &\,\, 0, \\
      [V_R] = &\,\, - \kappaR \, | \alpha_i| + \kappaRgrt \, \mathcal{O}( |\Cpm| \cdot | \alpha_i|). 
\end{align*} 
We now study simultaneously the RC-3 and CC-3 interactions with a weak wave $\alpha_i$ approaching $\Cpm$ from the left.
\begin{align*}
      [V_L] = & \,\, - \kappaL \, | \alpha_i| + \kappaLless \, \mathcal{O}( |\Cpm| \cdot | \alpha_i|), \\
      [V_M ] = &\,\, 0, \\
      [V_R] = &\,\, \kappaRgrt \, \mathcal{O}( |\Cpm| \cdot | \alpha_i|). 
\end{align*} 
\end{proof}
%*****************************************************************************************************************************
%*****************************************************************************************************************************
%*****************************************************************************************************************************
%*****************************************************************************************************************************
%*****************************************************************************************************************************
%*****************************************************************************************************************************
\begin{lemma}[Case 2. $\Npm \CUp \rightarrow \Cpm$]\label{case:2}
During the strong interaction $ \Npm \CUp \rightarrow \Cpm$, we have
\begin{align}
     [V_L] = & \,\, \kappaLless \cdot | \Npm|  |\CUp|   \mathcal{O}\big(  |v(\Npm) - v(\CUp)| \big), \nonumber \\
     [V_M] = & \,\,    0, \nonumber \\ 
      [V_R] = & \,\, \kappaRgrt  \cdot | \Npm|  |\CUp|   \mathcal{O}\big(  |v(\Npm) - v(\CUp)| \big), \nonumber \\
      [ \myW ] = & \,\,      (  \kappaLless +  \kappaRgrt ) \cdot | \Npm| |\CUp| \mathcal{O}\big(  |v(\Npm) - v(\CUp)| \big).   \label{lem2}      
%      [W ] = & \,\,   | \Npm|  |\CUp|   \mathcal{O}\big(  |v(\Npm) - v(\CUp)| \big), 
%      [ \myW ] = & \,\,   (\kappaW +  \kappaLless +  \kappaRgrt ) \mathcal{O}( | \Npm| \cdot |\CUp| \cdot |v(\Npm) - v(\CUp)| ).
\end{align} 
\end{lemma}

\begin{proof}
Consider the strong $\Npm \CUp \rightarrow \Cpm$ interaction. The strength of the weak outgoing  waves is obviously proportional to
$  | \Npm| \cdot |\CUp|  $ but also to
\begin{align*}
         \big| v(\Npm) - v(\CUp) \big|.
\end{align*}
These weak waves only contribute to $V_L$ and $V_R$ with $V_M = 0$ before and after the $ \Npm \CUp$ interaction.  
%We now compute the variation in $W$.
%\begin{align*}
%      [W] = & \,\, \big( \mui(u_\ell') - \mubo(u_r')  \big)    -  \big(  \mui(u_\ell) - \mubb(u_\ell) \big) - \big( \mubo(u_m) - \mubo(u_r)  \big)  \\
%              = & \,\,  \big( \mui(u_\ell') - \mui(u_\ell)  \big) 
%                     - \big( \mubo(u_r')  - \mubo(u_r) \big) 
%              =  | \Npm|  |\CUp| \cdot \mathcal{O} \big( |v(\Npm) - v(\CUp)|  \big),
%\end{align*}
%where the dependence on $ | \Npm| \cdot |\CUp| $ also applies since reducing $ | \Npm|$ and $|\CUp|$ means that $u_{\ell} \to u_r$ and
%$u_{\ell}' \to u_r'$.  % Since the waves $\Npm$ and $\CUp$ interact, then the wave $\CUp$ must be slower than $\Npm$. 
\end{proof}
%*****************************************************************************************************************************
%*****************************************************************************************************************************
%*****************************************************************************************************************************
%*****************************************************************************************************************************
%*****************************************************************************************************************************
%*****************************************************************************************************************************
\begin{lemma}[Case 3. $\alpha_i C  \rightarrow {C}'$ or $C \alpha_i \rightarrow {C}'$ ]\label{case:3}
Consider a weak wave $\alpha_i$ of the $i$-th family. The strong shock $C$ can be either a decreasing shock $\Cpm$
or the increasing shock $\CUp$ that appears when a nonclassical shock is present.
If the weak wave crosses the classical strong shock $\Cpm$  from the left, then
\begin{align}
     [V_L] = & \,\, - \kappaL \, | \alpha_i | + \kappaLless \, \mathcal{O}(|\Cpm| \cdot | \alpha_i|),  \nonumber \\
     [V_M] = & \,\,  0,  \nonumber \\ 
      [V_R] = & \,\,  + \kappaRgrt \, \mathcal{O}( |\Cpm| \cdot | \alpha_i|),  \nonumber \\
      [ \myW ] = & \,\,  |\alpha_i | \Big( -\kappaL 
                                    + \big(  \kappaLless +  \kappaRgrt \big) \mathcal{O}(|\Cpm|) \Big).   \label{lem3-a}
%      [W ] = & \,\, - \operatorname{sign}(\alpha_i) | \alpha_i | + \mathcal{O}( |\Cpm| \cdot | \alpha_i|), 
%      [ \myW ] = & \,\,  |\alpha_i | \big( -\kappaL - \operatorname{sign}(\alpha_i) \kappaW 
%                                    + (\kappaW +  \kappaLless +  \kappaRgrt ) \mathcal{O}(W) \Big).
\end{align} 
If the weak wave crosses the shock $\Cpm$ from the right, then
\begin{align}
     [V_L] = & \,\, + \kappaLless \, \mathcal{O}(|\Cpm| \cdot | \alpha_i|),  \nonumber \\
     [V_M] = & \,\, 0,  \nonumber \\ 
      [V_R] = & \,\, -  \kappaR \, | \alpha_i | + \kappaRgrt  \, \mathcal{O}(|\Cpm| \cdot | \alpha_i|),  \nonumber \\
      [ \myW ] =  & \,\,  |\alpha_i | \Big( -\kappaR 
               + \big(  \kappaLless +  \kappaRgrt \big) \mathcal{O}(|\Cpm|) \Big).      \label{lem3-b}
%      [W ] = & \,\, - \operatorname{sign}(\alpha_i) | \alpha_i | + \mathcal{O}(|\Cpm| \cdot | \alpha_i|), 
%      [ \myW ] =  & \,\,  |\alpha_i | \big( -\kappaR - \operatorname{sign}(\alpha_i) \kappaW 
%               + (\kappaW +  \kappaLless +  \kappaRgrt ) \mathcal{O}(W) \Big).
\end{align}
If the weak wave crosses the classical strong shock $\CUp$  from the left, then
\begin{align}
     [V_L] = & \,\, 0,   \nonumber \\
     [V_M] = & \,\, - \kappaM \, | \alpha_i | + \kappaMless \, \mathcal{O}(|\CUp| \cdot | \alpha_i|),  \nonumber \\ 
      [V_R] = & \,\, + \kappaRgrt \, \mathcal{O}(|\CUp| \cdot | \alpha_i|),  \nonumber \\
      [ \myW ] = & \,\,  |\alpha_i | \Big( -\kappaM  
                                    + \big(  \kappaMless +  \kappaRgrt \big) \mathcal{O}(|\CUp|) \Big). \label{lem3-c}
%      [W ] = & \,\, - \operatorname{sign}(\alpha_i) | \alpha_i | + \mathcal{O}(|\CUp| \cdot | \alpha_i|), 
%      [ \myW ] = & \,\,  |\alpha_i | \big( -\kappaL - \operatorname{sign}(\alpha_i) \kappaW 
%                                    + (\kappaW +  \kappaLless +  \kappaRgrt ) \mathcal{O}(W) \Big).
\end{align} 
If the weak wave crosses the shock $\CUp$ from the right, then
\begin{align}
     [V_L] = & \,\, 0,  \nonumber \\
     [V_M] = & \,\, + \kappaMless \, \mathcal{O}(|\CUp| \cdot | \alpha_i|),   \nonumber \\
      [V_R] = & \,\, -  \kappaR \, | \alpha_i | + \kappaRgrt \,  \mathcal{O}(|\CUp| \cdot | \alpha_i|),  \nonumber \\
      [ \myW ] =  & \,\,  |\alpha_i | \Big( -\kappaR 
               + \big(  \kappaMless +  \kappaRgrt \big) \mathcal{O}(|\CUp|) \Big).  \label{lem3-d}
%      [W ] = & \,\, - \operatorname{sign}(\alpha_i) | \alpha_i | + \mathcal{O}(|\CUp| \cdot | \alpha_i|), 
%      [ \myW ] =  & \,\,  |\alpha_i | \big( -\kappaR - \operatorname{sign}(\alpha_i) \kappaW 
%               + (\kappaW +  \kappaLless +  \kappaRgrt ) \mathcal{O}(W) \Big).
\end{align}
\end{lemma}

\begin{proof} 
The proof is relatively straightforward and we treat only the first two cases. When the wave $\alpha_i$ approaches from the left, then 
\begin{align*}
     [V_L] = & \,\, - \kappaL | \alpha_i | + \kappaLless  \mathcal{O}(|\Cpm| \cdot | \alpha_i|),  \\
     [V_M] = & \,\,  0, \\ 
      [V_R] = & \,\,   \kappaRgrt \mathcal{O}(|\Cpm| \cdot  | \alpha_i|).
\end{align*}
\end{proof}
%*****************************************************************************************************************************
%*****************************************************************************************************************************
%*****************************************************************************************************************************
%*****************************************************************************************************************************
%*****************************************************************************************************************************
%*****************************************************************************************************************************
\begin{lemma}[Case 4. $\alpha_i \Npm  \rightarrow {\Npm}' \alpha_i'$]\label{case:4}
Suppose a weak wave $\alpha_i$ of the $i$-th family approaches $\Npm$ from the left 
and that the nonclassical shock ${\Npm}'$ remains after the interaction. 
Then the only two possible interactions are RN and CN-$3$, and in both cases, we have
\begin{align*}
         [V_L] = & \,\, - \kappaL \, | \alpha_i|  + \kappaLless \, \mathcal{O}(|\Npm| \cdot  | \alpha_i|), \\
         [V_M] = & \,\,  + \kappaM \,  | \alpha_i'|  + \big( \kappaM + \kappaMgrt \big) \, \mathcal{O}(|\Npm| \cdot | \alpha_i|), \\
         [V_R] = & \,\, 0, 
%         [W] = & \,\, \operatorname{sign}(\alpha_i) \big(  | \alpha_i|  -  | \alpha_i'| \big)+ \mathcal{O}(|\Npm| \cdot | \alpha_i|),
\end{align*}
With the constant $\Cff < 1$ provided by property (H4) of $\Phif_i$, we have
\begin{align}\label{lem4}
      [ \myW ] \leq & \,\, |\alpha_i| \Big( -   \kappaL
                                                      + \Cff    \kappaM
                                +  \big(  \kappaM + \kappaLless +  \kappaMgrt  \big) \mathcal{O}(|\Npm|)  \Big).
\end{align}
\end{lemma}

\begin{proof} 
We will consider the changes in the total variation with respect to the Definition \ref{defn:waveLL} of wave strength. When the
interaction occurs it generates secondary waves in the families $j \neq i$ of strength
$    \mathcal{O}(|\Npm| \cdot | \alpha_i|)$,
and there is an outgoing wave $\alpha_i'$ to the right of ${\Npm}'$ and of the same type as $\alpha_i$.
%where $\epsilon$ is an upper bound on $\sigma(\alpha_i)$, say the variation of the perturbation, and $\delta_N = \sigma(\Npm).$
This implies that 
\begin{align*}
    | u_\ell - u_\ell' | =   \mathcal{O}(|\Npm| \cdot | \alpha_i|), \quad
    | u_r - u_r' | =  \mathcal{O}(|\Npm| \cdot | \alpha_i|), \quad
    | \Phi^\flat_i(u_m) - u_m' | =   \mathcal{O}(|\Npm| \cdot | \alpha_i|). 
\end{align*}
Recall that only RN and CN-$3$ interactions are admissible and so $\alpha_i$ could be a rarefaction or a shock. 
We always have $[V_R] = 0$ but the calculation of the two other variations is less obvious.
\begin{align*}
         [V_L]   = & \,\, - \kappaL  | \alpha_i | + \kappaLless \mathcal{O}(|\Npm| \cdot | \alpha_i|)   \\                                      
         [V_M]  = & \, \, + \kappaM | \alpha_i' | + \kappaMgrt \mathcal{O}(|\Npm| \cdot | \alpha_i|)  \\
                    = & \,\, + \kappaM \big| \mui( \Phifo( u_r') ) - \mui( \Phifo(u_m') ) \big| 
                                + \kappaMgrt \mathcal{O}(|\Npm| \cdot | \alpha_i|)\\
                    = & \,\, + \kappaM \big| \mui( \Phifo \circ \Phif_i ( u_m) ) - \mui( \Phifo \circ \Phif_i(u_{\ell}) ) \big| 
                         + (\kappaM + \kappaMgrt \big) \mathcal{O}(|\Npm| \cdot | \alpha_i|)\\          
                    = & \,\,  +  \kappaM \Cff \big| \mui( u_m ) - \mui( u_{\ell} ) \big| 
                         + (\kappaM + \kappaMgrt \big) \mathcal{O}(|\Npm| \cdot | \alpha_i|)    
\end{align*}
where we have used $C_i$ of condition (H4).
\end{proof}
%*****************************************************************************************************************************
%*****************************************************************************************************************************
%*****************************************************************************************************************************
%*****************************************************************************************************************************
%*****************************************************************************************************************************
%*****************************************************************************************************************************
\begin{lemma}[Case 5. $\beta_j \Cpm  \rightarrow {\Npm}' {\CUp}'$ or $\Cpm \beta_j \rightarrow {\Npm}' {\CUp}'$ ]\label{case:6}
Consider a weak wave $\beta_j$, with $j \neq i$, which after interacting with $\Cpm$ leads to 
 $\Npm \CUp$. When the weak wave is faster than $\Cpm$, that is $j>i$, then
\begin{align}
     [V_L] = & \,\, - \kappaLgrt \, |\beta_j| + \kappaLless \, \mathcal{O}(|\Cpm| \cdot |\beta_j |), \nonumber \\
     [V_M] = & \,\,     0, \nonumber \\ 
      [V_R] = & \,\, + \kappaRgrt \, |\beta_j|  + \kappaRgrt \,  \mathcal{O}(|\Cpm| \cdot | \beta_j |), \nonumber \\
    [ \myW ] = & \,\,   |\beta_j  | \Big( - \kappaLgrt  +  \kappaRgrt  
                      +  \big(  \kappaLless +  \kappaRgrt \big) \mathcal{O}(|\Cpm|) \Big).  \label{lem6-a}
%      [W ] = & \,\,   \mathcal{O}(W | \beta_j |).
%      [ \myW ] = & \,\,  |\beta_j  | \big( -\kappaLgrt + \kappaRgrt 
%                            +  (\kappaW +  \kappaLless +  \kappaRgrt ) \mathcal{O}(W) \Big).
\end{align} 
On the other hand, when $j<i$ we have
\begin{align}
     [V_L] = & \,\, + \kappaLless \, |\beta_j|  + \kappaLless \, \mathcal{O}(|\Cpm| \cdot | \beta_j |), \nonumber \\
     [V_M] = & \,\,     0, \nonumber \\ 
      [V_R] = & \,\, - \kappaRless \, |\beta_j|  + \kappaRgrt \, \mathcal{O}(|\Cpm| \cdot | \beta_j |), \nonumber \\
  [ \myW ] = & \,\,  |\beta_j  | \Big(  -  \kappaRless   +  \kappaLless    
                      +  \big(  \kappaLless +  \kappaRgrt \big) \mathcal{O}(|\Cpm|) \Big).      \label{lem6-b}
%      [W ] = & \,\,    \mathcal{O}(W | \beta_j |), 
%      [ \myW ] = & \,\,  |\beta_j  | \big( \kappaLless - \kappaRless +  (\kappaW +  \kappaLless +  \kappaRgrt ) \mathcal{O}(W) \Big).
\end{align} 
\end{lemma}

\begin{proof}
We begin with the case where $\beta_j$ is faster than $\Cpm$. The weak wave $\beta_j$ belongs to
the family $j$, hence up to a quadratic error, the strength of
the outgoing wave $j$-wave is the same strength as that of the incoming wave $\beta_j$. 
We begin by computing the change when the wave $\beta_j$ belongs to a family $j>i$. 
\begin{align*}
     [V_L] = & \,\,  - \kappaLgrt |\beta_j|  + \kappaLless  \mathcal{O}(|\Cpm| \cdot | \beta_j |) \\
     [V_M] = & \,\,     0  \\ 
      [V_R] = & \,\, + \kappaRgrt |\beta_j|  + \kappaRgrt  \mathcal{O}(|\Cpm| \cdot | \beta_j |)
\end{align*}

In the second case, that is when $j < i$, then the wave $\beta_j$ approaches from the right.
We now remark that
\begin{align*}
     [V_L] = & \,\, + \kappaLless |\beta_j|  + \kappaLless  \mathcal{O}(|\Cpm| \cdot | \beta_j |), \\
     [V_M] = & \,\,     0, \\ 
      [V_R] = & \,\, - \kappaRless |\beta_j|  + \kappaRgrt \mathcal{O}(|\Cpm| \cdot | \beta_j |).
\end{align*}
\end{proof}
%*****************************************************************************************************************************
%*****************************************************************************************************************************
%*****************************************************************************************************************************
%*****************************************************************************************************************************
%*****************************************************************************************************************************
%*****************************************************************************************************************************
\begin{lemma}[Case 6. $\beta_j \Npm \rightarrow \Npm' \beta_j'$ or $\Npm \beta_j \rightarrow  \beta_j' \Npm' $]\label{case:7}
Suppose a weak wave $\beta_j$ from the $j$-th family, $j>i$, crosses $\Npm$ from the left
and generates a small wave $\beta_i'$ in the $i$-th family, then
\begin{align}
     [V_L] = & \,\, - \kappaLgrt \, | \beta_j | + \kappaLless \, \mathcal{O}( | \Npm| \cdot |\beta_j |  ), \nonumber \\
     [V_M] = & \,\, +  \kappaMgrt \, | \beta_j | +  (\kappaM+ \kappaMgrt) \, \mathcal{O}(| \Npm| \cdot |\beta_j | ),  \nonumber  \\ 
     [V_R] = & \,\, 0,  \nonumber  \\
      [ \myW ] = & \,\,  |\beta_j| \big(  - \kappaLgrt + \kappaMgrt + (  \kappaM + \kappaLless +  \kappaMgrt ) \mathcal{O}(| \Npm| ) \big).     \label{lem7-a}
%     [W ] = & \,\, \mathcal{O}(| \Npm| \cdot |\beta_j | ), 
%      [ \myW ] = & \,\,  |\beta_j| \big(  - \kappaLgrt + \kappaMgrt + (\kappaW +  \kappaLless +  \kappaMgrt ) \mathcal{O}(| \Npm| ) \big).
\end{align} 
Suppose a weak wave $\beta_j$  from the $j$-th family, $j < i$, crosses $\Npm$ from the right. Then
\begin{align}
     [V_L] = & \,\, + \kappaLless \, | \beta_j | + \kappaLless \, \mathcal{O}( | \Npm| \cdot |\beta_j |),  \nonumber  \\
     [V_M] = & \,\, - \kappaMless \, | \beta_j | + (\kappaM+  \kappaMgrt) \, \mathcal{O}(| \Npm| \cdot |\beta_j |),  \nonumber  \\ 
     [V_R] = & \,\, 0,  \nonumber  \\
    [ \myW ] = & \,\,  |\beta_j| \big(   - \kappaMless + \kappaLless + (  \kappaM + \kappaLless +  \kappaMgrt ) \mathcal{O}(| \Npm| ) \big).  \label{lem7-b}   
%     [W ] = & \,\, \mathcal{O}(| \Npm| \cdot |\beta_j | ), 
%      [ \myW ] = & \,\,  |\beta_j| \big(  + \kappaLless - \kappaMless + (\kappaW +  \kappaLless +  \kappaMgrt ) \mathcal{O}(| \Npm| ) \big).
\end{align} 
\end{lemma}
%*****************************************************************************************************************************
%*****************************************************************************************************************************
%*****************************************************************************************************************************
%*****************************************************************************************************************************
%*****************************************************************************************************************************
%*****************************************************************************************************************************
\begin{lemma}[Case 7. $\beta_j C \rightarrow C' \beta_j'$ or $C \beta_j \rightarrow \beta_j' C' $]        \label{case:8}
Consider a weak wave $\beta_j$ belonging to the family $j>i$. The strong shock $C$
can be either decreasing $\Cpm$ or increasing $\CUp$. Suppose that $\beta_j$ crosses $\Cpm$ from the left. Then we have the estimates 
\begin{align}
     [V_L] = & \,\, - \kappaLgrt \, | \beta_j | + \kappaLless \, \mathcal{O}( | \Cpm| \cdot |\beta_j |  ), \nonumber \\
     [V_M ] = & \,\, 0,\nonumber   \\
     [V_R] = & \,\, + \kappaRgrt \, | \beta_j | +  \kappaRgrt \, \mathcal{O}(| \Cpm| \cdot |\beta_j | ),  \nonumber  \\
    [ \myW ] = & \,\,  |\beta_j| \big(  -  \kappaLgrt + \kappaRgrt + (  \kappaLless +  \kappaRgrt ) \mathcal{O}(| \Cpm| ) \big).     \label{lem8-a} 
%     [W ] = & \,\, \mathcal{O}(| \Cpm| \cdot |\beta_j | ), 
%      [ \myW ] = & \,\,  |\beta_j| \big(  -  \kappaLgrt + \kappaRgrt + (\kappaW +  \kappaLless +  \kappaRgrt ) \mathcal{O}(| \Cpm| ) \big). 
\end{align} 
Suppose a weak wave $\beta_j$  from the $j$-th family, $j < i$, crosses $\Cpm$ from the right. Then
\begin{align}
     [V_L] = & \,\, + \kappaLless \, | \beta_j | + \kappaLless \, \mathcal{O}( | \Cpm| \cdot |\beta_j |  ), \nonumber  \\
     [V_M] = & \,\, 0, \nonumber  \\ 
     [V_R] = & \,\, - \kappaRless \, | \beta_j | +  \kappaRgrt \, \mathcal{O}(| \Cpm| \cdot |\beta_j | ), \nonumber  \\
[ \myW ] = & \,\,  |\beta_j| \big(  - \kappaRless + \kappaLless  + (  \kappaLless +  \kappaRgrt ) \mathcal{O}(| \Cpm| ) \big).    \label{lem8-b} 
%     [W ] = & \,\, \mathcal{O}(| \Cpm| \cdot |\beta_j | ), 
%      [ \myW ] = & \,\,  |\beta_j| \big(  + \kappaLless - \kappaRless + (\kappaW +  \kappaLless +  \kappaRgrt ) \mathcal{O}(| \Cpm| ) \big).
\end{align} 
Suppose that $\beta_j$, with $j > i$, crosses $\CUp$ from the left, then we have the estimates
\begin{align}
     [V_L] = & \,\, 0, \nonumber  \\
     [V_M] = & \,\, - \kappaMgrt \, | \beta_j | + \kappaMless \, \mathcal{O}( | \CUp| \cdot|\beta_j |  ), \nonumber  \\
     [V_R] = & \,\, + \kappaRgrt \, | \beta_j | +  \kappaRgrt \, \mathcal{O}(| \CUp| \cdot |\beta_j | ), \nonumber  \\ 
    [ \myW ] = & \,\,  |\beta_j| \big(  -  \kappaMgrt + \kappaRgrt + (  \kappaMless +  \kappaRgrt ) \mathcal{O}(| \CUp| ) \big).   \label{lem9-a}
%     [W ] = & \,\, \mathcal{O}(| \CUp| \cdot |\beta_j | ), 
%      [ \myW ] = & \,\,  |\beta_j| \big(  -  \kappaMgrt + \kappaRgrt + (\kappaW +  \kappaMless +  \kappaRgrt ) \mathcal{O}(| \CUp| ) \big). 
\end{align} 
Consider a weak wave $\beta_j$ belonging to the family $j<i$ and crossing $\CUp$ from the right.
Then the interactions are such that
\begin{align}
     [V_L] = & \,\, 0, \nonumber  \\
     [V_M] = & \,\, + \kappaMless \, | \beta_j | + \kappaMless \, \mathcal{O}( | \CUp| \cdot |\beta_j |  ), \nonumber  \\
     [V_R] = & \,\, - \kappaRless \, | \beta_j | +  \kappaRgrt \, \mathcal{O}(| \CUp| \cdot |\beta_j | ),  \nonumber  \\
      [ \myW ] = & \,\,  |\beta_j| \big( - \kappaRless + \kappaMless  + (  \kappaMless +  \kappaRgrt ) \mathcal{O}(| \CUp| ) \big).  \label{lem9-b}
%     [W ] = & \,\, \mathcal{O}(| \CUp| \cdot |\beta_j | ), 
%      [ \myW ] = & \,\,  |\beta_j| \big(  + \kappaMless - \kappaRless + (\kappaW +  \kappaMless +  \kappaRgrt ) \mathcal{O}(| \CUp| ) \big).
\end{align} 
\end{lemma}
%*****************************************************************************************************************************
%*****************************************************************************************************************************
%*****************************************************************************************************************************
%*****************************************************************************************************************************
%*****************************************************************************************************************************
%*****************************************************************************************************************************
\subsection{Interaction potential}
\label{sec:quad}

We propose to study the following quadratic interaction potential. It will involve the usual product of the strength of the waves weighted 
by the the positive part of the  difference in  wave speeds. Weighted functionals were first introduced and studied by Liu \cite{Liu-monograph} and later by Iguchi and LeFloch \cite{IL}, as well as Bianchini  and Bressan \cite{BianchiniBressan}.   
As is usually the case, the interaction potential  is defined only for front-tracking
approximations, with waves localized at point $x$, in which case, we may speak of its
strength $\sigma_x$ and of its family $i_x$,
 \begin{align}\label{decompo} 
           Q\big( u(t) \big) 
               :=  \sum_{ \substack{\text{waves at $x<y$ }  \\ \text{approaching } \\ \text{family $i_x\neq i_y \neq i$ }  } } \big| \sigma_x \cdot \sigma_y \big|
                + \sum_{ \substack{ \text{weak approaching waves $x<y$} \\  \text{$i_x = i$ and/or} \\ \text{$i_y=i$}} } |v_x-v_y|^+ \cdot \big| \sigma_x \cdot \sigma_y \big|. 
\end{align}
The second term includes potential interactions involving both weak and strong i-waves.

This section will require a measure of the strength of the perturbation as a function of time, but this
is not immediately an obvious quantity to define since the strong waves travel in time. 
We define the strength of the perturbation at some time $t$ to be 
\begin{equation}     \label{defn:pertubation_wrt_t}
            \epsilon(t) := \sum_{\substack{\text{weak waves $\beta$} \\ \text{in $u(\cdot,t)$}}} |\beta|.
\end{equation}
Given the position of the waves $x^h(t), y^h(t)$ at some time $t$, we define the unperturbed waves as
$$
     \bar{u} (x,t) = \begin{cases}
                                  u_*,            & \text{ $x < x^h(t)$,} \\
                                  \Phifo(u_*), & \text{ $x^h(t) \leq x < y^h(t)$,} \\
                                  u_*^n,        & \text{ $y^h(t) \leq x$.}
                           \end{cases}      
$$
It is tempting to state that the strength \eqref{defn:pertubation_wrt_t} is equivalent to 
$$
     \epsilon(t) = \operatorname{TV}\big( u(\cdot, t) - \bar{u}(\cdot, t) \big),
$$
but if $u(\cdot, t) - \bar{u}(\cdot, t)$ takes non-zero values at $x^h(t)$ or $y^h(t)$, then the 
total variation of the difference may include the strength of the small perturbation of the strong waves
away from the classical-nonclassical three state pattern. Throughout the analysis below, 
we will study only the strength of the weak waves. If it is proven that $\epsilon(t)$ remains small for the weak waves,
then the strong waves will remain small perturbations of the strong waves $\bar{u}$.

First, we will study here interactions involving waves from the $i$-th family.
We distinguish between several cases of interactions. For the functional 
$Q$, we extend the decomposition \eqref{decompo} to also treat the change in $Q$.
In fact, in certain cases $[Q]$ will not be negative but $[\myW]$ will contribute
negative first order terms to control $[Q]$, and so it is important to accurately measure the change
in $Q$. We will introduce a decomposition of the terms appearing in the variation of $Q$ during interactions.
The analysis will study each term separately before accounting for the total change.

If an interaction occurs among weak waves of families different than $i$, the concave-convex family, then 
the change in $Q$ is quadratic in the strength of the waves interacting and the usual argument
of Glimm applies and shows that for some small total variation $\epsilon$ and a large constant $K$, then
$\myW + K Q$ is decreasing. Henceforth, we will focus exclusively on the second sum in
the interaction potential $Q$.

If an interaction occurs which involves at least one wave from the $i$-th family, then the change
in $Q$ involves (1) quadratic terms in the strength of the interacting waves weighted by their
difference in wave speeds and (2), the sum of the changes in terms $|v_x - v_y|^+ \sigma_x \sigma_y$
where one of them was involved in the interaction but the other wave was not.
More concretely, consider two incoming waves $C_{\ell}, C_r$, respectively the left and right incoming waves,
which generate $C_{\ell}', C_r'$, respectively the left and right outgoing waves.
Then, abusing our notation, we write
\begin{align*}
   \big[ Q(u(\cdot,t)) \big] = [Q_1] + [Q_2],
\end{align*}
where
\begin{align*}
    [Q_1] =  \,\,  \big| v(C_{\ell}') - v(C_{r}') \big|^+ |C_{\ell}'| |C_r'| -  \big| v(C_{\ell}) - v(C_{r}) \big|^+ |C_{\ell}| |C_r|,
\end{align*}
and
\begin{align*}    
    [Q_2] = & \,\, \sum_{ \substack{\text{weak  $\beta_j$ at $y$ }  \\ j \geq i \text{ and } y < x   } } 
                  \Big\{     \big| v(\beta_j) - v(C_{\ell}') \big|^+ |C_{\ell}'| |\beta_j|
                                             + \big| v(\beta_j) - v(C_{r}') \big|^+ |C_{r}'| |\beta_j| \\
                          & \, \phantom{\sum_{ \substack{\text{weak  $\beta_j$ at $y$ }  \\ j \geq i \text{ and } y < x   } } }   
                                             - \big| v(\beta_j) - v(C_{\ell}) \big|^+ |C_{\ell}| |\beta_j|
                                             - \big| v(\beta_j) - v(C_{r}) \big|^+ |C_{r}|  |\beta_j| \Big\} \\
            & \,\, + \sum_{ \substack{\text{weak  $\beta_j$ at $y$ }  \\ j \leq i \text{ and } x < y   } } 
                  \Big\{     \big| v(C_{\ell}')- v(\beta_j)  \big|^+ |C_{\ell}'| |\beta_j|
                                             + \big| v(C_{r}')  - v(\beta_j) \big|^+ |C_{r}'| |\beta_j| \\
                          & \, \phantom{\sum_{ \substack{\text{weak  $\beta_j$ at $y$ }  \\ j \geq i \text{ and } y < x   } } }   
                                             - \big|  v(C_{\ell}) - v(\beta_j)  \big|^+ |C_{\ell}| |\beta_j|
                                             - \big|  v(C_{r})  - v(\beta_j)\big|^+ |C_{r}|  |\beta_j| \Big\}.                              
\end{align*}
The analysis of $[Q]$ will demonstrate that $[Q_1]$ is always negative and that if the perturbation of the solution $u$
at some time $t$ away from the unperturbed state $\bar{u}$ satisfies an a priori bound with respect to the strength of
the strong waves, then the potential increase in $[Q_2]$ can be compensated by $[Q_1]$ and result in the sum
$[Q]$ being negative by a known negative quantity.

Among the two variations above, there are generalities that can be said about $[ Q_2]$.
In most cases, the variation in $Q_2$ will be caused by the interaction of two incoming waves and
two outgoing waves, and for the sake of argument, we will consider only
weak waves $\beta$ to the left of the interaction. Then for every such weak wave $\beta$, there
is a contribution to $[Q_2]$ of the form
\begin{align*}
        [Q_2]_{\beta} := & \,    |\beta| \Big(  \big| v(\beta) - v(C_{\ell}') \big|^+ |C_{\ell}'| 
                                             + \big| v(\beta) - v(C_{r}') \big|^+ |C_{r}'|  \\
                                  & \, \phantom{|\beta| \Big(}          - \big| v(\beta) - v(C_{\ell}) \big|^+ |C_{\ell}| 
                                             - \big| v(\beta) - v(C_{r}) \big|^+ |C_{r}| \Big).
\end{align*}
As a function of $v(\beta)$, this expression is continuous and piecewise linear with 
four points of discontinuity in the derivative. Clearly, if $v(\beta) < m_v := \min v(C)$
then all the weights vanish and $[Q_2]_{\beta} =0$.
We claim that over the unbounded domain
%$ ]-\infty, m_v]$, with $m_v = \min v(C)$ ( 
$ [M_v, \infty[$, with $M_v = \max v(C)$, 
the variation $[Q_2]_{\beta}$ is decreasing and therefore the maximum value
is taken at $v(\beta)=M_v$. This follows
from the fact that during interactions between waves of the family $i$, 
$$
        |C_{\ell}| + |C_r| - |C_{\ell}'| - |C_r'| \leq 0, 
$$
with equality only when all waves are shocks. Each weight can be written as
$$
    \big| v(\beta) - v(C) \big|^+ = | v(\beta) - M_v | + | M_v - v(C) | 
$$
and therefore we have
\begin{align*}
             [Q_2]_{\beta} = & \, |\beta| | v(\beta) - M_v |   \Big(    |C_{\ell}'| + |C_r'| - |C_{\ell}| -|C_r|  \Big) \\
                         & \,   +  |\beta|  \Big( 
                              \Delta v_1 |C_{\ell}'| 
                                             +  \Delta v_2 |C_{r}'| 
                                             - \Delta v_3 |C_{\ell}| 
                                             -  \Delta v_4 |C_{r}|  \Big).
\end{align*}
Each of the new weights $\Delta v_i$ is constant with respect to $v(\beta)$ but the first term
is negative. Hence the maximum of 
$[Q_2]_{\beta}$ must occur at one of the three largest speeds where the derivative is
discontinuous. 
When the weak wave $\beta$ is approaching from the right, similar arguments
show that the maximum value will occur when $v(\beta)$ coincides with one
of the the three lowest speeds in the interaction.  
Below, when studying $[Q_2]_{\beta}$ it will be easy to infer
the speed $v(\beta)$ for which the value will be largest.

%*************************************************************************
%**********************  case 1  *****************************************
%*************************************************************************
\begin{lemma}[Case 1. $\alpha_i \Cpm  \rightarrow N \CUp$ or $\Cpm \alpha_i \rightarrow N \CUp$] \label{Qcase:1}
Consider at some time $t$ a weak wave $\alpha_i$ interacting with $\Cpm$ to generate a nonclassical and a  classical shock pair
${\Npm}' {\CUp}'$. If the perturbation $\epsilon(t)$  satisfies an a priori bound 
\be     \label{qcase1:apriori}
          \epsilon(t) \leq \mathcal{O}(|\Cpm|),
\ee
then 
$$
      [Q] \leq - \big| \mathcal{O}(|\alpha_i|\cdot|\Cpm|^2) \big|.
$$      
\end{lemma}
%*****************************************************************************
\begin{proof}
There are three types of interactions that occur, RC-$3$, CC-3, and CR-$4$.
%*****************************************************************************
%************************************   RC - 3
%*****************************************************************************
All wave speeds, except $v(\alpha_i)$, are close.
In all these cases, 
\begin{align}
[Q_1] = - \big| v(\alpha_i) - v(\Cpm) \big| \, |\alpha_i|\cdot|\Cpm|
= - \big| \mathcal{O}(|\alpha_i|\cdot|\Cpm|^2) \big|.  \label{proof124a}
\end{align}
Consider first the case where $\alpha_i$ is a rarefaction and the shock strengths satisfy  
$$ 
          |\Npm'| + |\CDown'| = |\CDown| - |\alpha_i|+\mathcal{O}(|\CDown|\cdot|\alpha_i|),
$$
which implies   
$$ 
                     |\Npm'| + |\CDown'| < |\CDown|+\mathcal{O}(|\CDown|\cdot|\alpha_i|). 
$$
An inspection of the normalized wave speeds associated to the 
states $u_l$ and $u_r$ shows that
$$ 
       v(\Npm') < \min\big[ v(\Cpm),v(\CUp')\big] < v(\alpha_i), 
$$
and that $v(\Npm') - v(\Cpm)=\mathcal{O}(|\alpha_i|)$.
We will now analyze the interaction term $[Q_2]$ in \eqref{decompo}. 
Suppose there is a weak wave $\beta$ to the left of the interaction, then
\begin{align*}
  [Q_2]_{\beta} = &|\beta| \Big( \big|v(\beta) - v(\Npm') \big|^+ |\Npm'| +\big|v(\beta) - v(\CUp') \big|^+ |\CUp'| \\
            & \phantom{|\beta| \Big(} - \big|v(\beta) - v(\alpha_i) \big|^+ |\alpha_i|-\big|v(\beta) - v(\Cpm) \big|^+ |\Cpm|\Big).
\end{align*}
If $v(\beta) < v(\Npm')$, then $[Q_2]_{\beta}<0$. The largest value of this function
of $v(\beta)$ occurs when $v(\Npm')<v(\CUp')< v(\beta) =v(\Cpm)<v(\alpha_i)$, but in that case 
\begin{align}
         [Q_2]_{\beta} \leq & |\beta| \Big( \big|v(\Cpm) - v(\Npm') \big| |\Npm'| +\big|v(\Cpm) - v(\CUp') \big| |\CUp'| 
                              - \big|v(\Cpm) - v(\alpha_i) \big| |\alpha_i|\Big) \nonumber \\
                   \leq &|\beta|\Big(\mathcal{O}(|\alpha_i|)|\Npm'| + \mathcal{O}(|\alpha_i|)|\CUp'|\Big)
                   \leq \mathcal{O}(|\beta||\alpha_i||\Cpm|)      \label{proof124b}.
\end{align}
Still in  the case RC-$3$, suppose now there is a weak wave $\beta$ to the right of the interaction, then
\begin{align*}
  [Q_2]_{\beta} = &|\beta| \Big( \big|v(\Npm')- v(\beta)  \big|^+ |\Npm'| +\big|v(\CUp') - v(\beta) \big|^+ |\CUp'| \\
         & \phantom{|\beta| \Big(}- \big|v(\alpha_i)- v(\beta) \big|^+ |\alpha_i|-\big|v|(\Cpm) - v(\beta) \big|^+ |\Cpm|\Big).
\end{align*}
The variation is largest when  $v(\Npm')< v(\Cpm)=v(\beta) < v(\CUp')<v(\alpha_i)$, in which case
\begin{align}
  [Q_2]_{\beta}    \leq & |\beta| \Big( \big| v(\CUp')-v(\Cpm)\big| |\CUp'| -  \big| v(\alpha_i)-v(\Cpm)\big| |\alpha_i| \Big)\nonumber\\
               \leq & |\beta|\Big(\mathcal{O}(|\alpha_i|)|\Cpm|+\mathcal{O}(|\Cpm|)|\alpha_i| \Big)
               = \mathcal{O}(|\beta||\alpha_i||\Cpm|).\label{proof124c}
\end{align}
Combining \eqref{proof124a} with either \eqref{proof124b} or \eqref{proof124c}, we find 
\begin{align*}
        [Q] = [Q_1] + \sum_{\beta}[Q_2]_{\beta} \leq - \Big| \mathcal{O}(|\alpha_i|\cdot|\Cpm|^2) \Big|
                + \sum_\beta \mathcal{O}(|\beta||\alpha_i||\Cpm|),
\end{align*}
which is negative if a bound of the form $\epsilon(t)<\mathcal{O}(|\Cpm|)$ is satisfied by the perturbation.
%*****************************************************************************
%************************************ CC - 3
%*****************************************************************************
\vskip.15cm 
In the case CC-$3$,  the weak wave $\alpha_i$ can come either from the left or the right of $\Cpm$.
For the sake of brevity, we will consider only left incoming weak waves $\alpha_i$. The speeds satisfy 
\begin{equation}      \label{vspeedcc3}
           v(\Cpm) < v(\Npm') <  v(\CUp') < v(\alpha_i),
\end{equation}
and for the strengths we have 
\begin{equation*}\label{vspeed1}
         |{\Npm}'| + |{\Cm}'| = |\Cpm| - \operatorname{sign}(\alpha_i) |\alpha_i| +\mathcal{O}(|\Cpm|\cdot|\alpha_i|). 
\end{equation*}
Suppose that there is a weak wave $\beta$ to the left of the interaction, then
\begin{align*}
  [Q_2]_{\beta} = &|\beta| \Big( \big|v(\beta) - v(\Npm') \big|^+ |\Npm'| +\big|(v(\beta) - v(\CUp') \big|^+ |\CUp'| \\
         &   \phantom{|\beta| \Big(}- \big|v(\beta) - v(\alpha_i) \big|^+ |\alpha_i|-\big|v(\beta) - v(\Cpm) \big|^+ |\Cpm|\Big).
\end{align*}
The worst case occurs when $ v(\beta)=  v(\alpha_i)$ and the change is 
\begin{align}\label{proof124d}
  [Q_2]_{\beta}  \leq  &  |\beta|\Big(  \big| v(\alpha_i) - v(\Npm') \big|^+ |\Npm'|
                        +\big|v(\alpha_i) - v(\CUp') \big|^+ |\CUp'|-\big|v(\alpha_i) - v(\Cpm) \big|^+ |\Cpm|\Big).
\end{align}
Suppose now that there is a weak wave $\beta$ to the right of the interaction, then the largest value
is found when $v(\beta) = v(\Cpm)$.
\begin{align*}
  [Q_2]_{\beta} =  &   |\beta| \Big( \big| v(\Npm')- v(\Cpm)  \big|^+ |\Npm'| 
                                           +\big|v(\CUp') - v(\Cpm) \big|^+ |\CUp'| - \big|v(\alpha_i)- v(\Cpm) \big|^+ |\alpha_i|\Big) \\
         \leq& |\beta|\Big(\mathcal{O}(|\alpha_i|)|\Npm'| + \mathcal{O}(|\alpha_i|)|\CUp'| - \mathcal{O}(|\Cpm|)|\alpha_i|\Big) 
         = \mathcal{O}(|\beta||\alpha_i||\Cpm|).
\end{align*}
As in the case RC-$3$, imposing an a priori constraint $\sum_{\beta}| \beta | \leq \mathcal{O}(|\Cpm|)$ implies that $[Q]$ is negative.
\vskip.15cm 
%*****************************************************************************
%************************************ CR - 4
%*****************************************************************************
\vskip.15cm 
In the case CR-$4$, the speeds are ordered as 
\begin{equation}\label{vspeed}
 v(\alpha_i) < v(\Npm') <  \min\Big[v(\Cpm),v(\CUp')\Big].
\end{equation}
Suppose there is a weak wave $\beta$ to the left of the interaction, then
\begin{align*}
  [Q_2]_{\beta} = &|\beta| \Big( \big|v(\beta) - v(\Npm') \big|^+ |\Npm'| +\big|v(\beta) - v(\CUp') \big|^+ |\CUp'| \\
         & \phantom{|\beta| \Big(} - \big|v(\beta) - v(\alpha_i) \big|^+ |\alpha_i|-\big|v(\beta) - v(\Cpm) \big|^+ |\Cpm|\Big).
\end{align*}
The worst case corresponds to $ v(\beta)=v(\Cpm)>v(\CUp')$. Then, we estimate
\begin{align*}
  [Q_2]_{\beta} = &|\beta| \Big( \big|v(\Cpm) - v(\Npm') \big| |\Npm'| +\big|v(\Cpm) - v(\CUp') \big|^+ |\CUp'|- \big|v(\Cpm) - v(\alpha_i) \big| |\alpha_i|\Big)\\
          \leq & |\beta|\Big(\mathcal{O}(|\alpha_i|)|\Npm'| + \mathcal{O}(|\alpha_i|)|\CUp'| - \mathcal{O}(|\Cpm|)|\alpha_i|\Big) 
         =  \mathcal{O}(|\beta||\alpha_i||\Cpm|).
\end{align*}
Taking into account \eqref{vspeed1} and an a priori bond $| \beta | \leq \mathcal{O}(|\Cpm|)$, we deduce 
\begin{equation} \label{proof124f}
[Q_2]_{\beta} \leq |\beta|\big( \mathcal{O}(|\Cpm|)|\alpha_i|\big) \leq \mathcal{O}(|\alpha_i||\Cpm|^2).
\end{equation}
Suppose now there is a weak wave $\beta$ to the right of the interaction, then
\begin{align*}
  [Q_2]_{\beta} = &|\beta| \Big( \big| v(\Npm')- v(\beta)  \big|^+ |\Npm'| +\big|v(\CUp') - v(\beta) \big|^+ |\CUp'| \\
         & \phantom{|\beta| \Big(} - \big|v(\alpha_i)- v(\beta) \big|^+ |\alpha_i|-\big|v(\Cpm) - v(\beta) \big|^+ |\Cpm|\Big).
\end{align*}
The two extreme cases  correspond to $v(\beta)=v(\Cpm)$ and $ v(\beta)=v(\alpha_i)$, but we consider here only 
$v(\alpha_i) < v(\Npm')<v(\beta)=v(\Cpm)< v(\CUp')$. If this happens, 
\begin{align}
  [Q_2]_{\beta} = &|\beta| \Big( \big| v(\Npm')- v(\Cpm)  \big| |\Npm'| +\big|v(\CUp') - v(\Cpm) \big| |\CUp'|  
        - \big|v(\alpha_i)- v(\Cpm) \big| |\alpha_i|\Big)\nonumber\\
         & \leq |\beta|\Big(\mathcal{O}(|\alpha_i|)|\Npm'| + \mathcal{O}(|\alpha_i|)|\CUp'| - \mathcal{O}(|\Cpm|)|\alpha_i|\Big)
         \nonumber
         % \\ &
         = \mathcal{O}(|\beta||\alpha_i||\Cpm|). \label{proof124g}
\end{align}
Again, combining  \eqref{proof124a} with an a priori bond  $\epsilon(t) \leq \mathcal{O}(|\Cpm|)$, we deduce that $[Q] \leq 0$.
\end{proof}

%*************************************************************************
%**********************  case 2  *****************************************
%*************************************************************************
\begin{lemma}[Case 2. $ \Npm \CUp \rightarrow {\Cpm}'$]\label{Qcase:2}
If during the strong interaction $ \Npm \CUp \rightarrow {\Cpm}'$ at some time $t$, a bound of the form 
\be      \label{qcase2:apriori}
          \epsilon(t) \leq \mathcal{O}(|\CUp|),
\ee          
is satisfied by the perturbation, then 
$$
      [Q] \leq - \big|\mathcal{O}(\big| v(\Npm) - v(\CUp) \big| \, |\Npm| \cdot |\CUp| ) \big|.
$$   
\end{lemma}
%*****************************************************************************
\begin{proof}
Consider the strong $\Npm \CUp \rightarrow {\Cpm}'$ interaction. We have
\begin{gather*}
      |\Cpm'| =  |\Npm| + |\CUp| + \mathcal{O}(|\Npm| |\CUp| ), \\
      v( \CUp ) < v(\Npm) < v(\Cpm').
\end{gather*}
All wave speeds are close and in fact, for example 
$|v(\Npm) - v(\CUp)| = \mathcal{O}\big( |\Npm| - |\CUp| - |\Cpm'| \big) $,
is of the size of the perturbation $\epsilon(t)$.  Clearly, 
\begin{align}
[Q_1] = & - \, \big| v(\Npm) - v(\CUp) \big| \, |\Npm|\cdot|\CUp|. \label{Q1nc} 
\end{align}
Consider a weak wave $\beta$ located to the left of the interaction. Then for all speeds $v(\beta)$, 
\begin{align*}
  [Q_2]_{\beta} = &|\beta| \Big( \big|v(\beta) - v(\Cpm') \big|^+ |\Cpm'|- \big|v(\beta) - v(\Npm) \big|^+ |\Npm|-\big|v(\beta) - v(\CUp) \big|^+ |\CUp|\Big) \leq 0.
\end{align*}
Consider a weak wave $\beta$ located to the right of the interaction. The variation is 
\begin{align*}
  [Q_2]_{\beta} = &|\beta| \Big( \big| v(\Cpm')- v(\beta)  \big|^+ |\Cpm'| - \big|v(\Npm)- v(\beta) \big|^+ |\Npm|-\big|v(\CUp) - v(\beta) \big|^+ |\CUp|\Big),
\end{align*}
The worst case is $v(\beta)=v(\Npm)$, then we have the obvious overestimate
\begin{align*}
  [Q_2]_{\beta} \leq &|\beta| \big| v(\Cpm')- v(\Npm)  \big| |\Cpm'| =\mathcal{O}(|\beta| \big| v(\Cpm')- v(\Npm)  \big||\Npm|).
\end{align*}
And so, if $\epsilon(t) \leq \mathcal{O}(|\CUp|)$ then the result is demonstrated.
\end{proof}
%========================================================================

%*************************************************************************
%**********************  case 3  *****************************************
%*************************************************************************
\begin{lemma}[Case 3. $\alpha_i \Cpm  \rightarrow {\Cpm}'$ or $\Cpm \alpha_i \rightarrow {\Cpm}'$]\label{Qcase:3}
The interaction of a weak wave $\alpha_i$ of the $i$-th family with a strong wave of the same family $\Cpm$ splits into four distinct cases, depending on whether $\alpha_i$  is a shock or a rarefaction, and whether the wave approaches from the left or the right.
In all of these cases, if a bound of the form 
\be          \label{qcase3:apriori}
          \epsilon(t) \leq \mathcal{O}(|\CDown|) 
\ee 
is satisfied then 
$$
       [Q] \leq - \big| \mathcal{O}(|\alpha_i|\cdot|\Cpm|^2) \big|.
$$
\end{lemma}
%*****************************************************************************
\begin{proof} We remark that the ordering of the velocities are all distinct and 
include all cases with $v(\alpha_i)$ being either the slowest or fastest. In all of these cases only one wave is outgoing, so
\begin{align}
[Q_1] =& - \, \big| v(\alpha_i) - v(\Cpm) \big|^+ \, |\alpha_i|\cdot|\Cpm| 
=  -\mathcal{O}(|\alpha_i|\cdot|\Cpm|^2).\label{Q1case3} 
\end{align}
%-----------------------------------------------------------------
Consider the first interaction $\alpha_i \CDown$-${\CDown}'$ with $\alpha_i$ a shock. We have the identities
\begin{gather*}
          |\CDown'| =  |\alpha_i| + |\CDown|  + \mathcal{O}(|\alpha_i| |\CDown|), \\
         v( \CDown) < v(\CDown') < v(\alpha_i).
\end{gather*}
Suppose a weak wave $\beta$ is located to the left of the interaction, then the largest value of
\begin{align*}
  [Q_2]_{\beta} = &|\beta| \Big( \big|v(\beta) - v(\CDown') \big|^+ |\CDown'| -\big|v(\beta) - v(\CDown) \big|^+ |\CDown|- \big|v(\beta) - v(\alpha_i) \big|^+ |\alpha_i|\Big),
\end{align*}
occurs when $v(\beta) = v(\alpha_i)$, with which we find 
\begin{align*}
  [Q_2]_{\beta} \leq &|\beta| \Big( \big|v(\alpha_i) - v(\CDown') \big| |\CDown'| -\big|v(\alpha_i) - v(\CDown) \big| |\CDown|\Big)\\
              = &|\beta| \Big( \big|v(\alpha_i) - v(\CDown') \big| \big( |\CDown'|-|\CDown| \big) -\big|v(\CDown') - v(\CDown) \big| |\CDown|\Big)
              = \mathcal{O}(|\beta||\alpha_i||\CDown|).
\end{align*}
Suppose $\beta$ is to the right of the interaction.
\begin{align*}
  [Q_2]_{\beta} = &|\beta| \Big( \big| v(\CDown')- v(\beta)  \big|^+ |\CDown'| - \big|v(\CDown)- v(\beta) \big|^+ |\CDown|-\big|v(\alpha_i) - v(\beta) \big|^+ |\alpha_i|\Big) 
\end{align*}
The worst case is $v(\beta)=v(\CDown)$, we can compute
\begin{align*}
  [Q_2]_{\beta} \leq  |\beta| \Big( \big| v(\CDown')- v(\CDown)  \big| |\CDown'|-\big|v(\alpha_i) - v(\CDown) \big| |\alpha_i|\Big) 
                 = \mathcal{O}(|\beta||\alpha_i||\CDown|).
\end{align*}
Again $[Q]$ can be shown to be decreasing if an a priori bound 
$ \sum_{\beta}|\beta | \leq \mathcal{O}(|\CDown|)$ is satisfied.

Consider the second case $\alpha_i \CDown$-${\CDown}'$  with $\alpha_i$ a rarefaction. 
We observed that  the states of interest satisfy 
\begin{gather*}
|{\CDown}| - |\alpha_i|  =  |\CDown'| + \mathcal{O}(|\alpha_i| |\CDown|). \\
      v( \CDown') < v(\CDown) < v(\alpha_i).
\end{gather*}
Suppose $\beta$ is a weak wave to the left of the interaction, then 
\begin{align*}
  [Q_2]_{\beta} \leq &|\beta|  \big| v(\CDown)- v(\CDown')  \big|^+ |\CDown'|  = \mathcal{O}(|\beta||\alpha_i||\CDown|).
\end{align*}
This suffices to prove the result in that case.

%-----------------------------------------------------------------
The third interaction is $\Cpm \alpha_i$-${\Cpm}'$ where $\alpha_i$ is a shock. The following two identities hold
\begin{gather*}
      |\CDown'| =  |\Cpm| + |\alpha_i| + \mathcal{O}(|\alpha_i| |\CDown|), \\
      v( \alpha_i) < v(\Cpm) < v(\Cpm').
\end{gather*}
Suppose $\beta$  is to the left of the interaction, then the variation in $Q_2$ is 
\begin{align*}
  [Q_2]_{\beta} = &|\beta| \Big( \big|v(\beta) - v(\Cpm') \big|^+ |\Cpm'| -\big|v(\beta) - v(\Cpm) \big|^+ |\Cpm|- \big|v(\beta) - v(\alpha_i) \big|^+ |\alpha_i|\Big).
\end{align*}
For $v(\beta)$, $|v(\beta) - v(\Cpm')| < |v(\beta) - v(\Cpm) |$ and $|\Cpm'| < |\Cpm|$ hence this quantity is always negative. 
Suppose now that $\beta$  is to the right of the interaction, then
\begin{align*}
  [Q_2]_{\beta} = &|\beta| \Big( \big|v(\Cpm')- v(\beta) \big|^+ |\Cpm'| -\big| v(\Cpm)- v(\beta) \big|^+ |\Cpm|- \big|v(\alpha_i) - v(\beta) \big|^+ |\alpha_i|\Big),
\end{align*}
and the worst case is $v(\beta)=v(\Cpm)$, for which we find 
\begin{align*}
  [Q_2]_{\beta} = &|\beta|  \big|v(\Cpm')-  v(\Cpm) \big| |\Cpm'|  
           = \mathcal{O}(|\beta||\alpha_i||\Cpm|).
\end{align*}
In the third case, this shows that an a priori bound $\sum_{\beta}|\beta| \leq \mathcal{O}(|\Cpm|)$ is sufficient to show that 
$[Q]$ is decreasing.

%-----------------------------------------------------------------

The fourth and last interaction is $\Cpm \Rm$-${\Cpm}'$ where $\alpha_i$ is a rarefaction.
The states of interest imply  $|\Cpm'| < |\Cpm|$ and that the speeds are ordered 
$$
         v(\alpha_i) < v(\Cpm') < v(\Cpm).  
$$
Recall also that only one wave is outgoing, so we trivially have $[Q_1] = - |v(\Cpm) - v(\alpha_i)| |\alpha_i |\Cpm|$.  Assume now that a weak wave  $\beta$ is located 
to the left of the interaction, then the least favorable case corresponds to the situation where $v(\beta) = v(\Cpm)$, and
\begin{align*}
  [Q_2]_{\beta} \leq &|\beta| \Big( \big|v(\Cpm) - v(\Cpm') \big| |\Cpm'| - \big|v(\Cpm) - v(\alpha_i) \big| |\alpha_i|\Big)=\mathcal{O}(|\beta||\alpha_i||\Cpm|).
\end{align*}
If $\beta$ is to the right, then the worst case is  $v(\beta)=v(\alpha_i)$ and we have
$$
           [Q_2]_{\beta} \leq |\beta| \Big( \big| v(\Cpm') - v(\alpha_i) \big| |\Cpm'| - \big| v(\Cpm) - v(\alpha_i) \big| |\Cpm| \Big) \leq 0.
$$ 
In conclusion, four a priori bounds $\epsilon(t) \leq \mathcal{O}(|\Cpm|)$ imply that $Q$ is decreasing by an amount $\mathcal{O}(|\alpha_i||\Cpm|^2)$.
\end{proof}

%*************************************************************************
%**********************  case 4  *****************************************
%*************************************************************************
\begin{lemma}[Case 4. $\alpha_i \Npm  \rightarrow {\Npm}' \alpha_i'$]\label{Qcase:4}
Consider at some time $t$ an interaction involving a weak wave $\alpha_i$ approaching $ \Npm $ from the left  with outgoing waves $ \Npm' $ and $\alpha_i'$.
If the perturbation satisfies an a priori bound of the form 
\be          \label{qcase4:apriori}
        \epsilon(t) \leq \mathcal{O}(|\Npm|), 
\ee
then one has 
$$
   [Q] \leq -  \big| \mathcal{O}(|\alpha_i||\Npm|^2) \big|.
$$
\end{lemma}
%*****************************************************************************
\begin{proof}
Whether $\alpha_i$ is a shock or a rarefaction, we find 
\begin{align} \label{Q1case4} 
            [Q_1] =  - \, \big| v(\alpha_i) - v(\Npm) \big|^+ \, |\alpha_i|\cdot|\Npm|  =  -\mathcal{O}(|\alpha_i||\Npm|^2).      
\end{align}
%******************************** RN
%
When $\alpha_i$ is a rarefaction, then we have
$$
         |\Npm'| < |\Npm|, \qquad |\alpha_i'| < |\alpha_i|, 
$$
and the inequalities on the speeds 
\begin{align*}
           v(\Npm') <  v(\Npm) < \min \big\{ v(\alpha_i), v(\alpha_i') \big\}.
\end{align*}
Using the observations above, we verify that if a weak wave $\beta$ is located to the left of the interaction then
\begin{align*}  
  [Q_2]_{\beta} =  & |\beta| \Big( \big|v(\beta) - v(\Npm') \big|^+ |\Npm'| +\big|v(\beta)- v(\alpha_i') \big|^+ |\alpha_i'| \\
                 & \phantom{|\beta| \Big(} - \big|v(\beta) - v(\Npm) \big|^+ |\Npm|- \big|v(\beta) - v(\alpha_i) \big|^+ |\alpha_i|  \Big).
\end{align*}
The largest value of $[Q_2]$ occurs when $v(\beta)= v(\Npm)$, in which case
\begin{align*}
  [Q_2]_{\beta} = & |\beta|\big|v(\Npm) - v(\Npm') \big||\Npm'| 
   = \mathcal{O}(|\beta||\alpha_i||\Npm|).
\end{align*}
When a weak wave $\beta$ is located to the right of the interaction, then 
\begin{align*}
  [Q_2]_{\beta} =  &  |\beta| \Big( \big| v(\Npm')- v(\beta)\big|^+ |\Npm'| +\big|v(\alpha_i') - v(\beta) \big|^+ |\alpha_i'|\\
                 & \phantom{|\beta| \Big(} - \big|v(\Npm) - v(\beta) \big|^+ |\Npm|- \big|v(\alpha_i)- v(\beta) \big|^+ |\alpha_i|  \Big).
\end{align*}
We find the largest value by taking $v(\beta) = v(\alpha_i)$
\begin{align*}
  [Q_2]_{\beta} =  &  |\beta| \big|v(\alpha_i') - v(\alpha_i) \big| |\alpha_i'|  = \mathcal{O}(|\beta||\alpha_i||\Npm|).
\end{align*}
Once again, $[Q]$ is decreasing  if all weak waves satisfy a condition $\sum |\beta| \leq \mathcal{O}(|\Npm|)$.

%********************************** CN-3
We now consider the interaction $\alpha_i \Npm \to \Npm' \alpha_i'$  where $\alpha_i$ is a shock.
We saw in the proof of Lemma \ref{case:4} that 
$$ 
          |\Npm| < |\Npm'|, \qquad  |\alpha_i'|< |\alpha_i|.
$$ 
For this interaction, we have 
$$
         v(\Npm) < v(\Npm') < \min \big\{ v(\alpha_i),  v(\alpha_i') \big\}.
$$
For the functional $[Q_2]$, when there is a weak wave $\beta$ located to the left of the interaction, we have
\begin{align*}  
  [Q_2]_{\beta} =     &   |\beta| \Big( \big|v(\beta) - v(\Npm') \big|^+ |\Npm'| +\big|v(\beta)- v(\alpha_i') \big|^+ |\alpha_i'| \\
                    &   \phantom{|\beta| \Big(}  - \big|v(\beta) - v(\Npm) \big|^+ |\Npm|- \big|v(\beta) - v(\alpha_i) \big|^+ |\alpha_i|  \Big).
\end{align*}
The worst case occurs when $v(\alpha_i') < v(\alpha_i) = v(\beta)$, and
\begin{align*}  
  [Q_2]_{\beta} = &   |\beta| \Big( \big|v(\alpha_i) - v(\Npm') \big| |\Npm'| +\big|  v(\alpha_i)- v(\alpha_i') \big| |\alpha_i'|
                      - \big|v(\alpha_i) - v(\Npm) \big| |\Npm|  \Big) \\
             = &  |\beta| \Big(   \big|v(\alpha_i) - v(\Npm') \big| \big( |\Npm'| - |\Npm|  \big) 
                                    - \big| v(\Npm') - v(\Npm) \big| | \Npm| + \big| v(\alpha_i) - v(\alpha_i') \big| |\alpha_i'| \Big) \\
             = &  \mathcal{O}(|\beta||\alpha_i||\Npm|).                            
\end{align*}
In a similar way, it is easy to verify that if there is a weak wave $\beta$ located to
 the right of the interaction, the least favorable case is $v(\beta)=v(\Npm')< v(\alpha_i) < v(\alpha_i')$
\begin{align*} 
  [Q_2]_{\beta} = &|\beta| \Big( \big|v(\Npm') - v(\beta) \big|^+ |\Npm'| +\big|v(\alpha_i')-v(\beta)\big|^+ |\alpha_i'| 
        - \big|v(\Npm) - v(\beta) \big|^+ |\Npm|- \big|v(\alpha_i)-v(\beta) \big|^+ |\alpha_i|    \Big)\\
  =    & |\beta| \Big(  \big| v(\alpha_i')-v(\Npm') \big| |\alpha_i'|  - \big| v(\Npm) - v(\Npm') \big| |\Npm|  
         - \big|v(\alpha_i)-v(\Npm') \big| |\alpha_i|    \Big)\\
   = & \mathcal{O}(|\beta||\alpha_i||\Npm|).
\end{align*}
%==========================================================
\end{proof}

% % % % % % % % % % % % % % % % % % % % % % % % % % % % % % % % %
% % % % % % % % % % % % % % % % % % % % % % % % % % % % % % % % %
% % % % % % % % % % % % % % % % % % % % % % % % % % % % % % % % %

%\begin{lemma}[Case 5. $\Npm \CUp \rightarrow \Npm' \CUp'$]\label{Qcase:5}
%Consider an interaction at some time $t$ between a nonclassical shock $\Npm$ and an increasing shock $C^\uparrow$ leads
%to the same type of outgoing waves ${\Npm}'$ and ${C^\uparrow}'$. 
%If the perturbation satisfies an a priori bound $\epsilon(t) \leq \mathcal{O}(|\Npm|)$, then the functional
%$$
% [Q]  \leq -...
%$$
%\end{lemma}
%%*****************************************************************************
%\begin{proof}
%\begin{align}
%[Q_1] =& - \, \big| v(\Npm) - v(C^\uparrow) \big|^+ \, |\Npm|\cdot|C^\uparrow| \nonumber\\
%%=& -\mathcal{O}(|\alpha_i|\cdot |\Npm|\cdot|C^\uparrow|).\label{Q1case5} 
%\end{align}
%\end{proof}

%*************************************************************************
%**********************  case 5  *****************************************
%*************************************************************************
\begin{lemma}[Case 5. $\beta_j \Cpm  \rightarrow {\Npm}' {\CUp}'$ or $\Cpm \beta_j \rightarrow {\Npm}' {\CUp}'$ ]\label{Qcase:5}
Consider at some time $t$ an interaction between a weak wave $\beta_j$, $j \neq i$, with a strong classical shock $\Cpm$ that leads to a splitting
$\Npm' \CUp'$. If the  perturbation satisfies an a priori bound 
\be                   \label{qcase5:apriori}
            \epsilon(t) \leq \mathcal{O}(1), 
\ee
then
$$
   [Q] = - \big|\mathcal{O}\big( |\beta_j| \cdot |\Cpm| \big) \big|.
$$
\end{lemma}
% % % % % % % % % % % % % % % % % % % % % % % % % % % % % % % %
% % % % % % % % % % % % % % % % % % % % % % % % % % % % % % % %
\begin{proof}
The interaction between the weak wave $\beta_j$, $j \neq i$, with $\Cpm$ leads to $\Npm \CUp$. 
For this interaction, $|\Cpm|=|\Npm'|+|\CUp'|+\mathcal{O}(|\beta_j|\cdot |\Cpm|)$
and the wave speeds are all close,
$$
    v(\Cpm) = v(\Npm') + \mathcal{O}(|\beta_j|\cdot |\Cpm|) = v(\CUp') + \mathcal{O}(|\beta_j|\cdot |\Cpm|).
$$ 
First of all, we have 
$$
       [Q_1] = - \big| v(\beta_j) - v(\Cpm) \big| \cdot |\Cpm| \cdot |\beta_j| 
            = - \Big|\mathcal{O}\big( |\beta_j| |\Cpm| \big) \Big|.
$$ 
Let $\gamma$ be a weak wave at $x<y$ interacting with a strong wave. Then
\begin{align*}  
  [Q_2]_{\beta}
  =& \, |\gamma|\cdot \Big[ \big|v(\gamma) - v(\Npm') \big|^+ |\Npm'| +\big|v(\gamma)- v(\CUp') \big|^+|\CUp'|- \big|v(\gamma) - v(\Cpm) \big|^+|\Cpm|\Big]\\
  = & \, |\gamma|\cdot \big|v(\gamma) - v(\Cpm) \big|^+ \Big[ |\Npm'| + | \CUp'| - |\Cpm| \big]  + \mathcal{O}(|\gamma| |\beta_j| |\Cpm|) \\
  = & \,  \mathcal{O}(|\gamma|\cdot |\beta_j|\cdot|\Cpm|).
\end{align*}
Similarly, if $\gamma$ is a weak wave to the right of the interaction, we have 
\begin{align*}  
  [Q_2]_{\beta}
  =& \, |\gamma|\cdot \Big[ \big| v(\Npm')-v(\gamma) \big|^+ |\Npm'| +\big|v(\CUp')-v(\gamma) \big|^+|\CUp'|- \big| v(\Cpm)-v(\gamma) \big|^+|\Cpm|\Big] \\
 = & \, \mathcal{O}(|\gamma|\cdot |\beta_j|\cdot|\Cpm|).
\end{align*}
By imposing an a priori condition on $\sum |\gamma| \leq \mathcal{O}(1)$ we can conclude that
$
           [Q] \leq - \big| \mathcal{O}(|\beta_j|\cdot|\Cpm|) \big|. 
$
\end{proof}

%*************************************************************************
%**********************  case 6  *****************************************
%*************************************************************************
\begin{lemma}[Case 6. $\beta_j \Npm \rightarrow \Npm' \beta_j'$ or $\Npm \beta_j \rightarrow \beta_j' \Npm' $]\label{Qcase:6}
Consider at some time $t$ an interaction between a weak wave $\beta_j$, $j \neq i$ with a strong classical shock $\Npm$ that 
leaves of the nonclassical shock unperturbed. 
If the  perturbation satisfies an a priori bound 
\be     \label{qcase6:apriori}
       \epsilon(t) \leq \mathcal{O}(1),
\ee
then one has 
$$
   [Q] = - \Big| \mathcal{O}\big( |\beta_j| \cdot |\Npm|\big) \Big|.
$$
\end{lemma}
% % % % % % % % % % % % % % % % % % % % % % % % % % % % % % % %
% % % % % % % % % % % % % % % % % % % % % % % % % % % % % % % %
\begin{proof}
The weak wave $\beta_j$ from the $j$-th family crosses $\Npm$ from the 
left $(j>i)$ or from the right $(j<i)$, possibly generating weak secondary waves of 
strength $\mathcal{O}(|\beta_j| |\Npm|)$. 
We have the identity
\begin{gather*}
      |\Npm|=|\Npm'| + \mathcal{O}(|\beta_j|\cdot|\Npm|).
\end{gather*}
During such an interaction 
$$
    [Q_1] = -\big| v(\beta_j) - v(\Npm) \big| \cdot |\Npm| \cdot |\beta_j| 
           = - \big| \mathcal{O}\big( |\beta_j| |\Npm|\big) \big|.
$$    
For a weak wave $\gamma$ interacting from the left 
\begin{align*}  
  [Q_2]   = & \, |\gamma|\cdot \Big[ \big|v(\gamma) - v(\Npm') \big|^+ |\Npm'| 
                             +\big|v(\gamma)- v(\beta_j') \big|^+|\beta_j'|- \big|v(\gamma) - v(\Npm) \big|^+|\Npm|\Big]\\
             = & \, |\gamma| \big|v(\gamma) - v(\Npm) \big|^+ \big[  |\Npm'|  - |\Npm| \big]  +   \mathcal{O}(|\gamma|\cdot |\beta_j|\cdot|\Npm|) \\                
              \leq & \,  \mathcal{O}(|\gamma|\cdot |\beta_j|\cdot|\Npm|).
\end{align*}
If $\gamma$ is a weak wave coming from the right and interacting with the strong wave, we have 
\begin{align*}  
  [Q_2] 
  =& \, |\gamma|\cdot \Big[ \big| v(\Npm')-v(\gamma) \big|^+ |\Npm'| +\big|v(\beta_i')-v(\gamma) \big|^+|\beta_i'| 
                          - \big|v(\Npm)-v(\gamma) \big|^+|\Npm|\Big]\\
  \leq & \, \mathcal{O}(|\gamma|\cdot |\beta_j|\cdot|\Npm|).
\end{align*}
In conclusion, the usual a priori condition $\sum |\gamma| \leq \mathcal{O}(1)$ implies that
$$
     [Q] = [Q_1] + [Q_2]= - \Big| \mathcal{O}\big(  |\beta_j|\cdot|\Npm| \big) \Big|.
     \qedhere
$$
\end{proof}
%==========================================================
%==========================================================
%==========================================================
%==========================================================

%*************************************************************************
%**********************  case 7  *****************************************
%*************************************************************************
\begin{lemma}[Case 7. $\beta_j C \rightarrow C' \beta_j'$ or $C \beta_j \rightarrow \beta_j' C'$] \label{Qcase:7}
Consider at some time $t$ a weak wave $\beta_j$ belonging to the family $j\neq i$ and crossing $C = \Cpm$ or $\CUp$.
Then if the perturbation satisfies 
\be      \label{qcase7:apriori}
          \epsilon(t) \leq \mathcal{O}(1), 
\ee
then one has 
$$
        [Q] = - \big| \mathcal{O}\big( |\beta_j| \cdot  |C| \big) \big|.
$$
\end{lemma}
%==========================================================
\begin{proof}
Consider first the situation when $C = \Cpm$. We have the following estimates,
$$
        |\Cpm|=|\Cpm'|+\mathcal{O}(|\beta_j|\cdot|\Cpm|), \qquad   v(\Cpm)   = v(\Cpm') + \mathcal{O}(|\beta_j|\cdot|\Cpm|).
$$
We observe that
$$
   [Q_1] = - \big|v(\beta_j) - v(\Cpm) \big| \cdot |\Cpm| \cdot |\beta_j| = - \Big|\mathcal{O}\big( |\beta_j| |\Cpm| \big) \Big|.
$$
For a weak wave $\gamma$ interacting from the left 
\begin{align*}  
  [Q_2] 
  =& \, |\gamma|\cdot \Big[ \big|v(\gamma) - v(\Cpm') \big|^+ |\Cpm'|- \big|v(\gamma) - v(\Cpm) \big|^+|\Cpm|\Big]\\
  =& \, |\gamma|\cdot \big|v(\gamma) - v(\Cpm) \big|^+ \big[ |\Cpm'|- |\Cpm| \big] + \mathcal{O}(|\gamma| \cdot |\beta_j|\cdot|\Cpm|)\\
  \leq & \, \mathcal{O}(|\gamma|\cdot |\beta_j|\cdot|\Cpm|).
\end{align*}
If $\gamma$ is a weak wave coming from the right and interacting with the nonclassical wave, then the same argument shows that
\begin{align*}  
  [Q_2] = |\gamma|\cdot \Big[ \big| v(\Cpm')-v(\gamma) \big|^+ |\Cpm'| - \big|v(\Cpm)-v(\gamma) \big|^+|\Cpm|\Big] 
  \leq \mathcal{O}(|\gamma|\cdot |\beta_j|\cdot|\Cpm|).
\end{align*}
The required a priori condition is again of the form $\epsilon(t) = \sum|\gamma| \leq \mathcal{O}(1)$.

We now study the case when $C = \CUp$ and we begin with the  estimates.
$$
         |\CUp|=|\CUp'|+\mathcal{O}(|\beta_j|\cdot|\CUp|), \qquad (\CUp)=v(\CUp') +\mathcal{O}(|\beta_j|\cdot|\CUp|).
$$
During this interaction, 
$$
    [Q_1] = - \big| v(\beta_j) - v(\CUp) \big| \cdot |\CUp| \cdot |\beta_j| = - \big| \mathcal{O}(|\beta_j|\cdot|\CUp|) \big|.
$$    
The same proofs as in the last two lemmas show that when an a priori
condition $\epsilon(t) \leq \mathcal{O}(1)$ holds, then
$$
   [Q] = [Q_1] + [Q_2] = - \big| \mathcal{O}(|\beta_j|\cdot|\CUp|) \big|. \qedhere
$$
\end{proof}

%*************************************************************************
%**********************  case 8  *****************************************
%*************************************************************************
%\begin{lemma}[Case 8. $\beta_j \Npm \rightarrow \Npm' \beta_j'$ or $\Npm \beta_j \rightarrow \beta_j' \Npm'$]\label{Qcase:10}
%Consider at some time $t$ a weak wave $\beta_j$ belonging to the family $j\neq i$ and crossing $\Npm$.
%Then if the perturbation satisfies $\epsilon(t) \leq \mathcal{O}(|\Npm|)$, then 
%$$
%        [Q] = \mathcal{O}( |\beta_j| |\Npm|^2).
%$$
%\end{lemma}
%\begin{proof}
%We have the following estimates.
%$$
%       |\Npm|=|\Npm'|+\mathcal{O}(|\beta_j|\cdot|\Npm|), \qquad  v(\Npm) =  v(\Npm') +\mathcal{O}(|\beta_j|\cdot|\Npm|).
%$$
%These two estimates and an a priori condition  imply that $[Q_1] = 0$ and
%$[Q_2] = \mathcal{O}(|\beta_j|\cdot|\Npm|^2)$.
%\end{proof}

%*****************************************************************************************************************************
%*****************************************************************************************************************************
%*****************************************************************************************************************************
%*****************************************************************************************************************************
\subsection{Global estimates}
\label{sec:globalestimates}
%\subsection{Proof of Theorem \ref{mainresult}}

We now possess all the required estimates for interactions within splitting-merging solutions. We need to find a constant $K$
for which $\myW + K Q$ is decreasing during all interactions. We now reinterpret Theorem  \ref{mainresult} and state
the a priori conditions on the perturbation $\epsilon(t)$ and on the constants $k_*^{\lessgtr}, K$  
required in the proof.

\subsubsection*{Constraints on the weights in $W$}

The key is to choose the constants $K, k_*$ and $k_*^{\lessgtr}$ in such a way that waves crossing through the strong shocks
lead to decreases in $\myW+KQ$. The analysis of the various interactions, to be presented
immediately below, will lead to the following constraints on these constants.

\begin{itemize}
\item[(W1)] The constants $\kappaL, \kappaM$ and $\kappaR$ must satisfy
\begin{align}   
%                  \kappaL > & \,  \kappaW,  \label{bound:k1a} \\
%                  \kappaM > & \, \kappaW,  \label{bound:k1c}  \\                  
%                  \kappaR > & \, \kappaW,  \label{bound:k1b} \\
                  \kappaL > & \, (1+\Cff )\kappaM. \label{bound:k1d} 
\end{align}

\item[(W2) ] The weights $\kappaLless, \kappaMless$ and $\kappaRless$ must satisfy
\begin{equation}   
           \kappaLless <   \kappaMless <   \kappaRless, \label{bound:k2a} 
\end{equation}
\item[(W3) ] The weights $\kappaLgrt, \kappaMgrt$ and $\kappaRgrt$ must satisfy
\begin{equation}   
                  \kappaLgrt >   \kappaMgrt >  \kappaRgrt.
                   \label{bound:k3b} 
\end{equation}
\end{itemize}
The validity of the following lemma and the proposed values for the constants can be checked by inspection.  

%********************************
\begin{lemma}  \label{lem:weights}
Assuming only property (H4)  of the kinetic function $\Phi^{\flat}_i$, the conditions on the weights
 \eqref{bound:k1d}-\eqref{bound:k3b} can be satisfied independently of the size $\epsilon(t)$ 
 of the perturbation, of the strength of the nonclassical shock and of the constant $K$. 
 For each $\zeta>0$, one such a choice is given by 
\be
\begin{array}{ccc}
             \kappaL =  1+\Cff +\zeta,  &  \kappaM = 1,                   & \kappaR = 1, \\
             \kappaLless =  1-\zeta,      & \kappaMless = 1,              & \kappaRless = 1+\zeta, \\
            \kappaLgrt = 1+\zeta,        &  \kappaMgrt= 1,               & \kappaRgrt = 1-\zeta. 
\end{array}            
\ee
\end{lemma}

%*************************
\subsubsection*{Constraints on the perturbation}

We present conditions that the perturbation $\epsilon(t)$ must satisfy at all times $t$: 
%The bounds being independent of time, we will therefore have to eventually demonstrate
%that the perturbation is uniformly bounded in time.

\begin{itemize}
\item[(P1)] Lemmas \ref{Qcase:1}-\ref{Qcase:4} require the following constraints  
\begin{align} 
          \epsilon(t) \leq & \, \mathcal{O}\big( |\Cpm|  \big),  \label{perturbation:bound1} \\     
          \epsilon(t) \leq & \, \mathcal{O}\big( |\CUp|  \big),  \label{perturbation:bound2} \\  
          \epsilon(t) \leq & \, \mathcal{O}\big( |\Npm|  \big).  \label{perturbation:bound3}   
\end{align}
Note that Lemmas \ref{Qcase:1} and \ref{Qcase:3} require two different bounds of the form \eqref{perturbation:bound1}.

\item[(P2)] Lemmas \ref{Qcase:5}-\ref{Qcase:7} require three different bounds of the form
\be
             \epsilon(t) \leq  \mathcal{O}\big( 1  \big).  \label{perturbation:bound4}
\ee
\end{itemize}
The functions involved in the $\mathcal{O}$ bounds \eqref{perturbation:bound1}-\eqref{perturbation:bound4}
are all the result of quadratic terms in Glimm's interaction estimates, cf.~Theorem \ref{thm:Glimm_interaction_estimates}. 
%
% Verification of the statement above :
% Case 1 :
% Case 2 :

\subsubsection*{Constraints on the weights in $W+KQ$}

The following fundamental constraints are required to insure that decreases in the
interaction potential compensate for the introduction of spurious second-order waves during each interaction.
The analysis of the various interactions, to be presented
immediately below, will lead to the following constraints on these constants.

\begin{itemize}
\item[(Q1) ] The weight $K$ on the interaction potential must satisfy
\begin{align}   
    \kappaLless + \kappaRgrt                  < & \, K \big| \mathcal{O}(|\Cpm|) \big|, \label{bound_case1a} \\
    \kappaMless + \kappaRgrt                 < & \,  K \big| \mathcal{O}(|\CUp|) \big|, \label{bound_case1b} \\
    \kappaM + \kappaLless + \kappaMgrt < & \, K \big| \mathcal{O}(|\Npm|) \big|, \label{bound_case1c} \\
     \kappaLless + \kappaRgrt                 < & \,  K \big| \mathcal{O}(1) \big|, \label{bound_case1d} \\
    \kappaM + \kappaLless + \kappaMgrt < & \, K \big| \mathcal{O}(1) \big|, \label{bound_case1e} \\
     \kappaMless + \kappaRgrt                < & \, K \big| \mathcal{O}(1) \big|. \label{bound_case1f} 
\end{align}
\end{itemize}

We now summarize our earlier estimates by combining, case by case, the estimates for the total variation, Lemmas 
\ref{case:1}-\ref{case:7}, with the estimates for the interaction potential, Lemmas \ref{Qcase:1}-\ref{Qcase:7}.
Each time, we will indicate which assumptions on the weights (W1)-(W3), the perturbation (P1)-(P2) and on the interaction 
potential (Q1). are required to show that $[\myW + KQ]$ is negative. We have already showed in Lemma \ref{lem:weights}
that the constraints on the weights can be satisfied. As should be clear, for the moment we have not demonstrated
that $\epsilon(t)$, defined in \eqref{defn:pertubation_wrt_t},  is a priori bounded in time. The estimates summarized 
below are only valid at the time $t$ of a fixed interaction, and show that for the strength $\epsilon(t)$ of the perturbation at that time,
the constants can be adjusted to insure that $\myW+KQ$ will decrease during  that interaction.
We will eventually have to show that the bounds on the perturbation can be satisfied uniformly in time.

%***********************************************************************************
\begin{flushleft}
{\bf Case 1:} ($\alpha_i \Cpm  \rightarrow \Npm \CUp$ or $\Cpm \alpha_i \rightarrow \Npm \CUp$).
In Lemma \ref{case:1}, we showed that $\alpha_i$ is faster than $\Cpm$, then 
$$
       [ \myW ] =  |\alpha_i | \Big( -\kappaL 
                            + \big(   \kappaLless +  \kappaRgrt \big) \mathcal{O}(\Cpm) \Big),
$$   
while Lemma \ref{Qcase:1} has that $[Q] = - | \mathcal{O}(|\alpha_i||\Cpm|^2) |$  
when \eqref{perturbation:bound1} is imposed. Therefore,
if the constraint  \eqref{bound_case1a} is satisfied, then $\myW + KQ$ is decreasing.
Similarly, Lemma \ref{case:1} showed that when $\alpha_i$ is slower than $\Cpm$, then
$$
       [ \myW ] =  |\alpha_i | \Big( -\kappaR 
                            + \big(   \kappaLless +  \kappaRgrt \big) \mathcal{O}(\Cpm) \Big)
$$       
Assuming that the perturbation satisfies an a priori bound \eqref{perturbation:bound1}, 
then $[Q] = - | \mathcal{O}(|\alpha_i||\Cpm|^2) |$.
For weights satisfying  \eqref{bound_case1a}, then $[\myW+KQ] $ is negative.
\end{flushleft}
%***********************************************************************************

%***********************************************************************************
\begin{flushleft}
{\bf Case 2:} ($ \Npm \CUp \rightarrow \Cpm$).
For the strong interaction $ \Npm \CUp \rightarrow \Cpm$ Lemma \ref{case:2} stated that 
\begin{align*}
      [ \myW ] = & \,\,   (  \kappaLless +  \kappaRgrt ) \mathcal{O}( | \Npm|  |\CUp| |v(\Npm) - v(\CUp)| ).
\end{align*} 
Lemma \ref{Qcase:2} states that if the inequality \eqref{perturbation:bound2} is
enforced, then $[Q]$ is negative but proportional to the same quantity, and therefore 
$$
   [\myW + KQ] = \big( \kappaLless +  \kappaRgrt -K | \mathcal{O}(1) | \big) 
                   \Big| \mathcal{O}\big( | \Npm| |\CUp| |v(\Npm) - v(\CUp)| \big) \Big|.
$$
is negative if $K$ is chosen to satisfy an a priori bound of the form \eqref{bound_case1d}. 
\end{flushleft}
%***********************************************************************************

%***********************************************************************************
\begin{flushleft}
{\bf Case 3:} ($\alpha_i C  \rightarrow {C}'$ or $C \alpha_i \rightarrow {C}'$).  According to Lemma \ref{case:3}, we have that 
if the weak wave crosses the classical strong shock $C=\Cpm$  from the left, then
\begin{align*}
      [ \myW ] = & \,\,  |\alpha_i | \Big( -\kappaL 
                                    + (  \kappaLless +  \kappaRgrt ) \mathcal{O}(|\Cpm|) \Big).
\end{align*}
Lemma \ref{Qcase:3} shows that if \eqref{perturbation:bound1} holds, then 
$[Q] = - | \mathcal{O}(|\alpha_i||\Cpm|^2)|$.
Therefore, if the constraint \eqref{bound_case1a} is satisfied, then
the  variation  $[\myW+KQ] \leq -\kappaL |\alpha_i| $ is strictly negative.

If the weak wave $\alpha_i$ crosses $C = \Cpm$ from the right, then Lemma \ref{case:3}
states that the change in $\myW$  is
\begin{align*}
      [ \myW ] =  & \,\,  |\alpha_i | \Big( -\kappaR 
               + (  \kappaLless +  \kappaRgrt ) \mathcal{O}(|\Cpm|) \Big).
\end{align*}
Again, $[Q] = - \big| \mathcal{O}(|\alpha_i| |\Cpm|^2) \big|$ when \eqref{perturbation:bound1} holds, 
and so 
\begin{equation}  \label{Gcase:3}
    [\myW+KQ](t)  \leq -\kappaR |\alpha_i| 
\end{equation}    
when the bound \eqref{bound_case1a} holds.

Suppose now that $C = \CUp$ is a classical shock. If $\alpha_i$ crosses $\CUp$ from the left, then lemmas \ref{case:3} and
\ref{Qcase:3} state 
\begin{align*}
     [\myW] = & \, |\alpha_i| \big( -\kappaM  
                                                 + (   \kappaMless +  \kappaRgrt )  \mathcal{O}(|\CUp|) \big), \\
     [Q] = & \, - \big| \mathcal{O}(|\alpha_i||\CUp|^2)  \big|.
\end{align*}
If the constraint  \eqref{bound_case1b} is satisfied a priori, 
then $[\myW + KQ] \leq -\kappaM |\alpha_i|$ is negative.

Similarly, if $\alpha_i$ comes from the right, then 
$$
        [\myW] =  |\alpha_i| \big( -\kappaR 
                                                + (   \kappaMless +  \kappaRgrt )  \mathcal{O}(|\CUp|) \big).
$$
Using the estimate for $[Q]$ from Lemma \ref{Qcase:3} and the a priori condition
 \eqref{bound_case1b}, we conclude that $[\myW + KQ] \leq -\kappaR |\alpha_i|$ is strictly decreasing.
\end{flushleft}
%***********************************************************************************

%***********************************************************************************
\begin{flushleft}
{\bf Case 4:} ($\alpha_i \Npm  \rightarrow {\Npm}' \alpha_i'$).
According to Lemma \ref{case:4} and Lemma \ref{Qcase:4}
\begin{align*}
 [\myW] = & \,  |\alpha_i| \big(  - (\kappaL - \operatorname{sign}(\alpha_i)  \kappaM ) 
                   + \Cff \kappaM  
                    + (  \kappaM + \kappaLless +  \kappaMgrt )  \mathcal{O}(|\Npm|) \big),\\
 [Q] = & \, - \big|  \mathcal{O}(|\alpha_i| |\Npm|^2)  \big|,               
\end{align*}
where the last bound required \eqref{perturbation:bound3}.
We impose the bound \eqref{bound:k1d} and the constraint \eqref{bound_case1c}
on $K$ and these imply that 
\begin{equation}  \label{Gcase:4}
     [\myW+ KQ] \leq \big( -\kappaL + (1+C_i) \kappaM \big) \, |\alpha_i |.
\end{equation}     
\end{flushleft}
%***********************************************************************************

%***********************************************************************************
\begin{flushleft}
{\bf Case 5:} ($\beta_j \Cpm  \rightarrow {\Npm}' {\CUp}'$ or $\Cpm \beta_j \rightarrow {\Npm}' {\CUp}'$). 
Lemma \ref{case:6} and Lemma \ref{Qcase:6} states that when 
 a weak wave $\beta_j$, with $j > i$, approaches $\Cpm$ from the left, and if the perturbation
 satisfies \eqref{perturbation:bound4}, then
\begin{align*}
      [ \myW ] =  & \, |\beta_j  | \Big( -\kappaLgrt  + \kappaRgrt 
                                 +  (  \kappaLless +  \kappaRgrt ) \mathcal{O}(|\Cpm|) \Big),\\
        [Q] = & \, - \big| \mathcal{O}( |\beta_j||\Cpm|) \big|.                         
\end{align*} 
The linear terms in $[W]$ must control the possible increase in $[Q]$ and so we balance the terms as
\begin{align*}
 [\myW + KQ] = & \, |\beta_j| \Big\{  - \kappaLgrt +  \kappaRgrt    
                      + \big(  \kappaLless + \kappaRgrt   - K | \mathcal{O}(1)| \big) \mathcal{O}(|\Cpm|)\Big\}.
\end{align*}
We impose the zero-th order bound \eqref{bound:k3b} the a priori condition \eqref{bound_case1a} 
so that $\myW+KQ$ becomes decreasing.

Similarly, if $\beta_j$ is a weak wave approaching $\Cpm$ from the right and the perturbation
is small in the sense of \eqref{perturbation:bound4}, then 
\begin{align*}
 [\myW + KQ] = & \, |\beta_j| \Big\{  - \kappaRless + \kappaLless  
                                     + \big( \kappaLless +  \kappaRgrt   -  K \mathcal{O}(1) \big) 
                                              \mathcal{O}(|\Cpm|)\Big\}.
\end{align*}
The conditions and \eqref{bound:k2a} and  \eqref{bound_case1d} are sufficient to make the decreasing functional.
\end{flushleft}
%***********************************************************************************

%***********************************************************************************
\begin{flushleft}
{\bf Case 6:} ($\beta_j \Npm \rightarrow \Npm' \beta_j'$ or $\Npm \beta_j \rightarrow \Npm'  \beta_j'$).
Consider first the interaction with $\beta_j$ initially to the left of $\Npm$. Lemmas \ref{case:7} 
and \ref{Qcase:7} state that with the condition \eqref{perturbation:bound4}, then
\begin{align*}
      [ \myW ] =  & \, |\beta_j  | \Big( -\kappaLgrt + \kappaMgrt  
                                 +  ( \kappaM + \kappaLless +  \kappaMgrt ) \mathcal{O}(|\Npm|) \Big),\\
        [Q] = & \, - \big| \mathcal{O}( |\beta_j||\Npm|) \big|.                         
\end{align*} 
To balance these terms, we must examine
\begin{align*}
 [\myW + KQ] = & \, |\beta_j| \Big\{ -\kappaLgrt + \kappaMgrt   
                                 +  \big( \kappaM +  \kappaLless + \kappaMgrt  - K |\mathcal{O}(1)| \big) \mathcal{O}(|\Npm|) 
                                     \Big\}.
\end{align*}
This functional is decreasing if we further assume that conditions \eqref{bound:k3b} and 
\eqref{bound_case1e} holds, and the exact decrease is
\begin{equation}  \label{Gcase:6a}
     [\myW+KQ ] \leq \big( -\kappaLgrt + \kappaMgrt \big)  |\beta_j|.
\end{equation}

Similarly, when $\beta_j$ is initially to the right of $\Npm$, then we impose
the bound \eqref{perturbation:bound4} to deduce
$$
    [\myW + KQ] =  |\beta_j| \Big\{ -\kappaMless + \kappaLless 
                                 +  (  \kappaM + \kappaLless  + \kappaMgrt - K |\mathcal{O}(1)| ) \mathcal{O}(|\Npm|) 
                                     \Big\} \, 
$$
and then we conclude that this is negative if conditions \eqref{bound:k2a} and \eqref{bound_case1e} holds.
In fact, the amount of decrease is
\begin{equation}  \label{Gcase:6b}
     [\myW+KQ ] \leq \big( -\kappaMless + \kappaLless \big)  |\beta_j|.
\end{equation}
\end{flushleft}
%***********************************************************************************
 
%***********************************************************************************
\begin{flushleft}
{\bf Case 7:} ($\beta_j C \rightarrow C' \beta_j'$ or $C \beta_j \rightarrow \beta_j' C'$). 
When the weak wave $\beta_j$ is incoming from the left and $C = \Cpm$, then using lemmas \ref{case:7}
and \ref{Qcase:7} under the assumption \eqref{perturbation:bound4}, we find
\begin{align*}
 [\myW + KQ] = & \, |\beta_j| \Big\{ -\kappaLgrt + \kappaRgrt  
                                 +  \big(  \kappaLless + \kappaRgrt - K|\mathcal{O}(1)| \big) \mathcal{O}(|\Cpm|) 
                                   \Big\}  \leq -\kappaLgrt |\beta_j|. 
\end{align*}
 which requires the inequality \eqref{bound:k3b} and the a priori condition \eqref{bound_case1d} 
 in order to imply that it is strictly negative. 
 
 Similarly, for right incoming weak waves and $C = \Cpm$ and \eqref{perturbation:bound4}, we have
 \begin{align*}
 [\myW + KQ] = & \, |\beta_j| \Big\{ -\kappaRless + \kappaLless  
                                 +  \big( \kappaLless + \kappaRgrt - K|\mathcal{O}(1)| \big) \mathcal{O}(|\Cpm|) 
                                    \Big\}.
\end{align*}
This quantity is negative if the additional conditions \eqref{bound:k2a} and \eqref{bound_case1d} holds,
in which case the decrease is
\begin{equation*}   
      [\myW +KQ ] \leq ( -\kappaRless + \kappaMless ) |\beta_j|.
\end{equation*}

The last two interactions covered by Case 7. involve either a left incoming $\beta_j$ or a right
incoming $\beta_j$ interacting with $C = \CUp$, in which case the a priori bound \eqref{perturbation:bound4} implies respectively
\begin{equation}  \label{Gcase:7a}
 [\myW + KQ] =  |\beta_j| \Big\{ -\kappaMgrt + \kappaRgrt 
                                 +  \big(  \kappaMless + \kappaRgrt - K|\mathcal{O}(1)| \big) \mathcal{O}(|\CUp|) 
                                     \Big\} \leq (-\kappaMgrt + \kappaRgrt) |\beta_j|. 
\end{equation}
or
\begin{equation}  \label{Gcase:7b}
 [\myW + KQ] =  |\beta_j| \Big\{ -\kappaRless + \kappaMless   
                                 +  \big(  \kappaMless   + \kappaRgrt - K |\mathcal{O}(1)| \big) \mathcal{O}(|\CUp|) 
                                     \Big\} \leq ( -\kappaRless + \kappaMless ) |\beta_j| 
\end{equation}
The functional $\myW + KQ$ is strictly decreasing, as written above, if the bounds \eqref{bound:k3b}, \eqref{bound:k2a} and \eqref{bound_case1f}
are satisfied.
\end{flushleft}
%***********************************************************************************

\begin{proof} [Proof of Theorem \ref{mainresult}]
The remaining steps in the proof of existence of a weak solution using front-tracking approximations
are similar in every published proof \cite{Risebro-FT,Bressan-FT} and we will
only sketch the final argument. The key step is to demonstrate that the family of
front-tracking approximations, parameterized by $h$, possesses
a convergent subsequence that converges as $h \to 0$. The proof that the limit
is a weak solution is straightforward and will be omitted. Recall that $h$ is
a bound on the strength of rarefaction-shocks, on the errors in the propagation
speeds of those discontinuities, and on the approximation of the initial data. The
bound  $\eps_0$ on the size of the initial perturbation \eqref{3.4} is not to be confused with the parameter $h$.

We begin by choosing the state $u_* \in \Bone$ so that $\Phiis(u_*)$ and $\Phiif(u_*)$
are also in $\Bone$. The euclidean distance in phase-space is equivalent to
the generalized shock strength, hence there exists $\kappa_1$ for which the bound 
$$
          \| v_0 \|_{L^{\infty}} \leq \kappa_1 \operatorname{TV}(v_0) < \kappa_1 \eps_0
$$
can be used to ensure that all states in $u_0 = \overline{u}_0 + v_0$
belong to $\Bone$. To simplify the notation, the function $\eps(t)$ will denote
the size of the perturbation in $\operatorname{TV}$, that is the size of
the splitting-merging solution after subtracting the strong waves.

The analysis in Section \ref{sec:globalestimates} has shown that under certain conditions, the functional  
$\myW+KQ$ is decreasing. We need to establish that this functional is equivalent
to the quasi-norm $\operatorname{TV}$ for values of $\eps_0$ and $K$ that 
are compatible with the constraints. Clearly, there exists $\kappa_2$ such that
$$
         \eps(t) \leq \kappa_2 W(u(t)) 
                     \leq \kappa_2 \big(W+KQ \big) \big( u(t) \big),
$$
but the existence of a $\kappa_3$ for which the inverse inequality holds,
$$
        \kappa_3 \big(W+KQ \big) \big( u(t) \big) \leq  \eps(t),
$$
requires an a priori bound of the form $K W \leq \kappa_4$ for some
constant $\kappa_4$. 
The constant $K$ must satisfy the a priori
lower bounds (Q1), namely \eqref{bound_case1a}-\eqref{bound_case1f}, but since Lemma 
\ref{lem:weights} shows that for $\zeta < 1/2$ the lower bounds are all
at most $3$, these bounds can be rewritten as
\be   \label{eq:constraintK}
        3 \leq K \min  \Big\{ \big\| \mathcal{O}(1) \big\|_\infty,  \big\| \mathcal{O}\big( |\Cpm| \big)  \big\|_\infty, 
                 \big\| \mathcal{O}\big( |\Npm|)  \big\|_\infty,  \big\| \mathcal{O}\big( |\CDown| \big)  \big\|_\infty  \Big\},
\ee
where each of the uniform norms are taken over $\Bone$.
Assume that $K$ is sufficiently large to satisfy this constraint and then
impose that $\epsilon_0$ is sufficiently small that there exists a 
constant $\kappa_4$ for which the constraint $K W(0) \leq \kappa_4$ is satisfied initially. 
On the other hand, constraints (P1)  and (P2)  also require the total variation of the perturbation 
to decrease as the size of the strong waves decreases,
\be   \label{eq:constraintP1}
        \eps(t) \leq  \min  \Big\{ \big\| \mathcal{O}(1) \big\|_\infty,  \big\| \mathcal{O}\big( |\Cpm| \big)  \big\|_\infty, 
                 \big\| \mathcal{O}\big( |\Npm|)  \big\|_\infty,  \big\| \mathcal{O}\big( |\CDown| \big)  \big\|_\infty  \Big\}.
\ee
This upper bound on $\epsilon(t)$ can be satisfied simultaneously with the upper
bound  $W \leq \kappa_4 / K$. Under these conditions, because $\myW$ will remain uniformly
bounded above in time, the equivalence between $\operatorname{TV}$ and $W+KQ$ 
will hold for all time.

Assume that the state $u_*$ is fixed and thus, the strength of the strong waves
are only small perturbations of the initial configuration $\overline{u}_0$.
Under these conditions, $(W+KQ)(u(t))$ is monotone decreasing and therefore bounded until
the approximation leads to states outside $\Bone$. Since $(W+KQ)(u(t))$ is equivalent to 
$\operatorname{TV}(u(t) - \overline{u}_0)$, at the cost of further restricting
$\operatorname{TV}(v_0)$, the boundedness of $W+KQ$ implies that
the approximation cannot grow in $L^{\infty}$ outise of $\Bone$. Glimm's functional is therefore bounded
for all time. The front-tracking approximations are thus continuous with respect to time in
$L^1(\mathbb{R})$, and uniformly bounded in $\operatorname{TV}$ and $L^{\infty}$.
Helley's Theorem therefore provides a subsequence $\{ h_k \}_k$ converging to zero
for which $u_{h_k}$ converges in $L^1_{\text{loc}} \cap \operatorname{BV}$.  
\end{proof}

\begin{proof} [Proof of Theorem \ref{mainresult_small}]
The proof is almost identical to the earlier proof, except in this case we need to 
propose an explicit relation between the strength of the perturbation
and the strength of the strong shocks.

Consider now the case where the initial state $u_*^{(k)}$ belongs to a sequence in
$\Bone$ that converges to a single point on the manifold $\mathcal{M}_i$ where
the wave speed $\lambda_i$ has a critical point. For each index $k$, we may assume that
the states $\Phiis(u_*^{(k)})$ and $\Phiif(u_*^{(k)})$ are also in $\Bone$. The 
uniform $\mathcal{O}(1)$ bounds for $K$ and the total variation continue to hold.

In the proof of Theorem \ref{mainresult}, we saw that as the strength of the strong shock 
decreases, then the strength of the perturbation must decrease while the weight $K$ must increase.
More precisely, suppose that the constant $K$ satisfies exactly the constraint \eqref{eq:constraintK}, that is 
$$
       K   = \frac{3}{F_1\big(|\Cpm|, |\Npm|, |\CDown| \big)},
$$
for a function $F_1$ linear with respect to the wave strengths. We remark that 
the strength of these strong waves depend only on the left-hand state 
of $\Cpm$, or $\Npm$, and these strengths are all uniform ratios of each other
since $\Phiis$ and $\Phiif$ are Lipschitz.
On the other hand, the condition \eqref{eq:constraintP1} requires that
$$
    \eps(t) \leq F_2\big(|\Cpm|, |\Npm|, |\CDown| \big),
$$
for some function $F_2$ linear in the wave strengths.
Taking \eqref{eq:constraintK} and \eqref{eq:constraintP1} together, we have
$$
       K W \leq    K \cdot \frac{\eps(t)}{\kappa_3} = \frac{3}{\kappa_3} 
              \cdot \frac{F_2(|\Cpm|, |\Npm|, |\CDown| )}{F_1 (|\Cpm|, |\Npm|, |\CDown|)}.
$$
By uniform continuity over $\Bone$, the ratio $F_2/F_1$ is uniformly bounded and hence there
exists a constant $\kappa_4$ for which $KW \leq \kappa_4$. The 
equivalence between $\operatorname{TV}$ and $W+KQ$ is therefore 
uniform with respect to the vanishing strengths of the strong waves.

Constraints (P1)  are summarized by the condition \eqref{eq:constraintP1}
which, because the strengths of the strong waves are proportional to each other,
can be replaced by a simpler inequality
$$
       \eps_0 \leq \kappa_*  | \Npm | .
$$
The rest of the proof is the same as the one presented for Theorem \ref{mainresult}. 
\end{proof}

%****************************************************************
%****************************************************************
\section{Splitting-merging solutions with nucleation}
\label{sec:nucleation}

\subsection{Preliminaries}

The objective of this section is to show that splitting-merging solutions can only undergo a finite number
of splitting-merging cycles when the creation of nonclassical waves requires the more demanding
nucleation condition; see Definition \ref{state}. This is a generalization of an earlier result by LeFloch
and Shearer in the scalar case \cite{LeFlochShearer}.

The proof does not follow the techniques of the earlier work \cite{LeFlochShearer},
which nonetheless inspired preliminary work of two of the authors on the total variation for scalar laws \cite{LL},
and hence the current definition of wave strength. The key step of the proof presented here
is an extension of Lemma  5.2
from \cite{LL} which translates the nucleation condition into a lower bound on the
signed variation of the waves crossing $\Npm$ and $\CUp$ during a single splitting-merging
cycle. The proof therefore requires some careful accounting of all the waves which
interact with, and ultimately remove, the nonclassical wave.

This section begins with a description of the geometry and processes involved in a single
splitting-merging cycle. The extension of Lemma 5.2 from \cite{LeFlochShearer} is then presented. The
signed variation is introduced and related to the previous bound. We conclude the result by
demonstrating that the total variation decreases during each splitting-merging 
by an amount proportional to the strength of the nucleation condition 
cycle, and therefore, that a splitting-merging solution for a finite perturbation can only
undergo a finite number of splitting-merging cycles.

%************
\subsection{Nucleation and the splitting-merging cycle}

 This section contains a definition of nucleation, a statement of the main theorem, 
 and a detailed description of the
 waves and interactions one can expect to see during a splitting-merging cycle.

The nucleation condition and its associated Riemann solver were defined in Section \ref{Sec2}, and
in particular in Lemma \ref{thm:solver_nucleation}. For the sake of convenience, we repeat 
its definition. The nucleation condition assumes that there exists a $C^2$ function 
$\mu_i^n : \Bone \longrightarrow \mathbb{R}$ satisfying
$$
    \mu_i^{\natural}(u) \leq \mu_i^n(u) \leq \mu_i^{\sharp}(u), \quad \text{ for all } u \in \Bone.
$$
The nucleation criteria states that $i$-classical shocks are preferred over nonclassical shocks if
\be  \label{nucl:nucleation_condition}
        \big|  \mu_i(u_-) - \mu_i(u_+) \big| <  \big|  \mu_i(u_-) - \mu_i^n(u_-) \big|.
\ee
Lemma \ref{thm:solver_nucleation} describes the resulting Riemann solver. Furthermore,
we will assume that $\mu_i^n(u_-) < \mu_i^{\sharp}(u_-)$ uniformly over $\Bone$,
that is to say that there exists a strictly positive constant $\eta$  such that
$$
                   \eta <   \big| \mu_i^{\sharp}(u) - \mu^n_i(u) \big|, 
$$
for all states $u$ that are perturbations of the left-hand state $u_*$ in the 
splitting-merging solution. We are now in a position to state the main result of this section.

%*****************************

\begin{theorem}[Finiteness of splitting--merging patterns under nucleation]\label{nucleationresult} 
Given any strictly hyperbolic system \eqref{11} with genuinely nonlinear, linearly degenerate, or concave-convex 
characteristic families, together with
a nonclassical Riemann solver satisfying a nucleation condition, there exist  constants $\kappa_* $ 
and $\delta_* < \delta_0$ depending only upon the flux function and the
constant $\Cff$ arising in (H4), so that the following property holds: for every family of kinetic mappings $\Phif_i$ 
associated with concave-convex families (indexed by $i$), provided they satisfy 
the conditions (H1)--(H4), any three state pattern 
$\big( u_*, \Phif_i(u_*), \Phis_i(u_*) \big) \in \mathcal{B}(\delta_*)^3$ can only
undergo a finite number of splitting-merging cycles, that is,  
for every perturbation with total variation at most $\kappa_* |\sigma(u_*,\Phif_i(u_*))|$
there exists a global--in--time solution $u=u(t,x)$ to \eqref{11}  
satisfying the initial condition \eqref{3.2} and containing only a finite
number of interactions $\Npm \CUp \longrightarrow \Cpm$. 
\end{theorem}

%******************************

Before the nucleation condition can be related to the signed variation of the strong waves during one splitting-merging 
cycle, we need to identify all the waves involved in the interactions. Figure \ref{fig:nucleation} 
illustrates the waves in a splitting-merging cycle. Suppose that a classical shock $\Cpm$
interacts with a weak wave at time $t_0$ and splits into a nonclassical shock $\Npm(t_0+)$ and a faster
classical shock $\CUp(t_0+)$.  The nucleation condition \eqref{nucl:nucleation_condition} is satisfied at time $t_0$. We assume that
the nonclassical shock $\Npm$ persists until a time $t_f$ at which the classical shock $\CUp$ 
crosses the nonclassical shock and results in a single outgoing classical shock $\Cpm(t_f+)$. Suppose that
the left-hand state of $\Npm(t)$ at some time $t$ is denoted by $u_{\ell}(t)$ while the right-hand state of
$\CUp(t)$ is written as $u_r(t)$. Note that these are also the left and right-hand states of the classical 
shock $\Cpm$ when it is present. Immediately after the splitting, the intermediate state $u_m(t_0+)$
is well-defined, just as $u_m(t_f-)$ is well-defined immediately prior to the merging.

%**************************************
\begin{lemma} \label{lem:condition-sm}
When nucleation is present, a single splitting-merging cycle from $t=t_0$ to $t_f$ satisfies
\begin{align}    \label{nucl:condition}
    \eta < & \, \Big( \mu_i\big( \Phi^{\sharp}_i\big( u_{\ell}(t_0+) \big) \big) 
                              - \mu_i\big( \Phi^{\sharp}_i\big( u_{\ell}(t_f-) \big) \big) \Big) 
               + \Big( \mu_i\big(  u_{r}(t_f-) \big) - \mu_i\big( u_{r}(t_0+) \big)\Big).
\end{align}
\end{lemma}

\begin{proof}
The nucleation condition states that 
\be        \label{nucleation_condition}
       0 < \eta < \mu_i\big( \Phi^{\sharp}_i\big(u_{\ell}(t_0+)\big) \big) - \mu_i\big(u_{r}(t_0+)\big), 
\ee
while the subsequent merging requires that the speed of $\Npm(t_f-)$ be greater than 
that of $\CUp(t_f-)$, or more specifically
\be
          a \big( u_r(t_f-), u_m(t_f-) \big) <  a \big( u_{\ell}(t_f-), \Phi^{\flat}_i( u_{\ell}(t_f-) ) \big) 
                = a \big( \Phi^{\sharp}_i( u_{\ell}(t_f-) ), \Phi^{\flat}_i( u_{\ell}(t_f-) ) \big).
\ee
Since $u_m(t_f-)= \Phi^{\flat}_i\big( u_{\ell}(t_f-) \big)$, this inequality on the speeds
is equivalent to
\be         \label{merging_condition}
   \mu_i \Big( \Phi^{\sharp}_i\big( u_{\ell}(t_f-) \big) \Big) < \mu_i\big(  u_r(t_f-) \big).
\ee
Taking \eqref{nucleation_condition} and \eqref{merging_condition} together, we obtain that each
 splitting-merging cycle with nucleation must satisfy \eqref{nucl:condition}.
\end{proof}   
%********* 

Analysis of the two terms in \eqref{nucl:condition} will require us to relate variations 
in $\Npm$ and $\CUp$ to the variation in $u_{\ell}(t)$ and  $u_r(t)$. The left and right hand
states can be modified by waves entering the region between the strong shocks, {\it but also
by waves exiting it.} On the other hand, we will show that the strength of the waves 
exiting the domain is controlled by the strength of the waves entering the domain. It is
this last part that will make the analysis a little bit messier than one would have hoped.

\psfrag{t0}{$t=t_0$}
\psfrag{tf}{$t=t_f$}
\psfrag{t}{$t$}
\psfrag{x}{$x$}
\psfrag{Npm}{$\Npm$}
\psfrag{Cpm}{$\Cpm$}
\psfrag{CUp}{$\CUp$}

\psfrag{alphaL}{$\alpha^L$}
\psfrag{alphaR}{$\alpha^R$}
\psfrag{betaL}{$\beta^L$}
\psfrag{betaR}{$\beta^R$}
%\psfrag{alphaL}{$\alpha^L$}
\psfrag{alpha2R}{$\widetilde{\alpha}^R$}
\psfrag{beta2L}{$\widetilde{\beta}^L$}
\psfrag{beta2R}{$\widetilde{\beta}^R$}
\psfrag{jgei}{$j > i$}
\psfrag{jlei}{$j < i$}

\begin{figure}
\centering
\includegraphics[width=0.65\linewidth]{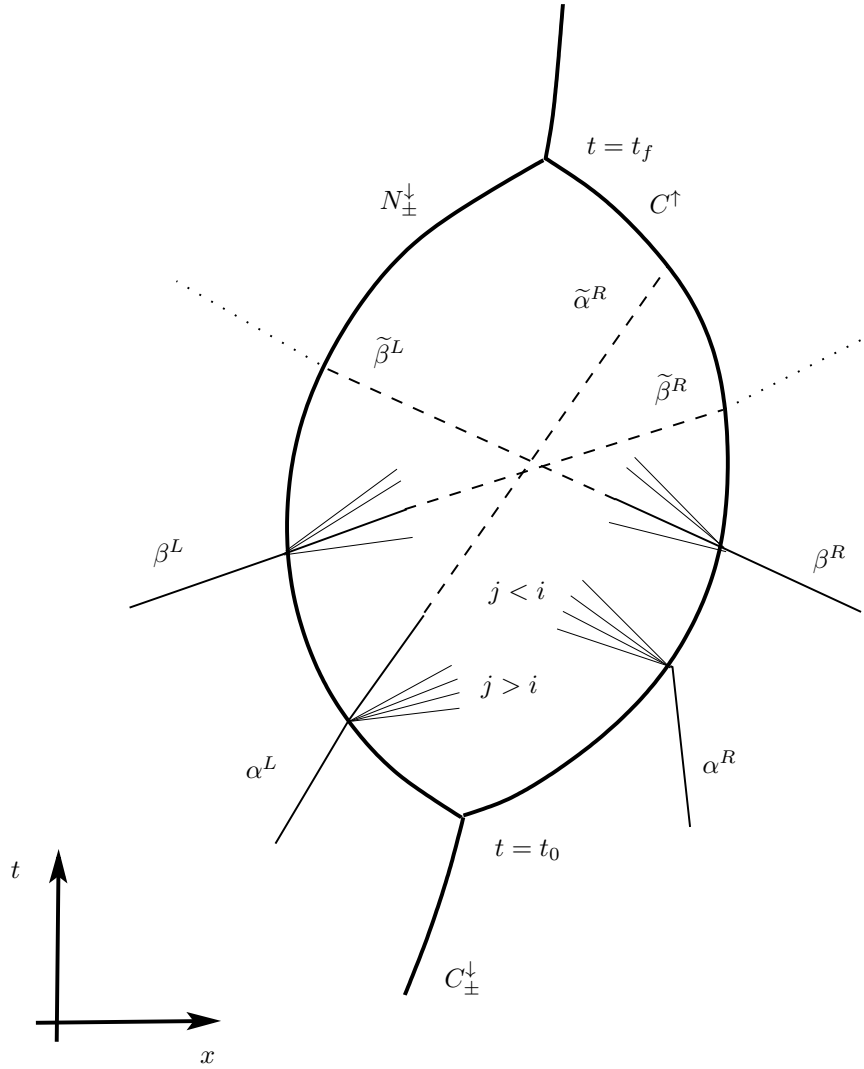}
 \caption{Illustration of the trajectories of the waves involved 
 in a single splitting-merging cycle. The $\alpha$ 
 waves belong to the family $i$, the non-convex family, and the $\beta$ waves
 are those belonging to families $j \neq i$.}
  \label{fig:nucleation}
  \end{figure}

We now describe all the interactions with $\Npm$ and $\CUp$
that could modify the left and right-hand states. Let $\alpha^L_n$,
$n=1,\ldots,n^L$ be the $i$-waves that reach $\Npm$ from the
left and let $\alpha^R_n$, $n=1,\ldots, n^R$ be the $i$-waves
that interact  with $\CUp$ from the right. The weak waves
$\{  \alpha^L_n\}_n$ have interactions of Case 4. with $\Npm$ while
the waves $\{  \alpha^R_n\}_n$ have interactions with $\CUp$ of Case 3.
Similarly, let $\widetilde{\alpha}^R_n, n=1, \ldots, \widetilde{n}^R$ be the 
$i$-waves crossing $\CUp$ from the left. Since $\Npm$ is an undercompressive,
no $i$-waves can interact with it from the right.

Let $\beta^L_{j,n}$, $n=1, \ldots, n^L_j$ be the weak $j$-waves with $j > i$ that interact
with $\Npm$ from the left and define $\beta^R_{j,n}$, $n=1,\ldots,n^R_j$ be
the $j$-waves with $j< i$ interacting with $\CUp$ from the right. The waves $\beta^L_{j,n}$
have interactions described by Case 6. while the waves $\beta^R_{j,n}$ have
interactions studied in Case 7. Let $\widetilde{\beta}^L_{j,n}, n=1, \ldots, \widetilde{n}^L_j$
be the $j$-waves with $j<i$ interacting with $\Npm$ from the right and let 
$\widetilde{\beta}^R_{j,n}, n=1, \ldots, \widetilde{n}^R_j$ be the $j$-waves with
$j>i$ reaching $\CUp$ from the left.

Observe that the interactions
of $\Npm$ and $\CUp$ with the waves 
$\alpha^L, \alpha^R, \widetilde{\alpha}^R, \beta^L, \widetilde{\beta}^L, \beta^R, \widetilde{\beta}^R$ generate new weak 
$i$-waves but all waves $\alpha^*, \beta^*$ generate $j$-waves between $\Npm$ and $\CUp$
These secondary weak waves will need to be accounted for. Finally, recall our earlier
notation $y^h(t)$ for the position of $\Npm$, $z^h(t)$ for the position of $\CUp$ and
$\SMdomain \subset \mathbb{R}^+ \times \mathbb{R}$ for the region between these
shocks.

The objective is to demonstrate that 
\begin{equation}
             \big[ \myW + K Q \big]_{t_0+}^{t_f-} \leq - \eta < 0,
\end{equation}
since it will imply that if  $\myW+KQ$ is a priori bounded in time, and therefore that the
splitting-merging cycles can only occur a finite number of times. As Lemma \ref{lem:condition-sm}
indicates, the quantity $-\eta$ can be bounded below  by the total change
in $u_{\ell}$ and $u_r$, which is itself, up to linear terms, the sum of the
strengths of the waves crossing $\Npm$ and $\CUp$ during one splitting-merging cycle; 
see Lemma \ref{lem:var_states_sm}.
The key observation is that waves crossing $\Npm$ from the left will belong to families
$j \geq i$ and waves crossing $\CUp$ will belong to families $j < i$, hence
the signed variation of the $i$-waves crossing $\Npm$ will be the same as the signed variation of the
$i$-waves interacting with $\CUp$ from the left. The trick will be to correctly 
account for changes in $Q$ that allow us to reduce the interactions to linear superpositions.

The change in $\myW + KQ$ is the sum of the changes during all interactions
occuring between time $t_0+$ and $t_f-$. Our earlier work has already shown that
$[\myW+KQ]$ is always negative during these interactions but we will examine only interactions involving at least
one wave among $\Npm, \CUp$ and the waves in $\SMdomain$. The other
interactions will lead to arbitrary decreases in $\myW+KQ$ that cannot be
controlled.

%***************** 
\subsubsection*{Weak interactions with $\Npm$}

We begin by examining only interactions between $i$-waves $\alpha^L_n$
and $\Npm$. These cross $\Npm$ and a fraction of the wave continues to travel
in $\SMdomain$. Secondary waves of families $j < i$ are created and propagate left of 
$\Npm$ while secondary waves of families $j > i$ propagate into $\SMdomain$.
Lemmas \ref{case:4} and \ref{Qcase:4} quantify these interactions, which are then
summarized in \eqref{Gcase:4},
\begin{equation}     \label{alphaL_V}
             [\myW+ KQ] \leq \big( -\kappaL + (1+\Cff) \kappaM \big) \,  |\alpha^L_n|.
\end{equation}
Similar calculations show that
\begin{equation}    \label{alphaL_mu}
           [\mu_i(u_{\ell})]  = -\Sgn(\alpha^L_n) |\alpha^L_n| + \mathcal{O}\big( |\Npm| \cdot |\alpha^L_n| \big), \qquad 
           [\mu_j(u_{\ell})]  =  \mathcal{O}\big( |\Npm| \cdot |\alpha^L_n| \big),  \quad j \neq i,    
\end{equation}
where $\Sgn (\alpha) = -1$ if $\alpha$ is a shock and $+1$ if $\alpha$ is a rarefaction.

Treating now interactions between $j$-waves, $j> i$, and the strong wave $\Npm$, as described in 
Lemmas \ref{case:6} and \ref{Qcase:6}, we obtain the estimate \eqref{Gcase:6a}
\begin{equation}   \label{betaL_V}
    [\myW+ KQ] \leq \big( -\kappaLgrt + \kappaMgrt \big) |\beta^L_{j,n} |.
\end{equation}
This weak wave also modifies the left-hand state by the amount
\begin{equation}   \label{betaL_mu}
        [\mu_j(u_{\ell})]  =  -\Sgn( \beta^L_{j,n} ) |\beta^L_{j,n} | + \mathcal{O}\big( |\Npm| \cdot |\beta^L_{j,n}| \big), \qquad
        [\mu_k(u_{\ell})]  =  \mathcal{O}\big( |\Npm| \cdot |\beta^L_{j,n}| \big), \quad k \neq j.
\end{equation}

Finally, there are $j$-waves $\widetilde{\beta}^L_{j,n}$ which cross $\Npm$ from the right, after having entered 
$\SMdomain$ and these waves decrease the total variation by
\begin{equation}   \label{betaL2_V}
    [\myW+ KQ] \leq \big( -\kappaMless + \kappaLless \big) |\widetilde{\beta}^L_{j,n} |.
\end{equation}
and modify the left-hand state $\mu_k(u_{\ell})$ by the amount
\begin{equation}   \label{betaL2_mu}
        [\mu_j(u_{\ell})]  =  -\Sgn( \widetilde{\beta}^L_{j,n} ) |\widetilde{\beta}^L_{j,n} | 
                                      + \mathcal{O}\big( |\Npm| \cdot |\widetilde{\beta}^L_{j,n}| \big), \qquad
        [\mu_k(u_{\ell})]  =  \mathcal{O}\big( |\Npm| \cdot |\widetilde{\beta}^L_{j,n}| \big), \quad k \neq j.
\end{equation}

%***************** 
\subsubsection*{Weak interactions with $\CUp$}

The interactions between $j$-waves, and $\CUp$ are simpler to treat than those with $\Npm$
since $i$-waves cannot cross $\CUp$. On the other hand, it is important to understand that the $i$-waves
approaching $\CUp$ from the left must be the remains of the $i$-waves
that previously crossed $\Npm$.

For $i$-waves reaching $\CUp$ from the right, we find 
\begin{equation}   \label{alphaR_V}
          [\myW + K Q ] \leq -\kappaR |\alpha^R_{n}|,
\end{equation}
and that the right-hand states are modified by
\begin{equation}    \label{alphaR_mu}
           [\mu_i(u_{r})]  = -\Sgn(\alpha^R_n) |\alpha^R_n| + \mathcal{O}\big( |\CUp| \cdot |\alpha^R_n| \big), \qquad 
           [\mu_j(u_{r})]  =  \mathcal{O}\big( |\CUp| \cdot |\alpha^R_n| \big), \quad  j \neq i.    
\end{equation} 
When $i$-waves reach $\CUp$ from the left, similar estimates hold, namely
\begin{equation}   \label{alphaR2_V}
     [\myW + K Q ] \leq -\kappaM |\widetilde{\alpha}^R_{n}|,
\end{equation}
and
\begin{equation}   \label{alphaR2_mu}
           [\mu_k(u_{r})]  =  \mathcal{O}\big( |\CUp| \cdot |\widetilde{\alpha}^R_n| \big), \quad \text{ for all } k.  
\end{equation}
The interactions involving $j$-waves and $\CUp$ are simple. If $j < i$, the estimate \eqref{Gcase:7b} provides
\begin{equation}   \label{betaR_V}
          [\myW + KQ ] \leq ( - \kappaRless + \kappaMless ) \, |\beta^R_{j,n}|.
\end{equation}
and changes in the state on the right-hand side are easily seen to be
\begin{equation}    \label{betaR_mu}
        [\mu_j(u_{r})]  =  -\Sgn( \beta^R_{j,n} ) |\beta^R_{j,n} | + \mathcal{O}\big( |\CUp| \cdot |\beta^R_{j,n}| \big), \qquad
        [\mu_k(u_{r})]  =  \mathcal{O}\big( |\CUp| \cdot |\beta^R_{j,n}| \big), \quad k \neq j.
\end{equation}
On the other hand, for $j > i$,  the estimate \eqref{Gcase:7a} is
\begin{equation}    \label{betaR2_V}
        [\myW + KQ ] \leq ( - \kappaMgrt + \kappaRgrt ) \, |\widetilde{\beta}^R_{j,n}|,
\end{equation}
and the right-hand states change by
\begin{equation}   \label{betaR2_mu}
        [\mu_j(u_{r})]  =  -\Sgn( \widetilde{\beta}^R_{j,n} ) |\widetilde{\beta}^R_{j,n} | 
                                      + \mathcal{O}\big( |\CUp| \cdot |\widetilde{\beta}^R_{j,n}| \big), \qquad
        [\mu_k(u_{r})]  =  \mathcal{O}\big( |\CUp| \cdot |\widetilde{\beta}^R_{j,n}| \big), \quad k \neq j.
\end{equation}

%***************** 
\subsubsection*{Wave crossings between $\Npm$ and $\CUp$}

The waves entering the region between the two strong shocks are either
\begin{itemize}
\item  $i$-waves $\alpha^L_n$ entering from the left;
\item $j$-waves $\beta^L_{j,n}$, $j > i$, entering from the left;  
\item $j$-waves $\beta^R_{j,n}$, $j < i$, entering from the right.
\end{itemize}
Our goal is to relate the strength of these waves to the strength of the
outgoing waves $\widetilde{\alpha}^R_n, \widetilde{\beta}^R_{j,n}$ and $\widetilde{\beta}^L_{j,n}$.

If all interactions within $\SMdomain$ were linear, then the $i$-waves
exiting on the right, namely $\widetilde{\alpha}^R_n$ would be the result of the interactions
among the fraction of the $i$-waves $\Cff \alpha^L_n$ which had crossed $\Npm$ from the left.
Similarly, since the only $j$-waves with $j>i$ ($j<i$) that entered the domain were $\beta^L_{j,n}$ ($\beta^R_{j,n}$),
then $\widetilde{\beta}^R_{j,n}$ ($\widetilde{\beta}^L_{j,n}$) would be the
result of the interactions among $\beta^L_{j,n}$ ($\beta^R_{j,n}$). In the linear
case, this would justify the following three bounds
\begin{align}
    \sum_{n=1}^{\widetilde{n}^R} |\widetilde{\alpha}^R_n| \leq & \, C_i \sum_{n=1}^{n^L} |\alpha^L_n|,   \label{crossing_est1} \\
    \sum_{n=1}^{\widetilde{n}^L_j} |\widetilde{\beta}^L_{j,n}| \leq & \, \sum_{n=1}^{n^R_j} |\beta^R_{j,n}|,\quad \text{ for all } j < i, \label{crossing_est2} \\
     \sum_{n=1}^{\widetilde{n}^R_j} |\widetilde{\beta}^R_{j,n}| \leq & \,  \sum_{n=1}^{n^L_j} |\beta^L_{j,n}|,\quad \text{ for all } j > i. \label{crossing_est3}
\end{align}

In the weakly nonlinear regime, these estimates are correct up to
second-order corrections weighted by the total strength of all 
of the secondary waves which might be generated by interactions inside $\SMdomain$.
The correct upper bounds would therefore be the same, except
for the presence of weights of the form $(1+\epsilon(t))$, which we 
can simply neglect for the purposes of this section.

%***************************************
\subsection{Total variation estimate for the splitting-merging cycle}

We now combine all the bounds computed to produce a negative upper bound on changes in $\myW + KQ$
during a single splitting-merging cycle. We begin by returning to the quantity $\eta$ from Lemma \ref{lem:condition-sm} which
was found to be strictly negative during a cycle.

%**************************************
\begin{lemma} \label{lem:var_states_sm}
Suppose $C_j^\sharp$ is the Lipschitz constant for the map $\mu_j\circ \Phi_i^\sharp$. Then
the nucleation condition is bounded by the signed variation of the incoming waves according to
\begin{align}
         \eta \leq & \,     \Big( C_i^{\sharp} +  \mathcal{O}\big(  |\Npm| \big)  \Big)  \sum_{n=1}^{n^L}  |\alpha^L_n |  
                    +  \Big(1 +  \mathcal{O}\big( |\CUp| \big) \Big)   \sum_{n=1}^{n^R}   |\alpha^R_n |     \label{nucl:eta_bound}  \\
                        & \,   + \sum_{j>i} \Big( C_j^{\sharp} + \mathcal{O}\big(  |\Npm| \big)  \Big) \sum_{n=1}^{n^L_j} | \beta^L_{j,n} | 
                                                + \Big(  1 + \mathcal{O}\big( |\CUp| \big) \Big) \sum_{j<i}  \sum_{n=1}^{n^R_j}| \beta^R_{j,n} |.  \nonumber 
\end{align}
\end{lemma}

\begin{proof}
We observe immediately that the Lipschitz continuity  of the map
$\Phi^{\sharp}_i$ implies that there exists constants $C^{\sharp}_k, \, k=1,\ldots, n,$ such that
$$      
     \Big| \mu_i \big( \Phi^{\sharp}_i (u_{\ell}(t_0+) \big) - \mu_i \big( \Phi^{\sharp}_i( u_{\ell}(t_f-) \big) \Big|
           <  \sum_{k=1}^n C^{\sharp}_k \Big| \mu_k\big( u_{\ell}(t_0+) \big) - \mu_k\big( u_{\ell}(t_f-) \big) \Big|\,.
$$
The first term of \eqref{nucl:condition} is bounded by the sum \eqref{alphaL_mu}, \eqref{betaL_mu} and \eqref{betaL2_mu},
while the second term is bounded by the sum of \eqref{alphaR_mu}, \eqref{alphaR2_mu}, \eqref{betaR_mu} 
and \eqref{betaR2_mu}.
\begin{align*}
        \eta < & \sum_{k=1}^n C^{\sharp}_k \Big| \mu_k\big( u_{\ell}(t_0+) \big) - \mu_k\big( u_{\ell}(t_f-) \big) \Big|
                         +  \Big( \mu_i\big(  u_{r}(t_f-) \big) - \mu_i\big( u_{r}(t_0+) \big)\Big) \\
              \leq &  \,  C_i^{\sharp} \Big| \sum_{n=1}^{n^L} - \Sgn (\alpha^L_n) |\alpha^L_n | \Big| 
                                    + \sum_{n=1}^{n^L} \mathcal{O}\big(  | \Npm(t_n^L-)| \cdot |\alpha^L_n| \big) \\
                   & \,  + \sum_{j>i} C_j^{\sharp} \Big| \sum_{n=1}^{n^L_j} -\Sgn(\beta^L_{j,n})| \beta^L_{j,n} | \Big| 
                           + \sum_{j>i}\sum_{n=1}^{n^L_j} \mathcal{O}\big( |\Npm(t^L_{j,n}-)| \cdot |\beta^L_{j,n}| \big) \\
                  & \,    + \sum_{j<i} C_j^{\sharp} \Big| \sum_{n=1}^{\widetilde{n}^L_j} -\Sgn(\widetilde{\beta}^L_{j,n})| \widetilde{\beta}^L_{j,n} | \Big| 
                           + \sum_{j<i}\sum_{n=1}^{\widetilde{n}^L_j} \mathcal{O}\big( |\Npm(\widetilde{t}^L_{j,n}-)| \cdot |\widetilde{\beta}^L_{j,n}| \big) \\
                  & \,             +  \sum_{n=1}^{n^R} - \Sgn (\alpha^R_n) |\alpha^R_n | 
                                      + \sum_{n=1}^{n^R} \mathcal{O}\big(|\CUp(t^R_n-)| \cdot |\alpha^R_n| \big)  \\
                  & \,              +  \sum_{n=1}^{\widetilde{n}^R} - \Sgn (\widetilde{\alpha}^R_n) |\widetilde{\alpha}^R_n | 
                                      + \sum_{n=1}^{\widetilde{n}^R} \mathcal{O}\big(|\CUp(\widetilde{t}^R_n-)| \cdot |\widetilde{\alpha}^R_n| \big)     \\
                  & \,   + \sum_{j<i} \sum_{n=1}^{n^R_j} -\Sgn(\beta^R_{j,n})| \beta^R_{j,n} | 
                           + \sum_{j<i}\sum_{n=1}^{n^R_j} \mathcal{O}\big( |\CUp(t^R_{j,n}-)| \cdot |\beta^R_{j,n}| \big) \\
                  & \,     + \sum_{j>i} \sum_{n=1}^{\widetilde{n}^R_j} -\Sgn(\widetilde{\beta}^R_{j,n})| \widetilde{\beta}^R_{j,n} | 
                           + \sum_{j>i}\sum_{n=1}^{\widetilde{n}^R_j} \mathcal{O}\big( |\CUp(\widetilde{t}^R_{j,n}-)| \cdot |\widetilde{\beta}^R_{j,n}| \big)                               
\end{align*}
Up to second-order terms, we may assume that the strength of the strong waves is equal
to the constant strength of the unperturbed waves in $\overline{u}_0$. Our earlier bounds show that
the strength of the waves exiting $\SMdomain$ are bounded by the strength of the waves entering it, and therefore,
after simplifying the expressions above, we can rewrite the upper bound above as in \eqref{nucl:eta_bound}.
\end{proof}

We complete this section with a proof of the main theorem. The proof resembles
the accounting that was necessary in the proof of the previous lemma, except for the
manipulation of the weights $k_*$ now appearing in $\myW + KQ$.

\begin{proof}[Proof of Theorem \ref{nucleationresult}]
Variations in $\myW+KQ$ involving only waves in $V_L$ and $V_R$ always lead to decreases
in Glimm's functional, and therefore it suffices to account for changes involving interactions with at least one wave
from among $\Npm$ or $\CUp$. We invoke the bounds \eqref{alphaL_V}, \eqref{betaL_V}, \eqref{betaL2_V},
 \eqref{alphaR_V}, \eqref{alphaR2_V}, \eqref{betaR_V} and \eqref{betaR2_V}.
\begin{align*}
   [\myW + KQ] \leq & \, \sum_{n=1}^{n^L} \big( - \kappaL + (1+\Cff) \kappaM \big) \, |\alpha^L_n| \\
                                  & \, + \sum_{j>i} \sum_{n=1}^{n^L_j} \big( - \kappaLgrt + \kappaMgrt \big) \, |\beta^L_{j,n} |
                                      +  \sum_{j<i} \sum_{n=1}^{\widetilde{n}^L_j} \big( - \kappaMless + \kappaLless \big) \, |\widetilde{\beta}^L_{j,n} | \\
                                  & \, + \sum_{n=1}^{n^R} - \kappaR |\alpha^R_n| + \sum_{n=1}^{\widetilde{n}^R} - \kappaM | \widetilde{\alpha}^R_n| \\
                                  & \,  + \sum_{j<i} \sum_{n=1}^{n^R_j} \big( - \kappaRless + \kappaMless \big) \, |\beta^R_{j,n} |
                                      +  \sum_{j>i} \sum_{n=1}^{\widetilde{n}^R_j} \big( - \kappaMgrt + \kappaRgrt \big) \, |\widetilde{\beta}^R_{j,n} |   
\end{align*}
We observe that up to second-order terms
$$
 (-\kappaL+(1+\Cff)\kappaM) \sum_{n=1}^{n^L} |\alpha^L_n| - \kappaM \sum_{n=1}^{\widetilde{n}^R} |\widetilde{\alpha}^R_n| 
      \leq  (-\kappaL+ \kappaM) \sum_{n=1}^{n^L} |\alpha^L_n|.
$$
Similarly, since the weights are negative, the bounds \eqref{crossing_est1}-\eqref{crossing_est3} imply that
\begin{align*}
  \sum_{j<i}\sum_{n=1}^{\widetilde{n}^L_j} (-\kappaMless + \kappaLless) |\widetilde{\beta}^L_{j,n}| 
              \leq & \,    \sum_{j<i}\sum_{n=1}^{n^L_j} (-\kappaMless + \kappaLless) |\beta^L_{j,n}|, \\
  \sum_{j>i}\sum_{n=1}^{\widetilde{n}^R_j} (-\kappaMgrt + \kappaRgrt) |\widetilde{\beta}^R_{j,n}| 
              \leq & \,    \sum_{j>i}\sum_{n=1}^{n^R_j} (-\kappaMgrt + \kappaRgrt) |\beta^R_{j,n}|.           
\end{align*}
These bounds allow us to further simplify the estimates to
\begin{align}                                      
   [\myW + KQ]    \leq & \, \big( - \kappaL +  \kappaM \big) \sum_{n=1}^{n^L} |\alpha^L_n| -\kappaR \sum_{n=1}^{n^R} |\alpha^R_n| \nonumber \\
                                   & \,+  \big( - \kappaLgrt  + \kappaRgrt \big) \,  \sum_{j>i}\sum_{n=1}^{n^L_j} |\beta^L_{j,n} |  
                                         + \big( - \kappaRless + \kappaLless \big) \, \sum_{j<i}  \sum_{n=1}^{n^R_j} |\beta^R_{j,n} |.  \label{sm:deltaTV}
\end{align}
For the choices of weights $k_*$ proposed in Lemma \ref{lem:weights}, we find that $\kappaR =1$ and
\begin{align*}
    -\kappaL + \kappaM = & \, -(1+\Cff + \zeta) + 1 = -\Cff - \zeta, \\
    -\kappaLgrt + \kappaRgrt = & \, -(1+\zeta) + (1-\zeta) = - 2 \zeta, \\
    -\kappaRless + \kappaLless = & \, -(1+\zeta) + (1-\zeta) = -2 \zeta. 
\end{align*} 
In order to bound above the negative quantity  \eqref{sm:deltaTV} by a multiple of $-\eta$, it suffices to observe
that the weights on \eqref{nucl:eta_bound} are all bounded uniformly, and hence for every fixed positive value of $\zeta$,  
there exists a positive constant $c$ such that
$$
  [\myW + KQ]  \leq -c \eta.
$$
Splitting and mergings cannot proceed indefinitely since each cycle leads $W+KQ$ to 
decrease by a constant amount. The constant $c \eta$ being uniform as 
the front-tracking scheme converges to a weak solution, the result also 
holds for that weak solution.  
\end{proof}

%**************************************** 

\end{document}